\documentclass[12pt,leqno]{article}

\usepackage{fullpage}

\usepackage{amsmath,mathrsfs,amsthm,rotating,amsxtra}
\usepackage{txfonts,graphics,float,cancel}

\usepackage{amsmath, amssymb, amsfonts, amsthm,amscd}

\usepackage[all]{xy}

\usepackage[utf8]{inputenc}
\usepackage [english]{babel}
\usepackage {geometry}
\usepackage {fancyhdr}
\usepackage {amsmath,amsfonts,amssymb,amsthm,epsfig,epstopdf,titling,url,array}
\usepackage {mathrsfs}
\usepackage {graphicx}
\usepackage {hyperref}
\usepackage {enumerate}
\usepackage {wrapfig}
\usepackage {cleveref}
\swapnumbers

\theoremstyle{plain}
\newtheorem{theorem}{Theorem}[section]

\newtheorem{lemma}[theorem]{Lemma}
\newtheorem{proposition}[theorem]{Proposition}
\newtheorem{corollary}[theorem]{Corollary}

\theoremstyle{definition}
\newtheorem{definition}[theorem]{Definition}
\newtheorem{remark}[theorem]{Remark}

\DeclareMathOperator{\Pic}{Pic}
\DeclareMathOperator{\Iim}{Im}
\DeclareMathOperator{\SL}{SL}
\DeclareMathOperator{\codim}{codim}

\newcommand{\OpIm}{\operatorname{Im}}
\newcommand{\cd}{\mathcal{D}}

  \newcommand{\om}{\omega}

  \newcommand{\del}{\delta}   \newcommand{\Del}{\Delta}
  \newcommand{\gam}{\gamma}   \newcommand{\Gam}{\Gamma}

  \def\b1{\text{\large 1}}  
  
  \def\from{\leftarrow}

  \def\ip<#1>{\langle#1\rangle}   

  \newcommand{\depth}{\operatorname{depth}}
  
  \newcommand{\Hom}{\mathscr{H}om}
  \newcommand{\Image}{\operatorname{Image}}
  
  \newcommand{\Invlt}{\operatorname{Inv.lt.}}
 \newcommand{\Dirlt}{\operatorname{Dir.lt.}}
  \newcommand{\Ker}{\operatorname{Ker}}
  \newcommand{\Lie}{\operatorname{Lie}}
  \newcommand{\Limit}{\operatorname{Limit}}
  \newcommand{\lt}{\operatorname{Limit}}
  \newcommand{\optop}{\operatorname{top}}
  \newcommand{\Tor}{\operatorname{Tor}}

 \newcommand{\gr}{\operatorname{gr}}

\newcommand{\beqn}{\begin{equation}}
\newcommand{\eeqn}{\end{equation}}

 \newcommand{\fp}{\mathfrak{p}}
 
\newcommand{\ft}{\mathfrak{t}}
 \newcommand{\fu}{\mathfrak{u}}
 \newcommand{\fv}{\mathfrak{v}}
 \newcommand{\fw}{\mathfrak{w}}

\newcommand{\bc}{\mathbb{C}}
\newcommand{\bp}{\mathbb{P}}

\newcommand{\bz}{\mathbb{Z}}

 \newcommand{\cc}{\mathcal{C}}
 \newcommand{\cf}{\mathcal{F}}
 \newcommand{\cg}{\mathcal{G}}
 
 \newcommand{\cl}{\mathcal{L}}
 \newcommand{\cm}{\mathcal{M}}
 \newcommand{\co}{\mathcal{O}}
 \newcommand{\cp}{\mathcal{P}}
 \newcommand{\cs}{\mathcal{S}}

\newcommand{\OO}{\mathscr{O}}
\newcommand{\ZZ}{\mathcal{Z}}

\newcommand{\seteq}{\mathbin{:=}}

\newcommand{\ext}{\mathscr{E}xt}
\newcommand{\tor}{\mathscr{T}or}

\newcommand{\home}{\mathscr{H}om}
\renewcommand{\O}{\mathscr{O}}
\renewcommand{\co}{\mathscr{O}}
\newcommand{\C}{{\mathbb C}}

\newcommand{\pt}{\operatorname{pt}}
\newcommand{\eqn}{\begin{eqnarray*}}
\newcommand{\eneqn}{\end{eqnarray*}}
\newcommand{\bl}{\bigl}

\newcommand{\eq}{\begin{eqnarray}}
\newcommand{\ssum}{\mathop{\mbox{\small$\sum$}}}
\newcommand{\Z}{{\mathbb{Z}\mspace{1mu}}}
\newcommand{\Rhom}{\mathrm{R}\kern-.2em\hom}
\newcommand{\scbul}{{\,\raise.4ex\hbox{$\scriptscriptstyle\bullet$}\,}}
\newcommand{\bwr}{{\mbox{\large$\wr$}}}
\newcommand{\db}[1]{\raisebox{-.5ex}[2ex][1.8ex]{$#1$}}

\begin{document}

\title{Positivity in $T$-Equivariant $K$-theory of flag varieties associated to
Kac-Moody groups}

\author{Shrawan Kumar\\
(with an appendix by M. Kashiwara)}

\maketitle

\section{Introduction}\label{intro}

Let $G$ be any symmetrizable Kac-Moody group over $\bc$ completed
along
 the negative roots and $G^{\min}\subset G$ be the `minimal' Kac-Moody group.
 Let $B$ be the standard (positive) Borel
subgroup, $B^{-}$ the standard negative Borel subgroup, $H=B\cap B^{-}$ the
standard maximal torus and $W$ the Weyl group.  Let
  $\bar{X} = G/B$
be the `thick' flag variety (introduced by Kashiwara) which contains the standard
KM flag
ind-variety
  $ X = G^{\min}/B.$ Let $T$ be the quotient torus
$H/Z(G^{\min})$, where $Z(G^{\min})$ is the center of $G^{\min}$. Then, the action
of $H$ on $\bar{X}$ (and $X$) descends to an action of $T$.
We denote the
representation ring of $T$ by $R(T)$.
  For any $w\in W$, we
have the Schubert cell
$
C_{w}:={BwB/B}\subset X ,
$
the Schubert variety
$
X_{w}:=\overline{C_w}\subset X ,
$
the opposite Schubert cell
$
C^{w}:={B^{-}wB/B}\subset \bar{X},
$
and the opposite Schubert variety
$
X^{w}:=\overline{C^w}\subset \bar{X}.
$ When $G$ is a (finite dimensional) semisimple group, it is referred to as the {\it finite case}.

 Let  $K^{\optop}_T(X)$ be the $T$-equivariant topological $K$-group of the
 ind-variety $X$.
 Let $\{\psi^w\}_{w\in W}$ be the `basis'  of $K^{\optop}_T(X)$ given by
 Kostant-Kumar (cf. Definition \ref{psibasis}).

Express the product in topological $K$-theory $K^{\optop}_T(X)$:
  \beqn \label{neqq1}
\psi^u\cdot\psi^v = \sum_w p^w_{u,v} \psi^w, \quad\text{for } p^w_{u,v}\in R(T).
  \eeqn
  Then, the following result is our main theorem (cf. Theorem \ref{verymain}).
  This was conjectured by
Graham-Kumar [GK, Conjecture 3.1] in the finite case and proved in this case by Anderson-Griffeth-Miller [AGM, Corollary 5.2].

\begin{theorem} \label{intromain} For any $u,v,w\in W$,
\[
 (-1)^{\ell (u)+\ell (v)+ \ell (w)} \,p^w_{u,v}\in \bz_+[(e^{-\alpha_1}-1), \dots,
 (e^{-\alpha_r}-1)],
 \]
 where $\{\alpha_1, \dots, \alpha_r\}$ are the simple roots, i.e., $ (-1)^{\ell (u)+\ell (v)+ \ell (w)} \,p^w_{u,v}$ is a polynomial 
in the variables $x_1=e^{-\alpha_1}-1, \dots,
 x_r=e^{-\alpha_r}-1$ with non-negative integral coefficients.
\end{theorem}
By a result of Kostant-Kumar [KK, Proposition 3.25], 
\beqn \label{eqn3.12}
K^{\optop}(X) \simeq \bz\otimes_{R(T)}\,  K^{\optop}_T(X),
\eeqn  
where $\bz$ is considered as an $R(T)$-module via the
evaluation at $1$ and $K^{\optop}(X)$ is the topological (non-equivariant) $K$-group of $X$. 
Thus, as an immediate consequence of the above theorem (by evaluating at $1$), we obtain
the following result (cf. Corollary \ref{maincor}). The following corollary  was conjectured by 
A.S. Buch in the finite case and proved in this case by  Brion [B].

\begin{corollary}  For any $u,v,w\in W$,
$$(-1)^{\ell (u)+\ell (v)+ \ell (w)}\,a^w_{u,v} \in \bz_+,$$
where $a^w_{u,v}$ are the structure constants of the product in $K^{\optop}(X)$
with respect to  the basis $\psi^w_o:=1\otimes \psi^w $.
\end{corollary}

Further,  Theorem \ref{intromain} also gives the positivity for  the multiplicative structure constants 
in the Schubert basis for the $T$-equivariant cohomology  $H^*_T(X, \bc)$  with complex coefficients  as described  below.

The representation ring $R(T)$ has a decreasing
filtration $\{R(T)_n\}_{n \geq 0}$, where  
$$R(T)_n:=\{f\in R(T): \text{mult}_1(f)\geq n\},$$
where $\text{mult}_1(f)$ denotes the multiplicity of the zero of $f$ at $1$. 

We first recall the following result from [KK, $\S\S$2.28 -- 2.30 and Theorem 3.13]. 
\begin{theorem} \label{gr}There exists a decreasing filtration $\{\mathcal{F}_n\}_{n \geq 0}$ of the ring $K^{\optop}_T(X)$ compatible with the filtration of $R(T)$  such that
there is a ring  isomorphism  of the associated graded ring
$$\beta:\bc\otimes_\bz\gr \left(K^{\optop}_T(X)\right) \simeq H^*_T(X, \bc).$$
  Moreover,
for any $w\in W$,   $\psi^w\in \mathcal{F}_{\ell(w)}$ and under this isomorphism,
$$\beta(\overline{\psi^w})={\hat{\varepsilon}}^w,$$
where $\overline{\psi^w}$ denotes the element $\psi^w$ (mod $ \mathcal{F}_{\ell(w)+1}$) in $\gr_{\ell(w)} \left(K^{\optop}_T(X)\right)$
and $ {\hat{\varepsilon}}^w$ is the (equivariant) Schubert basis of $H^*_T(X, \bc)$ as in [K, Theorem 11.3.9].
\end{theorem}

Express the  product in  $H^*_T(X)$:
  \[
{\hat{\varepsilon}}^u \cdot {\hat{\varepsilon}}^v = \sum_w h^w_{u,v} {\hat{\varepsilon}}^w, \quad\text{for } h^w_{u,v}\in S(\ft^*),
  \]
where $\ft$ is the Lie algebra of $T$ and $h^w_{u,v} $ is a homogeneous polynomial of degree $\ell (u)+\ell (v)- \ell (w)$.
 Combining Theorems \ref{intromain} and \ref{gr}, we obtain the following result proved by Graham [Gr].
 
\begin{theorem} \label{graham} For any $u,v,w\in W$,
\[
 h^w_{u,v}\in \bz_+[\alpha_1, \dots,
 \alpha_r],
 \]
i.e., $h^w_{u,v}$ is a homogeneous polynomial in $\{\alpha_1, \dots, \alpha_r\}$ of degree $\ell(u)+ \ell(v) -\ell(w)$ with non-negative integral coefficients.
\end{theorem} 

We can further specialize the above theorem to obtain the positivity for  the multiplicative structure constants $b^w_{u,v}$ 
in the standard Schubert basis $\{\varepsilon^w\}_{w\in W}$ obtained from specializing ${\hat{\varepsilon}}^w$ at $0$ for the singular (non-equivariant) cohomology $H^*(X, \bc)$ 
 because of the following result:

\beqn \label{eqn3.12'}
H^*(X, \bc) \simeq \bc\otimes_{S(\ft^*)}\, H^*_T(X, \bc),
\eeqn  
where $\bc$ is considered as an $S(\ft^*)$-module via the
evaluation at $0$ (cf. [K, Proposition 11.3.7]). We get the following corollary due to Kumar-Nori [KuN] from Theorem \ref{graham}
by evaluating at $0$.
\begin{corollary} \label{kun} For any $u,v,w\in W$,
\[
 b^w_{u,v}\in \bz_+ .
 \]
\end{corollary} 

The proof of Theorem \ref{intromain} relies heavily on algebro-geometric techniques. We realize the structure
constants $ p^w_{u,v}$ from equation \eqref{neqq1} as the coproduct structure constants in the structure sheaf
basis
$\{\co_{X_w}\}_{w\in W}$  of the $T$-equivariant
$K$-group $K^T_o(X)$ of finitely supported $T$-equivariant coherent sheaves on $X$
(cf Proposition \ref{n3.1}). Let $K^0_T(\bar{X})$ denote the Grothendieck group of
$T$-equivariant coherent $\co_{\bar{X}}$-modules $\cs$.
Then, there is a `natural' pairing (cf. Section \ref{sec5})
$$
\ip< \, ,\, > : K^0_T(\bar{X}) \otimes K^T_0(X) \to R(T),$$ coming from
the $T$-equivariant Euler-Poincar\'{e} characteristic.  For any character $e^{\lambda}$ of $H$, let
 $\mathcal{L}(\lambda)$  be the
$G$-equivariant line bundle on $\bar{X}$
associated to the character $e^{-\lambda}$ of $H$ (cf. $\S$2). 
Define
the $T$-equivariant coherent sheaf
$\xi^u := e^{-\rho} \cl (\rho )\om_{X^u}$ on $\bar{X}$, where
\[\om_{X^u} := \ext^{\ell (u)}_{\co_{\bar{X}}}
\bigl(\co_{X^u}, \co_{\bar{X}} \bigr)\otimes\cl (-2\rho )
  \]
  is the dualizing sheaf of $X^u$. We show that
the basis $\{[\xi^w]\}$  is dual to the basis
$\{[\co_{X_w}]\}_{w\in W}$ under the above pairing (cf. Proposition \ref{prop2.6}).

Following [AGM], we define the `mixing group' $\Gamma$ in Definition \ref{mixdef} and prove its
 connectedness  (cf. Lemma \ref{connected}).
Then, we prove  our main technical result
Theorem \ref{thma14} on vansihing of some Tor sheaves as well as some cohomology vanishing. The proof of its two parts are given
in Sections \ref{sec7} and \ref{sec3} respectively.

From the connectedness of $\Gamma$ and Theorem \ref{thma14}, we get Corollary \ref{cor3.9}. This corollary allows us 
to easily obtain our main theorem (Theorem  \ref{intromain}).

Rest of the paper is devoted to prove Theorem \ref{thma14}.

In Section \ref{sec7},
we prove various local Ext and Tor vanishing results crucially using the `Acyclicity Lemma'
of Peskine-Szpiro (cf. Corollary \ref{4.3}). The following is one of the
main results of this section (cf. Propositions \ref{propa17} and \ref{propa18}).

  \begin{proposition}  For any $u,w\in W$,
    \[
\ext^j_{\co_{\bar{X}}} (\co_{X^u}, \co_{X_w}) =0, \quad\text{ for all } j\neq \ell (u).
    \]
    Thus,
    \[\tor_j^{\co_{\bar{X}}} (\xi^u, \co_{X_w})
    =0,\,\,\,\text{ for \,all\,}  j> 0.\]
  \end{proposition}
This proposition allows us to prove the (a) part of Theorem \ref{thma14}.

We also prove the following local Tor vanishing result (cf. Lemma \ref{4.5} and Corollary \ref{newcor5.5}),
which is a certain cohomological
analogue of the proper intersection property of $X^u$ with $X_w$.

\begin{lemma}  For any $u,w\in W$,
  \[
\tor_j^{\co_{\bar{X}}} (\co_{X^u}, \co_{X_w}) = \tor_j^{\co_{\bar{X}}} (\co_{\partial X^u}, \co_{X_w}) =0, \qquad\text{for all $j>0$.}
  \]
  \end{lemma}
In Section \ref{sec1} we show that the Richardson varieties $X^v_w := X_w\cap X^v
\subset \bar{X}$ are irreducible,
  normal and Cohen-Macaulay, for short CM (cf. Proposition \ref{n5.6}). Then, we construct a
  desingularization  $Z^v_w$   of $X^v_w$ (cf. Theorem \ref{thm3}). In this section, we prove
  that various
  maps appearing in the big diagram in Section \ref{sec2} are smooth or flat morphisms. Though not used in the paper, we determine the dualizing sheaf of 
the Richardson varieties $X^v_w$ (cf. Lemma \ref{n7.3}).

  In Section \ref{sec2}, we introduce the crucial irreducible scheme ${\mathcal{Z}}$
  and its desingularization $f: \tilde{\mathcal{Z}} \to {\mathcal{Z}}.$ We also introduce
  a diviser $\partial{\mathcal{Z}}$ of ${\mathcal{Z}}$ and show that
    ${\mathcal{Z}}$ and  $\partial{\mathcal{Z}}$ are CM (cf. Propositions
    \ref{propn6.2} and  \ref{lem11} respectively). We further show that
    ${\mathcal{Z}}$ is irreducible and normal (cf. Lemma \ref{normal}). We show, in fact,  that 
${\mathcal{Z}}$ has rational singularities (cf. Proposition \ref{ratlsing}), which is crucially used in the proof of 
Theorem \ref{prop20}.

In Section \ref{sec4},  we
 use
the relative Kawamata-Viehweg vanishing theorem (cf. Theorem \ref{thm15})
 to obtain  two
crucial vanishing results on the higher direct images of the dualizing sheaf of $\tilde{\mathcal{Z}}$ 
twisted by $\partial \tilde{\mathcal{Z}}$ under $\tilde{\pi}$ and $f$, where $\partial \tilde{\mathcal{Z}} 
:=f^{-1} \partial {\mathcal{Z}}$ and $\tilde{\pi}: \tilde{\mathcal{Z}} \to \bar{\Gamma}$ is the map 
from the big diagram in Section 7 (cf.  Proposition  \ref{prop19} and Theorem \ref{prop20}
respectively).
This sets the stage to prove our main technical Theorem \ref{thma14} (b), which
is achieved in Section \ref{sec3}. 

Finally, we have included an appendix by M. Kashiwara where he determines the
dualizing sheaf of $X^u$.

An informed reader will notice many ideas taken from very interesting papers
[B] and [AGM]  by Brion and Anderson-Griffeth-Miller respectively. However, there
are  several technical difficulties to deal with arising from the infinite dimensional
setup, which has required various different formulations and more involved proofs.
Some of the major differences are:

(1) In the finite case one just works with the opposite Schubert varieties $X^u$ and
 their very explicit BSDH desingularizations. In our
 general symmetrizable Kac-Moody set up,
we need to consider the Richardson varieties $X^u_w$ and their desingularizations
 $Z^u_w$. Our desingularization $Z^u_w$ is not as explicit as the BSDH
 desingularization. Then, we need to draw upon the result due to Kumar-Schwede
 [KuS] that $X^u_w$ has Kawamata log terminal singularities (in particular, rational singularities) 
and use this result (together with a result due to Elkik) 
 to prove that  ${\mathcal{Z}}$ has rational singularities (cf. Proposition \ref{ratlsing}).

(2) Instead of considering just one flag variety in the finite case, we need to
consider the `thick' flag variety and the standard ind flag variety
and the pairing between them. Moreover, the identification of the basis of
$K_T^0(\bar{X})$ dual to the basis of $K^T_0({X})$ given by the structure sheaf
of the Schubert varieties $X_w$  is more delicate.

(3) In the finite case one uses Kleiman's transversality result for the flag
variety $X$. In our infinite
 case, to circumvent the absence of Kleiman's transversality result, we needed to
 prove various local Ext and Tor vanishing results.

We feel that some of the  local  Ext and Tor vanishing results and the results on the
geometry  of Richardson varieties (including the construction of their
desingularizations) proved in this paper are of independent interest.
\vskip2ex

\noindent
{\bf Acknowledgements.} I am very grateful to M. Kashiwara for many
helpful correspondences, for carefully reading a large part of the
 paper and making various suggestions for improvement in the exposition, and
 determining  the dualizing sheaf of the opposite
 Schubert varieties (contained in the appendix by him). It is my pleasure to thank
 M. Brion for pointing out his work on the construction of a desingularization
 of Richardson varieties in the finite case and some suggestions on an earlier
 draft of this paper;
  to N. Mohan Kumar for
 pointing out the `acyclicity lemma' of Peskine-Szpiro (which was also pointed out by
 Dima Arinkin) and his help with the proof of Theorem \ref{thm13} and some other
 helpful conversations; to E. Vasserot for going through the paper, and to the referee 
 for several useful suggestions to improve the exposition (including the shorter proof, 
than our original proof, of Proposition \ref{prop2.6} included here). The result on rational singularity of 
$\mathcal{Z}$ (cf. Proposition \ref{ratlsing}) is added here (during revision of the paper) from our recent joint work with 
S. Baldwin [BaK]. This result is used to give  a shorter proof of Theorem \ref{prop20} (b). 
This work was supported partially by the NSF grant DMS-1201310.

\section{Notation}\label{sec0}

We take the base field to be the field of complex numbers $\bc$. By a variety, we
 mean an algebraic variety over $\bc$, which is reduced but not necessarily irreducible.
For a scheme $X$ and a closed subscheme $Y$, $\co_X(-Y)$ denotes the ideal sheaf of $Y$ in $X$. 

Let $G$ be any symmetrizable Kac-Moody group over $\bc$ completed
along
 the negative roots (as opposed to completed along the positive roots as in [K,
 Chapter 6]) and $G^{\min}\subset G$ be the `minimal' Kac-Moody group  as in
 [K, \S7.4].  Let $B$ be the standard (positive) Borel
subgroup, $B^{-}$ the standard negative Borel subgroup, $H=B\cap B^{-}$ the
standard maximal torus and $W$ the Weyl group  (cf.
[K, Chapter 6]).  Let
  \[
\bar{X} = G/B
  \]
be the `thick' flag variety which contains the standard KM-flag
variety
  \[   X = G^{\min}/B.   \]
If $G$ is not of finite type, $\bar{X}$ is an infinite
dimensional  non quasi-compact  scheme (cf. [Ka, \S4]) and $X$ is an
ind-projective variety (cf. [K, \S7.1]). The group $G^{\min}$; in
particular, maximal torus $H$ acts on $\bar{X}$ and $X$. Let $T$ be the quotient
$H/Z(G^{\min})$, where $Z(G^{\min})$ is the center of $G^{\min}$. (Recall that,
by [K, Lemma 6.2.9(c)], $Z(G^{\min})=\{h\in H: e^{\alpha_i}(h)=1 \,\,
\text{for all the simple roots}\,\,\alpha_i\}.$)
Then, the action of $H$ on $\bar{X}$ (and $X$) descends to an action of $T$.

 For any $w\in W$, we
have the Schubert cell
$$
C_{w}:={BwB/B}\subset X ,
$$
the Schubert variety
$$
X_{w}:=\overline{C_w}\subset X ,
$$
the opposite Schubert cell
$$
C^{w}:={B^{-}wB/B}\subset \bar{X},
$$
and the opposite Schubert variety
$$
X^{w}:=\overline{C^w}\subset \bar{X},
$$
all endowed with the reduced subscheme structures. Then, $X_{w}$ is
a (finite dimensional) irreducible projective subvariety of $X$ and
$X^{w}$ is a finite codimensional irreducible subscheme
of $\bar{X}$ (cf. [K, Section 7.1] and [Ka, \S4]). For any integral
weight $\lambda$ (i.e., any character $e^{\lambda}$ of $H$), we have
a $G$-equivariant line bundle $\mathcal{L}(\lambda)$ on $\bar{X}$
associated to the character $e^{-\lambda}$ of $H$. Explicitly, the character 
 $e^{-\lambda}$ of $H$ extends uniquely to a character (still denoted by  $e^{-\lambda}$) 
of $B$ since $H\simeq B/U$, where $U$ is the unipotent radical of $B$. Now, let $\mathcal{L}(\lambda)$
be the line bundle over $\bar{X}=G/B$ associated to the principal $B$-bundle $G \to G/B$ via the
 one dimensional representation of $B$ given by the character  $e^{-\lambda}$ .

We denote the
representation ring of $T$ by $R(T)$.

Let $\{\alpha_1,\ldots,\alpha_{r}\}\subset \mathfrak{h}^{*}$ be the
set of simple roots,
$\{\alpha_1^{\vee},\ldots,\alpha^{\vee}_{r}\}\subset \mathfrak{h}$
the set of simple coroots and $\{s_1,\ldots, s_{r}\}\subset W$ the
corresponding simple reflections, where $\mathfrak{h}=\Lie H$. Let
$\rho\in \mathfrak{h}^{*}$ be any integral weight satisfying
$$
\rho(\alpha^{\vee}_{i})=1,\quad\text{for all}\quad 1\leq i\leq r.
$$

When $G$ is a finite dimensional semisimple group, $\rho$ is unique,
but for a general Kac-Moody group $G$, it may not be unique.

For any $v\leq w \in W$, consider the \emph{Richardson variety}
$$
X^{v}_{w}:=X^{v}\cap X_{w} \subset X
$$
and its boundary
$$\partial X^v_w := (\partial X^v)\cap X_w
$$
 both endowed with the reduced subvariety structures, where $\partial X^{v}:=X^{v}
 \backslash
C^v$. We also set $\partial
X_{w}:=X_{w}\backslash C_w$. (By [KuS, Proposition 5.3], $X^{v}_{w}$ and
$\partial X^v_w$, endowed with the scheme theoretic intersection structure, are
Frobenius split in char. $p>0$; in particular, they are reduced. More generally, any scheme theoretic intersection $X_{w_1}\cap \cdots \cap X_{w_m}\cap
X^{v_1} \cap \cdots \cap X^{v_n}$ is reduced by loc. cit.)

\section{Identification of the dual of the structure sheaf basis}\label{sec5}

  \begin{definition} \label{n2.1}
  For a quasi-compact scheme $Y$, an $\co_{Y}$-module $\cs$ is called {\it coherent}
  if it is finitely presented as an $\co_{Y}$-module and any $\co_{Y}$-submodule of finite
  type admits a
finite presentation.

  A subset $S\subset W$
is called an {\it ideal} if
 for $x\in S$ and $y\leq x\Rightarrow y\in S$. An $\co_{\bar{X}}$-module $\cs$ is called {\it coherent} if
   $\cs_{|V^S}$ is a
coherent $\co_{V^S}$-module for any finite ideal $S\subset W$, where $V^S$ is the quasi-compact open subset
 of $\bar{X}$ defined by
 $$V^S = \bigcup_{w\in S} wU^- B/B.$$
 Let $K^0_T(\bar{X})$ denote the Grothendieck group of
$T$-equivariant coherent $\co_{\bar{X}}$-modules $\cs$. Observe that since the coherence condition on 
$\cs$  is imposed only for  $\cs_{|V^S}$ for finite ideals $S\subset W$,  $K^0_T(\bar{X})$ can be thought 
of as the inverse limit of $K^0_T(V^S)$, as $S$ varies over the finite ideals of $W$ (cf. [KS, $\S$2]).

 Similarly,
define $K^T_0(X) := \Limit_{n\to\infty} K^T_0(X_n)$, where $\{
X_n\}_{n\geq 1}$ is the filtration of $X$ giving the ind-projective
variety structure (i.e., $X_n = \bigcup_{\ell (w)\leq n} BwB/B$) and
$K^T_0(X_n)$ is the Grothendieck group of $T$-equivariant coherent
sheaves on the projective variety $X_n$.

We also define
  \[
K^{\optop}_T(X) := \Invlt_{n\to\infty} K^{\optop}_T(X_n),
  \]
where $K^{\optop}_T(X_n)$ is the $T$-equivariant topological $K$-group of the
 projective variety $X_n$.

Let $*:K^{\optop}_T(X_n)\to K^{\optop}_T(X_n)$ be the involution induced from
the operation which takes a $T$-equivariant vector bundle to its dual. This,
 of course, induces the involution $*$ on $K^{\optop}_T(X)$.
\end{definition}
We  recall the `basis' $\{\psi^w\}_{w\in W}$ of $K^{\optop}_T(X)$ given by Kostant-Kumar. (Actually, our  $\psi^w = *\tau^{w^{-1}}$, where $\tau^w$ is the original `basis' given by them in [KK, $\S$3].)
\begin{definition}  \label{psibasis} For $w\in W$, fix a reduced decomposition
  $\fw =(s_{i_1}, \dots, s_{i_n})
  $ for $w$ (i.e., $w =s_{i_1} \dots s_{i_n}$ is a reduced decomposition) and
  let $\theta_{\fw} : Z_{\fw} \to X_w$ be the Bott-Samelson-Demazure-Hansen (for short
  BSDH) desingularization (cf. [K, \S7.1]).  By [KK, Proposition 3.35],
  $K^0_T(Z_{\fw}) \to K_T^{\optop}(Z_{\fw})$ is an isomorphism,  where
  $K^0_T(Z_{\fw})$ is the Grothendieck group associated to the semigroup of
  $T$-equivariant algebraic vector bundles on $Z_{\fw}$. (Observe that the action of $H$ on
$Z_{\fw}$ descends to an action of $T$.)

For any $\psi \in K^{\optop}_T(X)$ and $w\in W$, define the `virtual' Euler-Poincar\'e characteristic by
\[\tilde{\chi}(X_w, \psi):= \chi(Z_\fw, \theta_\fw^*(\psi))\in R(T).\]
By [KK, Proposition 3.36], $\tilde{\chi}(X_w, \psi)$ is well defined, i.e., it does not depend upon the particular
 choice of the reduced decomposition $\fw$ of $w$. 

Now, define $\psi^w \in K^{\optop}_T(X)$ as the unique element satisying 
\beqn 
\tilde{\chi}(X_v, \psi^w)= \delta_{v, w}, \,\,\,\text{for all} \,\, v\in W.
\eeqn
Such an element  $\psi^w$ exists and is uniqe. 
Moreover, $\{\psi^w\}_{w\in W}$ is a `basis' in the sense that any element of $K^{\optop}_T(X)$ is uniquely written as a linear
combination of  $\{\psi^w\}_{w\in W}$ with possibly infinitely many nonzero coefficients. 
Conversely, an arbitrary linear combination of $\psi^w$ is an element of  $K^{\optop}_T(X)$.
  \end{definition}
For any $w\in W$,
  \[  [\co_{X_w}] \in K^T_0(X).  \]

  \begin{lemma}  $\bigl\{ [\co_{X_w}]\bigr\}_{w\in W}$ forms a basis of $K^T_0(X)$ as an $R(T)$-module.
  \end{lemma}

  \begin{proof} By [CG, \S 5.2.14 and Theorem 5.4.17], the result follows.
  \end{proof}

  For $u\in W$, by [KS, \S 2], $\co_{X^u}$ is a coherent $\co_{\bar{X}}$-module.
  In particular, $\co_{\bar{X}}$ is a coherent $\co_{\bar{X}}$-module.

 Consider the quasi-compact open subset $V^u := uU^- B/B \subset \bar{X}$.
The following lemma is due to Kashiwara-Shimozono [KS, Lemma 8.1].

  \begin{lemma} \label{2.4}   Any $T$-equivariant coherent sheaf $\cs$ on $V^u$ admits a free resolution in
   ${Coh}_T (\co_{V^u}):$
   $$
    0 \to S_n\otimes \co_{V^u} \to \cdots\to S_1\otimes\co_{V^u}
    \to S_0\otimes\co_{V^u} \to \cs \to 0,
  $$
where $S_k$ are finite dimensional $T$-modules and ${Coh}_T
(\co_{V^u})$ denotes the abelian category of $T$-equivariant
coherent $\co_{V^u}$-modules. \qed
  \end{lemma}

Define a pairing
$$
\ip< \, ,\, > : K^0_T(\bar{X}) \otimes K^T_0(X) \to R(T),\,\,
\ip<[\cs], [\cf]>  = \sum_i (-1)^i \chi_T \bigl(X_n, \tor_
i^{\co_{\bar{X}}}
 (\cs,\cf ) \bigr),$$
if $\cs$ is a $T$-equivariant coherent sheaf on $\bar{X}$ and $\cf$
is a $T$-equivariant coherent sheaf on ${X}$ supported in $X_n$ (for
some $n$), where $\chi_T$ denotes the $T$-equivariant
Euler-Poincar\'{e} characteristic.

  \begin{lemma}  \label{2.5} The above pairing is well defined.
  \end{lemma}

  \begin{proof}  By Lemma \ref{2.4}, for any $u\in W$, there exists $N(u)$
  (depending upon $\cs$) such that
  $\tor_j^{\co_{\bar{X}}}(\cs,\cf) =0$
for all $j > N(u)$ in the open set $V^u$.  Now, let $j > \max_{\ell
(u)\leq n} N(u)$,
 where $\cf$ has support in $X_n$.  Then,
  \[
\tor_j^{\co_{\bar{X}}} (\cs,\cf) = 0\quad\text on \quad
\bigcup_{\ell (u)\leq n} V^u
  \]
and hence $\tor_j^{\co_{\bar{X}}}(\cs,\cf)=0$ on $\bar{X}$, since
$BuB/B \subset uB^-B/B$ and hence supp $\cf\subset X_n \subset
\cup_{\ell (u)\leq n}\, V^u$.

Of course, for any $j\geq 0$, $\tor_j^{\co_{\bar{X}}}(\cs,\cf)$ is a
sheaf supported on $X_n$ and it is  $\co_{X_n}$-coherent on the open set $X_n\cap
V^u$ of $X_n$ for any $u\in W$.  Thus, $\tor_j^{\co_{\bar{X}}}
(\cs,\cf)$ is a $\co_{X_n}$-coherent sheaf and hence
  \[
\chi_T\bigl( \bar{X}, \tor_j^{\co_{\bar{X}}} (\cs,\cf)\bigr) =
\chi_T \bigl( X_n, \tor_j^{\co_{\bar{X}}} (\cs,\cf)\bigr)
  \]
  is well defined. This proves the lemma.
  \end{proof}
By [KS, Proof of Proposition 3.4], for any $u\in W$,
  \beqn\label{eq1.0}
\ext^k_{\co_{\bar{X}}} (\co_{X^u}, \co_{\bar{X}}) =0 \quad\forall k\neq \ell (u).
  \eeqn
Define the sheaf
  \beqn\label{neweq5'}
\om_{X^u} := \ext^{\ell (u)}_{\co_{\bar{X}}}
\bigl(\co_{X^u}, \co_{\bar{X}} \bigr)\otimes\cl (-2\rho ),
  \eeqn
  which, by the analogy with the Cohen-Macaulay (for short CM) schemes of finite type, will be called
  the {\it dualizing sheaf} of $X^u$.

Now, set the $T$-equivariant sheaf on $\bar{X}$
  \begin{align*}
\xi^u &:= e^{-\rho} \cl (\rho )\om_{X^u} \\
&= e^{-\rho} \cl (-\rho ) \ext^{\ell (u)}_{\co_{\bar{X}}}
(\co_{X^u}, \co_{\bar{X}} ).
  \end{align*}
By Theorem \ref{prop:main}, $\xi^u$ is the ideal sheaf of $\partial X^u$ in $X^u$.

By Lemma \ref{2.4}, for any $v\in W$, $\co_{X^u \cap V^v}$ admits the resolution
  \[
0 \to \cf_n \to \cdots \to \cf_0 \to \co_{X^u\cap V^v} \to 0
  \]
by free $\co_{V^v}$-modules of finite rank.  Thus, the sheaf
$\ext^{\ell (u)}_{\co_{\bar{X}}}(\co_{X^u}, \co_{\bar{X}})$ restricted to $V^v$
is given by the $\ell (u)$-th cohomology of the sheaf sequence
  \[
0 \leftarrow \home_{\co_{\bar{X}}} (\cf_n, \co_{\bar{X}}) \leftarrow
\home_{\co_{\bar{X}}} (\cf_{n-1}, \co_{\bar{X}}) \leftarrow \cdots \leftarrow
\home_{\co_{\bar{X}}} (\cf_0, \co_{\bar{X}})\leftarrow 0.
  \]
In particular, $\ext^{\ell (u)}_{\co_{\bar{X}}} (\co_{X^u},
\co_{\bar{X}})$ restricted to $V^v$ is $\co_{V^v}$-coherent and hence so is
$\xi^u$ as an $\co_{\bar{X}}$-module.  Hence,
$$[\ext^{\ell (u)}_{\co_{\bar{X}}} (\co_{X^u},
 \co_{\bar{X}})]\in K^0_T(\bar{X}).$$

  \begin{proposition} \label{prop2.6} For any $u,w\in W$,
    \[
\ip<[\xi^u], [\co_{X_w}]> = \delta_{u,w}.
   \]
  \end{proposition}

  \begin{proof}\footnote{We thank the referee for this shorter proof than our original proof.}  By definition, 
\[
\ip<[\xi^u], [\co_{X_w}]> = \sum_i (-1)^i \chi_T \bigl( X_n, \tor_i^{\co_{\bar{X}}} (\xi^u,\co_{X_w})\bigr),
   \]
where $n$ is taken such that $n \geq \ell(w)$.
Thus, by (subsequent) Proposition \ref{propa18},
\beqn \label{neweqn1}
\ip<[\xi^u], [\co_{X_w}]> = \chi_T \bigl( X_n, \xi^u\otimes_{\co_{\bar{X}}} \,\co_{X_w}\bigr).
\eeqn
By Theorem \ref{prop:main} and Corollary \ref{newcor5.5}, we have the sheaf exact sequence:
\[0 \to \xi^u\otimes_{\co_{\bar{X}}}\,\co_{X_w} \to \co_{X^u}\otimes_{\co_{\bar{X}}}\,\co_{X_w} \to 
 \co_{\partial X^u}\otimes_{\co_{\bar{X}}}\,\co_{X_w} \to 0.\]
Thus,
\beqn\label{neweqn2}
\chi_T (X_n,  \xi^u\otimes_{\co_{\bar{X}}}\,\co_{X_w})=\chi_T (X_n,  \co_{X_w^u}) -
\chi_T\left(X_n,  \co_{(\partial X^u) \cap X_w}\right),
\eeqn
since $\co_Y\otimes_{\co_{\bar{X}}}\,\co_{Z} =\co_{Y\cap Z}.$ By Proposition \ref{n5.6}, when nonempty, 
$X^u_w$ is an irreducible variety and hence $(\partial X^u) \cap X_w =\cup_{w\geq v >u}\,X^v_w$ is 
connected (if nonempty) since $w\in X^v_w$
for all $u< v\leq w$. If $u \not\leq w$, $X^u_w$ is empty and hence by \eqref{neweqn1} - \eqref{neweqn2},
\[\ip<[\xi^u], [\co_{X_w}]> =0.\]
So, assume that $u\leq w$. In this case, $X^u_w$ is nonempty. Moreover, by [KuS, Corollary 3.2],
\[H^i(X_n,  \co_{X_w^u}) = 0, \,\,\forall i>0.\]
Also, by Corollary \ref{newcor5.5},
\[H^i (X_n,  \co_{(\partial X^u) \cap X_w}) = 0, \,\,\forall i>0.\]
Thus, for $u\leq w$,
\beqn\label{neweq3}
\chi_T (X_n,  \co_{X_w^u}) =1,
\eeqn
and for $u< w$, 
\beqn \label{neweqn4}
\chi_T\left(X_n,  \co_{(\partial X^u) \cap X_w}\right)= 1.
\eeqn
Thus, by  \eqref{neweqn1} - \eqref{neweqn2},
\[\ip<[\xi^u], [\co_{X_w}]> =0, \,\,\,\text{for}\,\, u<w.\]
Finally, take $u=w$. In this case 
\[\ip<[\xi^u], [\co_{X_w}]> =1.\]
This proves the proposition.\qed
\end{proof}

\section{Geometric identification of the $T$-equivariant $K$-theory
structure constants and  statements of the main results}\label{sec6}

Express the product in topological $K$-theory $K^{\optop}_T(X)$:
  \[
\psi^u\cdot\psi^v = \sum_w p^w_{u,v} \psi^w, \quad\text{for } p^w_{u,v}\in R(T).
  \]
(For fixed $u,v\in W$, infinitely many $p^w_{u,v} $ could be nonzero.)

Also, express the co-product in $K^T_0(X)$:
  \[
\Del_* [\co_{X_w}] = \sum_{u,v} q^w_{u,v} [\co_{X_u}]\otimes [\co_{X_v}] ,
  \]
where $\Del : X\to X\times X$ is the diagonal map.

  \begin{proposition} \label{n3.1} For all $u,v,w\in W$,
    \[  p^w_{u,v} = q^w_{u,v}.  \]
  \end{proposition}

  \begin{proof}  For $w\in W$, fix a reduced decomposition
  $\fw =(s_{i_1}, \dots, s_{i_n})
  $ for $w$ and let $\theta=\theta_{\fw}:Z_\fw \to X_w$ be the 
  BSDH desingularization as in Definition \ref{psibasis}.
  By [KK, Proposition 3.39]
  (where $\chi_T$ is the $T$-equivariant Euler-Poincar\'e characteristic),
  \begin{align}\label{eq3.1}
\chi_T \Bigl( \theta^* (\psi^u \cdot \psi^v)\bigr) &= \chi_T\Biggl( \sum_{w_1} p^{w_1}_{u,v}
\, \theta^*(\psi^{w_1}) \Biggr) \notag\\
 &= p^w_{u,v} .
   \end{align}
On the other hand,
  \begin{align}\label{neweqn5}
\theta^*(\psi^u \cdot \psi^v) &= \theta^*\Del^* (\psi^u\boxtimes\psi^v)\notag\\
&= \Del^*_{\fw} (\theta\times\theta )^* (\psi^u \boxtimes\psi^v)\notag
\\
&= \Del^*_{\fw}\bigl(\theta^* \psi^u\boxtimes\theta^*\psi^v\bigr),
  \end{align}
  where $\Del_{\fw}: Z_{\fw}\to Z_{\fw}\times Z_{\fw}$
is the diagonal map.

In the following proof, for any morphism $f$ of schemes, we abbreviate $Rf_*$ by $f_!$.

Let $\pi : Z_{\fw} \to pt$ and let ${\Del_{\fw}}_*[\co_{Z_{\fw}}] = \sum_{\fu ,\fv
\leq\fw} \hat{q}^{\fw}_{\fu ,\fv}[\co_{Z_{\fu}}]\boxtimes [\co_{Z_{\fv}}]$ for
some unique $\hat{q}^{\fw}_{\fu ,\fv} \in R(T)$, where $\fu \leq \fw$ means that
$\fu$ is a subword of $\fw$.  (This decomposition is due to the fact that
$ [\co_{Z_{\fu}}]_{\fu\leq \fw}$ is an $R(T)$-basis of
$$K_0^T(Z_{\fw}) \simeq K^0_T(Z_{\fw}) \simeq K_T^{\optop}(Z_{\fw}),
 $$
 where $K_0^T(Z_{\fw})$ is the Grothendieck group associated to the semigroup of
  $T$-equivariant coherent sheaves on $Z_{\fw}$. For the latter isomorphism, see [KK, Proposition 3.35].)
Then,
  \begin{align}\label{e3.2}
\chi_T\Bigl( \theta^* (\psi^u\cdot \psi^v)\Bigr) &=
\pi_! \Bigl(\Del^*_{\fw} (\theta^*\psi^u\boxtimes\theta^*\psi^v )
\Bigr),\,\,\,\text{by}\,\,\eqref{neweqn5} \notag\\
 &= (\pi\times\pi )_!  \Bigl(\Del_{\fw *}\bigl(\Del^*_{\fw}
 (\theta^*\psi^u\boxtimes\theta^*\psi^v)\bigr)\Bigr)
 \notag\\
 &= (\pi\times\pi )_!  \bigl((\theta^*\psi^u\boxtimes\theta^*\psi^v\bigr) \cdot
 \bigl(\Del_{\fw *}[\co_{Z_{\fw}}])\bigr),\,\,\,\text{by the projection formula} \notag\\
 &= (\pi\times\pi )_!  \Bigl( (\theta^*\psi^u\boxtimes\theta^*\psi^v) \cdot
 \bigl(\sum_{\fu ,\fv} \hat{q}^{\fw}_{\fu ,\fv} [\co_{Z_{\fu}}]\boxtimes
 [\co_{Z_{\fv}}]\bigr)\Bigr),\,\,\text{for some}\,\,  \hat{q}^{\fw}_{\fu ,\fv}\in R(T) \notag\\
 &= \sum_{\fu ,\fv} \hat{q}^{\fw}_{\fu ,\fv} \,\chi_T\bigl( \theta^*\psi^u\cdot
 [\co_{Z_{\fu}}]\bigr) \chi_T\bigl(\theta^*\psi^v \cdot [\co_{Z_{\fv}}]\bigr) \notag\\
 &= \sum_{ \substack{\mu (\fu )=u\\ \mu (\fv )=v} } \hat{q}^{\fw}_{\fu ,\fv},
   \end{align}
where the last equality follows since
\beqn \label{e3.4'} \chi_T(\theta^*\psi^u\cdot
 [\co_{Z_{\fu}}])= \delta_{u, \mu(\fu)}, \eeqn
 where $\mu (\fu )$ denotes the Weyl group element $u$ if the standard map
 $Z_{\fu} \to G/B$ has image precisely equal to $X_u$. To prove the above identity
 \eqref{e3.4'},
 use [KK, Propositions 3.36, 3.39],
  and the proof of [K, Corollary 8.1.10]. (Actually, we need the extension
 of [KK, Proposition 3.36] for non-reduced words $\fv$, but the proof of
 this extension is identical.)

 From the identity
  \[
\Del_{\fw *} [\co_{Z_{\fw}}] = \sum_{\fu ,\fv\leq\fw} \hat{q}^{\fw}_{\fu ,\fv} [\co_{Z_{\fu}}]\boxtimes [\co_{Z_{\fv}}] ,
  \]
we get
  \begin{align}\label{e3.3}
\Del_* \theta_! [\co_{Z_{\fw}}] &= (\theta\times\theta )_!\, \Del_{\fw *}[\co_{Z_{\fw}}] \notag \\
 &= \sum_{\fu ,\fv} \hat{q}^{\fw}_{\fu ,\fv} \theta_![\co_{Z_{\fu}}]\boxtimes \theta_![\co_{Z_{\fv}}] \notag\\
 &= \sum_{u_1,v_1\leq w}\; \sum_{ \substack{ \mu (\fu )=u_1\\ \mu (\fv )=v_1} } \hat{q}^{\fw}_{\fu ,\fv} [\co_{X_{u_1}}]\boxtimes [\co_{X_{v_1}}],
   \end{align}
by [K, Theorem 8.2.2(c)].  Moreover, since
  \beqn \label{e3.4}
\Del_*\, \theta_! [\co_{Z_{\fw}}] = \Del_*\, [\co_{X_{w}}] = \sum\, q^w_{u_1,v_1}
[\co_{X_{u_1}}]\boxtimes [\co_{X_{v_1}}],
  \eeqn
we get (equating \eqref{e3.3} and \eqref{e3.4}) for any $u_1, v_1 \leq w :$
  \beqn \label{e3.5}
 q^w_{u_1, v_1} = \sum_{ \substack{ \mu (\fu )=u_1\\ \mu (\fv )=v_1} }
 \hat{q}^{\fw}_{\fu ,\fv}.
 \eeqn
Combining \eqref{eq3.1}, \eqref{e3.2} and \eqref{e3.5}, we get $p^w_{u,v} = q^w_{u,v}$.
 This proves the proposition.
 \end{proof}

\begin{lemma} (Due to M. Kashiwara)  The $R(T)$-span of $\{[\xi^u]\}_{u\in W}$ inside $K^{0}_T(\bar{X})$
(where we allow an arbitrary infinite sum, which makes sense as an element of $K^{0}_T(\bar{X})$) coincides with 
 $K^{0}_T(\bar{X})$. 
\end{lemma}
\begin{proof}
 To prove this, write $[\xi^u]$ as a linear
combination of $[\co_{X^v}]$ by Theorem \ref{prop:main}. Then, it is an upper triangular $R(T)$-matrix with
 diagonal terms equal to $1$. By [KS, \S2],  $[\co_{X^v}]$ is a `basis' of  $K^{0}_T(\bar{X})$. This proves the lemma.
\end{proof}
  
By Proposition \ref{prop2.6},
 $\{[\xi^u]\}_{u\in W}$ are independent over $R(T)$ even allowing infinite sums.

Now, express the product in $K^{0}_T(\bar{X})$:
  \[
[\xi^u]\cdot[\xi^v] =  \sum_w d^w_{u,v} [\xi^w] , \quad\text{for } d^w_{u,v}\in R(T).
  \]

Let $\bar{\Del} : \bar{X}\to \bar{X}\times \bar{X}$ be the diagonal map. Then,
\[[\xi^u]\cdot[\xi^v] = {\bar{\Del}}^*([\xi^u\boxtimes \xi^v]).\]
  \begin{lemma}  For all $u,v,w\in W$,
    \[  p^w_{u,v} = d^w_{u,v}.  \]
  \end{lemma}

  \begin{proof}  For any $w\in W$,
  \begin{align*}
  \langle{\bar{\Del}}^*([\xi^u\boxtimes \xi^v]), [\co_{X_w}]\rangle &=
  \langle [\xi^u\boxtimes \xi^v], {{\Del}}_*[\co_{X_w}]\rangle\\
  &=\langle [\xi^u\boxtimes \xi^v],\sum_{u',v'}p^w_{u',v'} [\co_{X_{u'}}]\otimes
  [\co_{X_{v'}}]\rangle, \,\,\,\text{by Proposition \ref{n3.1}}\\
  &=p^w_{u,v}, \,\,\,\text{by Proposition \ref{prop2.6}}.
  \end{align*}
  On the other hand,
  \begin{align*}
  \langle{\bar{\Del}}^* ([\xi^u\boxtimes \xi^v]), [\co_{X_w}]\rangle  &=
  \langle [\xi^u] \cdot [\xi^v], [\co_{X_w}]\rangle \\
  &=d^w_{u,v}, \,\,\, \text{by Proposition \ref{prop2.6} again}.
\end{align*}
This proves the lemma.
  \end{proof}

Fix a large $N$ and let
  \[
\bp = (\bp^N)^{r} \qquad\text{($r =  \dim T$)} .
  \]
For any ${\bf j} = (j_1, \dots , j_{r}) \in [N]^{r}$, where $[N] = \{ 0,1,\dots ,N\}$, set
  \[
\bp^{\bf j} = \bp^{N-j_1}\times \cdots \times \bp^{N-j_{r}} .
  \]

{\it We fix an identification $T\simeq (\bc^*)^{r}$ throughout the paper satisfying
the condition that for any positive root $\alpha$, the character $e^\alpha$
(under the identification) is given by $z_1^{d_1(\alpha)}\dots  z_r^{d_r(\alpha)}$
for some $d_i(\alpha) \geq 0$, where $(z_1, \dots, z_r)$ are the standard coordinates
on  $(\bc^*)^{r}$ .} One such identification $T\simeq (\bc^*)^{r}$ is given by
$t \mapsto \bigl(e^{\alpha_1}(t), \dots, e^{\alpha_r}(t)\bigr).$ {\it This will be our
default choice.}

Let $E(T)_{\bp}:= (\bc^{N+1}\setminus \{0\})^r$ be the total space of the
standard principal $T$-bundle $E(T)_{\bp} \to \bp$. We can view
$E(T)_{\bp} \to\bp$ as a finite dimensional
 approximation of the classifying bundle for $T$.   Let $\pi : X_{\bp} :=
 E(T)_\bp\times^T X \to
 \bp$ be the fibration with fiber $X=G/B$ associated to the principal $T$-bundle
 $E(T)_{\bp} \to \bp$, where we twist the standard action of $T$ on $X$ via
  \beqn \label{e3.5.1}
 t\odot x=t^{-1}x.
 \eeqn
 For any $T$-subscheme $Y\subset X$, we denote
 $Y_{\bp} := E(T)_\bp \times^T Y\subset X_{\bp}$.

The following theorem follows easily by using [CG, \S5.2.14] together with
[CG, Theorem 5.4.17] applied to the vector bundles $(BwB/B)_\bp \to \bp$.
 \begin{theorem}
$K_0(X_{\bp}):=\lt_{n\to\infty} K_0((X_n)_{\bp})$
 is a free module over the ring $K_0(\bp ) = K^0(\bp )$ with basis
 $\{[\co_{(X_w)_{\bp}}]\}_{w\in W}$, where $K_0$ (resp. $K^0$) denotes the
 Grothendieck group
 associated to the semigroup of coherent sheaves (resp. locally free sheaves).

Thus, $K_0(X_{\bp})$ has a $\bz$-basis
  \[
\bigl\{ \pi_X^* ([\co_{\bp^{\bf j}}])\cdot [\co_{(X_w)_{\bp}}]
\bigr\}_{{\bf j}\in [N]^{r}, w\in W } \,,
  \]
  where we view $[\co_{\bp^{\bf j}}]$ as an element of $K_0(\bp)=K^0(\bp)$. \qed
  \end{theorem}

Let $Y:=X\times X$.  The diagonal map $\Del : X\to Y$ gives rise to the embedding
  \[
\tilde{\Del} : X_{\bp} \to Y_{\bp} = E(T)_\bp \times^T Y\simeq
X_{\bp}\times_\bp X_{\bp}.
  \]
Thus, we get (denoting the projection $Y_{\bp} \to \bp$ by $\pi_Y$)
 \beqn \label{e0.0}
\tilde{\Del}_* [\co_{(X_w)_{\bp}}] =
\sum_{ \substack{ u,v\in W\\ {\bf j}\in [N]^{r}} }
c^w_{u,v}({\bf j}) \pi_Y^*([\co_{\bp^{\bf j}}])\cdot  [\co_{(X_u\times X_v)_\bp}]
\in K_0(Y_{\bp}),
   \eeqn
for $c^w_{u,v}({\bf j}) \in\bz$.  Let
$$\bp_{\bf j} = \bp^{j_1}\times \cdots \times \bp^{j_{r}},$$
 $$\partial\bp_{\bf j} = \bigl( \bp^{j_1-1}\times \bp^{j_2}\times \cdots \times \bp^{j_{r}}
 \bigr) \cup \cdots \cup \bigl(\bp^{j_1}\times\cdots\times \bp^{j_{r -1}}
 \times\bp^{j_{r}-1}\bigr),$$
 where we interpret $\bp^{-1}=\emptyset$. It is easy to see that, under the standard pairing
 on $K^0(\bp )$,
  \beqn \label{e3.6}
\ip< [\co_{\bp^{\bf j}}] , [\co_{\bp_{{\bf j}'}} (-\partial \bp_{{\bf j}'})]> =
\del_{{\bf j},{\bf j}'}.
  \eeqn
Alternatively, it is a special case of [GK, Proposition 2.1 and \S6.1].

Let $\bar{Y}= \bar{X}\times \bar{X}$ and $K^0({\bar{Y}_\bp})$ denote
the Grothendieck group
 associated to the semigroup of coherent $\co_{{\bar{Y}_\bp}}$-modules $\cs$, i.e.,
 those $\co_{\bar{Y}_\bp}$-modules $\cs$ such that $\cs_{|{(V^{S_1}\times
 V^{S_2})}_\bp}$ is a coherent $\co_{{(V^{S_1}\times
 V^{S_2})}_\bp}$-module for all finite ideals $S_1,S_2 \subset  W$. Also, let
 $\hat{\cl}(\rho \boxtimes \rho)$ be the line bundle on $\bar{Y}_\bp$
 defined as
 \[E(T)_\bp \times^T e^{-2\rho} \bigl(\cl(-\rho)\boxtimes \cl(-\rho)\bigr) \to
 \bar{Y}_\bp,\]
 where the action of $T$ on the line bundle $e^{-2\rho} \bigl(\cl(-\rho)\boxtimes
 \cl(-\rho)\bigr)$ over $\bar{Y}$ is also twisted the same way as in \eqref{e3.5.1}.
  \begin{lemma} \label{lemma3.3} With the notation as above,
  \[
c^w_{u,v}({\bf j}) = \Big\langle\pi_{\bar{Y}}^*[\co_{\bp_{\bf j}}(-\partial\bp_{\bf j})]\cdot
[\widetilde{\xi^u\boxtimes\xi^v}], \tilde{\Del}_*[\co_{(X_w)_{\bp}}]\Big\rangle ,
  \]
where  $\pi_{\bar{Y}}: \bar{Y}_{\bp} \to \bp$ is the projection,  the coherent sheaf  $\widetilde{\xi^u\boxtimes\xi^v}$ on
$\bar{Y}_\bp$  is defined as:
\[ \hat{\cl}(\rho \boxtimes \rho)\otimes \ext^{\ell (u)+\ell(v)}_{\co_{\bar{Y}_\bp}}
\bigl(\co_{{(X^u\times X^v)}_\bp}, \co_{\bar{Y}_\bp} \bigr),\]
and the pairing
$$\ip<\, ,\, >: K^0(\bar{Y}_\bp)\otimes K_0(Y_{\bp})\to\bz $$
 is similar to the pairing defined earlier.  Specifically,
 \[ \ip<[\cs] ,[\cf]> = \sum_i (-1)^i\, \chi (\bar{Y}_\bp,
\tor_i^{\co_{\bar{Y}_\bp}} (\cs ,\cf)),
  \]
  where $\chi$ is the Euler-Poincar\'e characteristic.
  \end{lemma}

  \begin{proof}
    \begin{align*}
\Big\langle \pi_{\bar{Y}}^*[\co_{\bp_{\bf j}}(-\partial\bp_{\bf j})] & \cdot
[\widetilde{\xi^u\boxtimes\xi^v}], \tilde{\Del}_*[\co_{(X_w)_{\bp}}]\Big\rangle\\
 &= \Big\langle \pi_{\bar{Y}}^*[\co_{\bp_{\bf j}}(-\partial\bp_{\bf j})]\cdot
 [\widetilde{\xi^u\boxtimes\xi^v}],
 \sum_{ \substack{ u', v'\in W,\\ {\bf j}'\in [N]^{r}}} c^w_{u',v'}({\bf j}')\,
 \pi_Y^*([\co_{\bp^{{\bf j}'}}])[\co_{{(X_{u'}\times X_{v'})}_\bp}]\Big\rangle\\
 &= c^w_{u,v}({\bf j}),\,\,\,\text{by Proposition \ref{prop2.6} and the identity
 \eqref{e3.6}.}
   \end{align*}
This proves the lemma.
  \end{proof}

  \begin{definition}[Mixing group] \label{mixdef} Let $T$ act on $B$ via the  inverse conjugation, i.e.,
    \[  t\cdot b = t^{-1}bt, \;\; t\in T, b\in B.   \]
Consider the ind-group scheme (over $\bp$):
  \[  B_\bp = E(T)_\bp \times^T B \to \bp .   \]

Note that $B_\bp$ is not a principal $B$-bundle since there is no
right action of $B$ on $B_\bp$.  Let $\Gam_0$ be the group of global sections
of the bundle $B_\bp$ under the pointwise multiplication.  (Recall that $\Gam_0$ can be identified with the set of regular maps
$f:  E(T)_\bp \to B$ such that $f(e\cdot t)=t^{-1}\cdot f(e), \forall e\in  E(T)_\bp$ and $t\in T$.)  Since $GL(N+1)^{r}$
acts canonically on $B_\bp$
 compatible with its action on $\bp = (\bp^N)^{r}$, it also acts on $\Gam_0$ via the pull-back.
 Let $\Gam_B$ be the semi-direct product $\Gam_0\rtimes GL(N+1)^{r}$:
  \[
1 \to \Gam_0 \to \Gam_B \to GL(N+1)^{r} \to 1.
  \]

Then, $\Gam_B$ acts on $X_{\bp}$ with orbits precisely equal to
$\{ {(BwB/B)}_\bp\}_{w\in W}$, where the action of the subgroup $\Gamma_0$
is via the standard action of $B$ on $X$. This follows from the following lemma.

  \end{definition}

\begin{lemma} \label{lem6}
For any $\bar{e}\in \mathbb{P}$ and any $b$ in the fiber of
$B_\bp$ over $\bar{e}$, there exists a
section $\gamma\in \Gamma_{0}$ such that $\gamma(\bar{e})=b$.
\end{lemma}

\begin{proof} For a character $\lambda$ of $T$, let $\co(\lambda)$ be the line
bundle on $\bp$ associated to the principal $T$-bundle $E(T)_\bp \to \bp$ via
 the character $\lambda$ of $T$. For any positive real root $\alpha$, let $U_\alpha
 \subset U$ be the corresponding one parameter subgroup (cf. [K, \S6.1.5(a)]), where
 $U$ is the unipotent radical of $B$.
 Then, $B_\bp$ contains the subbundle $H\times \co(-\alpha)$. By the assumption
 on the identification $T\simeq (\bc^*)^r$, for any positive root $\alpha$,
 $\co(-\alpha)$ is globally generated. Thus, $\Gamma_{0}(\bar{e})\supset H\times
 U_\alpha$. Since $\Gamma_{0}$ is a group and by [K, Definition 6.2.7] the group $U$ is generated by
 the subgroups $\{U_\alpha\}$, where $\alpha$ runs over the positive real roots,
 we get the lemma.
\end{proof}

\begin{lemma}\label{connected} $\Gamma_B$ is connected.
\end{lemma}
\begin{proof} It suffices to show that $\Gamma_0$ is connected. But, $\Gamma_0
\simeq H \times \Gamma \bigl(E(T)_\bp \times^T U\bigr)$,
where  $\Gamma \bigl(E(T)_\bp \times^T U\bigr)$ denotes the group of sections
of the bundle $E(T)_\bp \times^T U  \to \bp$. Thus, it suffices
 to show that $\Gamma \bigl(E(T)_\bp \times^T U\bigr)$ is connected. Using the
$T$-equivariant 
 contraction of $U$ (in the analytic topology) given in [K, Proposition 7.4.17],
 it is easy to see that $\Gamma \bigl(E(T)_\bp \times^T U\bigr)$ is contractible.
 In particular, it is connected.
 \end{proof}

Similarly, we define $\Gam_{B\times B}$ by replacing $B$ by $B\times B$ and $T$ by
 the diagonal $\Delta T\subset T\times T$ and we abbreviate it by $\Gam$. Observe that
Lemmas \ref{lem6} and \ref{connected} remain true (by the same proof) for $\Gamma_B$ 
replaced by $\Gamma$. (For the proof of Lemma \ref{lem6}, observe that the weights of
 $U_\alpha\times U_\beta$ under the $\Delta T$-action are $\alpha, \beta$. Similarly, for the
 proof of Lemma \ref{connected}, observe that $U\times U$ is contractible under a $T\times  T$
(in particular, $\Delta T$)-equivariant contraction.)

  \begin{proposition} \label{prop4.9} For any coherent sheaf $\cs$ on $\bp$, and any $u,v\in W$,
    \[
\pi^*[\cs]\cdot [\widetilde{\xi^u\boxtimes\xi^v}] = [\pi^*(\cs )
\underset{\co_{\bar{Y}_\bp}}{\otimes} (\widetilde{\xi^u\boxtimes\xi^v})] \in
K^0(\bar{Y}_\bp),
  \]
where we abbreviate $\pi_{\bar{Y}}$ by $\pi$ and $\pi^* (\cs ) := \co_{\bar{Y}_\bp} \underset{\co_{\bp}}{\otimes} \cs$.

In particular,
\[
\pi^*[\co_{\bp_{\bf j}} (-\partial\bp_{\bf j})] \cdot
[\widetilde{\xi^u\boxtimes\xi^v}] =
\bigl[\pi^*\bigl( \co_{\bp_{\bf j}} (-\partial\bp_{\bf j})\bigr)
\otimes_{\co_{\bar{Y}_\bp}} (\widetilde{\xi^u\boxtimes\xi^v})\bigr].
  \]
  \end{proposition}

  \begin{proof}  By definition,
   \[
\pi^*[\cs]\cdot [\widetilde{\xi^u\boxtimes\xi^v}] = \sum_{i\geq 0} (-1)^i
\bigl[\tor_i^{\co_{\bar{Y}_\bp}} (\pi^* (\cs ), \widetilde{\xi^u\boxtimes\xi^v})
\bigr].
  \]
Thus, it suffices to prove that
  \[
\tor_i^{\co_{\bar{Y}_\bp}} (\pi^*\cs , \widetilde{\xi^u\boxtimes\xi^v}) =0, \;\forall i>0 .
  \]
Since the question is local in the base, we can assume
that $\bar{Y}_\bp\cong \bp \times \bar{Y}$.  Observe that, locally on the base,
  \begin{gather*}
\pi^*\cs \simeq \cs\boxtimes \co_{\bar{Y}}\quad\text{ and }\\
\widetilde{\xi^u\boxtimes\xi^v} = \co_{\bp}\boxtimes (\xi^u\boxtimes\xi^v),
  \end{gather*}
  where $\cs\boxtimes \co_{\bar{Y}}$ means $\cs \otimes_\bc \co_{\bar{Y}}$ etc.
Now, the result follows since, for algebras $R$ and $S$ over a field $k$ and $R$-module $M$, $S$-module $N$,
  \[
\Tor_i^{R\boxtimes S} (M\boxtimes S, R\boxtimes N) = 0,\quad\text{ for all }i>0.
  \]
  \end{proof}

The following is our main technical result. The proof of its two parts are given
in Sections \ref{sec7} and \ref{sec3} respectively.

  \begin{theorem}  \label{thma14}  For general $\gam\in\Gam = \Gam_{B\times B}$,
  any $u,v,w\in W$, and ${\bf j}\in [N]^r$,
  \[
\tor_i^{\co_{\bar{Y}_\bp}}\bigl( \pi^* (\co_{\bp_{\bf j}} (-\partial\bp_{\bf j}))\otimes
(\widetilde{\xi^u\boxtimes\xi^v}), \gam_* \tilde{\Del}_*\co_{(X_w)_{\bp}}\bigr) = 0,
 \,\,\text{ for all}\, i>0,\tag{$a$}\]
 where we view any element $\gamma \in \Gamma$ as an automorphism of the scheme
 $\bar{Y}_\bp$.
\vskip1ex

(b) Assume that $c^w_{u,v}({\bf j}) \neq 0$, where $c^w_{u,v}({\bf j}) \neq 0$ is defined by the identity \eqref{e0.0}. Then, 
\[H^p\bigl( \bar{Y}_\bp, \pi^* (\co_{\bp_{\bf j}}(-\partial\bp_{\bf j})) \otimes
(\widetilde{\xi^u\boxtimes\xi^v})\otimes \gam_* \tilde{\Del}_*
\co_{(X_w)_{\bp}}\bigr)
=0\]
 for all $p\neq |{\bf j}| +\ell (w)-(\ell (u)+\ell (v)),$ where $|{\bf j}|
 :=\sum^{r}_{i=1} j_i$.
  \end{theorem}

Since $\Gamma$ is connected, we get the following result as an immediate corollary
of Lemma \ref{lemma3.3}, Proposition \ref{prop4.9} and Theorem \ref{thma14}.

  \begin{corollary}\label{cor3.9}
  \[
(-1)^{\ell (w)-\ell (u)-\ell (v)+|{\bf j}|} c^w_{u,v}({\bf j})\in\bz_+.
  \]

  \end{corollary}

Recall the definition of the structure constants $p^w_{u,v}\in R(T)$
 for the product in $K^{\optop}_T(X)$ from
the beginning of this section. The following lemma follows easily from
Proposition \ref{n3.1}, identity  \eqref{e0.0} and
[GK, Lemma 6.2] (also see [AGM, $\S$3]).
\begin{lemma} \label{lem3.10}For any $u,v,w\in W$, we can choose
large enough $N$ (depending upon $u,v,w$) and express (by [GK, Proposition 2.2(c) and Theorem 5.1] valid in the Kac-Moody case as well)
\beqn \label{maineq1} p^w_{u,v}=\sum_{{\bf j}\in [N]^r}\, p^w_{u,v}({\bf j})
 (e^{-\alpha_1}-1)^{j_1}
\dots (e^{-\alpha_r}-1)^{j_r},\eeqn
for some unique $p^w_{u,v}({\bf j})\in \bz$,
where ${\bf j}= (j_1, \dots, j_r).$
Then,
\beqn \label{maineq2}
 p^w_{u,v}({\bf j}) =
(-1)^{|{\bf j}|} c^w_{u,v}({\bf j}).
\eeqn
\end{lemma}

As an immediate consequence of Corollary \ref{cor3.9} and Lemma \ref{lem3.10}, we
get the following main theorem of this paper, which was conjectured  by
Graham-Kumar [GK, Conjecture 3.1] in the finite case and proved in this case by Anderson-Griffeth-Miller [AGM, Corollary 5.2].
\begin{theorem} \label{verymain} For any $u,v,w\in W$, and any symmetrizable Kac-Moody group
$G$, the structure constants in $K^{\optop}_T(X)$ satisfy:
\beqn \label{eq3.11}
 (-1)^{\ell (u)+\ell (v)+ \ell (w)} \,p^w_{u,v}\in \bz_+[(e^{-\alpha_1}-1), \dots,
 (e^{-\alpha_r}-1)].
 \eeqn
 \qed
\end{theorem}
Recall that
\beqn \label{eq3.12}
K^{\optop}(X) \simeq \bz\otimes_{R(T)}\,  K^{\optop}_T(X)
\eeqn (cf.
[KK, Proposition 3.25]), where $\bz$ is considered as an $R(T)$-module via the
evaluation at $1$. Express the product in $K^{\optop}(X)$ in the `basis'
$\{\psi^u_o:= 1\otimes \psi^u\}_{u\in W}$:
\[
\psi^u_o\cdot\psi^v_o = \sum_w a^w_{u,v} \psi^w_o, \quad\text{for } a^w_{u,v}\in
\bz.
  \]
  Then, by the isomorphism \eqref{eq3.12},
  $$a^w_{u,v}=p^w_{u,v}(1).$$
Thus, from Theorem \ref{verymain}, we immediately obtain
the following result which was conjectured by A.S. Buch in the finite case and proved in this case by Brion [B].

\begin{corollary}  \label{maincor} For any $u,v,w\in W$,
$$(-1)^{\ell (u)+\ell (v)+ \ell (w)}\,a^w_{u,v} \in \bz_+.$$
\end{corollary}
\begin{remark}
 We conjecture\footnote{This conjecture has now been proved by Baldwin-Kumar [BaK].} that the analogue of 
Theorem \ref{verymain} is true  for the `basis' $\xi^u$ replaced
by the structure sheaf `basis' $\{\phi^u=[\co_{X^u}]\}_{u\in W}$ of $K_T^0(\bar{X})$.
 In the finite case,
this was conjectured by Griffeth-Ram [GR] and proved in this case by Anderson-Griffeth-Miller [AGM, Corollary 5.3].

For the affine Kac-Moody group $G=\widehat{\SL_N}$ associated to $\SL_N$, and its standard maximal parahoric subgroup 
$P$, let $\bar{\mathcal{X}}:= G/P$ be the corresponding infinite Grassmannian. Then, $K^0(\bar{\mathcal{X}})$
has the structure sheaf `basis' $\{[\co_{X^u}]\}_{u\in W/W_o}$ over $\bz$, where $W$ is the (affine) Weyl group of $G$ and $W_o=S_N$ is the Weyl group of $\SL_N$. Write, for any  $u,v\in W/W_o$,
$$[\co_{X^u}]\cdot [\co_{X^v}]= \sum_{w\in W/W_o}\,  b^w_{u,v}[\co_{X^w}], \,\,\,\text{for some unique integers}\,\, b^w_{u,v}.$$
Now, Lam-Schilling-Shimozono conjecture  [LSS, Conjectures 7.20 (2) and 7.21 (3)] the following:
$$(-1)^{\ell (u)+\ell (v)+ \ell (w)}\,b^w_{u,v} \in \bz_+,$$
if $u,v, w$ are the minimal coset representatives in their cosets.
\end{remark}

\section{Study of some $\ext$ and $\tor$ functors and proof of Theorem
 \ref{thma14} (a)}\label{sec7}

  \begin{proposition} \label{propa17} For any $j\in\bz$ and $u,w\in W$, as
  $T$-equivariant sheaves,
    \[
\tor_j^{\co_{\bar{X}}} (\xi^u, \co_{X_w}) \simeq e^{-\rho} \cl (-\rho )
\underset{\co_{\bar{X}}}{\otimes} \bigl( \ext^{\ell (u)-j}_{\co_{\bar{X}}} (\co_{X^u}, \co_{X_w})\bigr) .
  \]
In particular, $\ext^j_{\co_{\bar{X}}} (\co_{X^u}, \co_{X_w})=0$,\,\, for all $j>
\ell (u)$.
  \end{proposition}

  \begin{proof}  By definition,
  \[
\xi^u = e^{-\rho} \cl (-\rho )\, \ext^{\ell (u)}_{\co_{\bar{X}}} (\co_{X^u}, \co_{\bar{X}} ).
  \]

By Lemma \ref{2.4}, $\co_{X^u\cap V^v}$ admits a $T$-equivariant resolution (for any $v\in W$):
  \beqn \label{e4.1}
0 \to \cf_n \overset{\del_{n-1}}{\longrightarrow} \cdots \overset{\del_0}
{\longrightarrow} \cf_0 \rightarrow \co_{X^u\cap V^v} \to 0
  \eeqn
by $T$-equivariant free $\co_{V^v}$-modules of finite rank.

Since $\cm_j := \ext^j_{\co_{\bar{X}}} (\co_{X^u}, \co_{\bar{X}})=0$ for all
$j\neq \ell (u)$ (cf. the identity \eqref{eq1.0}), the dual complex
  \beqn\label{e4.2}
0 \leftarrow \cf^*_n \overset{\del^*_{n-1}}{\longleftarrow} \cf^*_{n-1} \leftarrow
 \cdots \leftarrow \cf^*_{\ell (u)} \leftarrow \cdots
 \overset{\del^*_0}{\longleftarrow} \cf^*_0 \leftarrow 0
  \eeqn
 gives rise to the resolution
  \[
0\from \cm_{\ell (u)} := \frac{\Ker \del^*_{\ell (u)}}{\Image \del^*_{\ell (u)-1}}
\from \Ker \del^*_{\ell (u)} \from \cf^*_{\ell (u)-1} \from
\cdots \from \cf_0^* \from 0,
  \]
where $\cf^*_i := \home_{\co_{V^v}} (\cf_i, \co_{V^v})$.

We next claim that $\Ker \del^*_j$ is a $\co_{V^v}$-module direct summand of $\cf^*_j$ for all $j\geq\ell (u)$:

We prove this by downward induction on $j$. Since \eqref{e4.2} has cohomology only
in degree $\ell (u)$, if $n>\ell (u)$, $\Image \del^*_{n-1} = \cf^*_n$ and
hence $\Ker \del^*_{n-1}$ is a direct summand $\co_{V^v}$-submodule of $\cf^*_{n-1}$.  Thus, $\Ker \del^*_{n-2}$ is a direct summand of $\cf^*_{n-2}$ if $n-2 \geq \ell (u)$.  Continuing this way, we see that $\Ker \del^*_{\ell (u)}$ is a direct summand $\co_{V^v}$-submodule of $\cf^*_{\ell (u)}$.

Thus, we get a projective resolution:
  \[
0 \to \cp_{\ell (u)} \to \cdots \to \cp_1 \to \cp_0 \to \cm_{\ell (u)} \to 0,
  \]
where $\cp_0 := \Ker \del^*_{\ell (u)}$, $\cp_i := \cf^*_{\ell (u)-i}$ for
$1\leq i\leq\ell (u)$.  Hence, restricted to the open subset $V^v$, $\tor_*^{\co_{\bar{X}}} (\xi^u, \co_{X_w})$ is the homology of the complex
  \[
0 \to \bigl( e^{-\rho}\cl (-\rho )\, \cp_{\ell (u)}\bigr)\underset{\co_{V^v}}
{\otimes} \co_{X_w} \to \cdots \to \bigl( e^{-\rho}\cl (-\rho )\, \cp_0\bigr)
\underset{\co_{V^v}}{\otimes} \co_{X_w} \to 0.
  \]
Now, we show that the $j$-th homology of the complex
  \[
\cc: \qquad\qquad 0\to\cp_{\ell (u)}\underset{\co_{V^v}}{\otimes}\co_{X_w}
 \to \cdots \to \cp_0\underset{\co_{V^v}}{\otimes}\co_{X_w} \to 0
  \]
is isomorphic with $\ext^{\ell (u)-j}_{\co_{\bar{X}}} (\co_{X^u}, \co_{X_w})$:

 Since
  \[
\cp_i \underset{\co_{V^v}}{\otimes}\co_{X_w} \simeq \home_{\co_{V^v}}
(\cf_{\ell (u)-i}, \co_{X_w}),\,\,\,\text{for all} \,\, i\geq 1 ,
  \]
we get
  \beqn \label{e4.3}
\mathscr{H}_j(\cc ) \simeq \ext^{\ell (u)-j}_{\co_{V^v}} (\co_{X^u}, \co_{X_w}), \quad\text{ for all }j\geq 2.
  \eeqn
Moreover, since $\cp_0$ is a direct summand of $\cf^*_{\ell (u)}$, we get
  \beqn \label{e4.4}
\mathscr{H}_1(\cc ) \simeq \ext^{\ell (u)-1}_{\co_{V^v}} (\co_{X^u}, \co_{X_w}).
  \eeqn
Now,
  \beqn \label{e4.5}
\mathscr{H}_0(\cc ) = \frac{\cp_0\otimes_{\co_{V^v}}\co_{X_w}}
{\Image (\cp_1\otimes_{\co_{V^v}}\co_{X_w})} \simeq \ext^{\ell (u)}_{\co_{V^v}} (\co_{X^u}, \co_{X_w}),
  \eeqn
since $\Ker \del^*_{\ell (u)}$ is a direct summand of $\cf^*_{\ell (u)}$ and $\Ker \del^*_{\ell (u)+1} =$ $\Image \del^*_{\ell (u)}$ is a direct summand of $\cf^*_{\ell (u)+1}$.

Finally,
  \beqn \label{e4.6}
\ext^j_{\co_{V^v}} (\co_{X^u}, \co_{X_w}) =0, \quad\text{ for all }j>\ell (u).
  \eeqn
To prove this, observe that, for $j>\ell (u)$,
  \[
0\to\Ker \del^*_j \to \cf^*_j \overset{\del^*_j}{\longrightarrow} \Image \del^*_j = \Ker \del^*_{j+1} \to 0
  \]
is a split exact sequence since $\Ker \del^*_{j+1}$ is projective.  Thus,
  \[
0 \to \Ker \del^*_j \underset{\co_{V^v}}{\otimes} \co_{X_w} \to \cf^*_j
\underset{\co_{V^v}}{\otimes}\co_{X_w} \to (\Image \del^*_j) \underset{\co_{V^v}}{\otimes}\co_{X_w} \to 0
  \]
is exact.  Moreover, $\Image  \del^*_j \hookrightarrow \cf^*_{j+1}$ is a direct summand and hence
  \[
\Image \del^*_j \underset{\co_{V^v}}{\otimes} \co_{X_w} \hookrightarrow \cf^*_{j+1}
\underset{\co_{V^v}}{\otimes} \co_{X_w}.
  \]
From this \eqref{e4.6}  follows.

Combining \eqref{e4.3} -- \eqref{e4.6}, we get the proposition.
  \end{proof}

The following is a minor generalization of the `acyclicity lemma' of Peskine-Szpiro
[PS, Lemme 1.8].

  \begin{lemma} \label{4.4} Let $R$ be a local noetherian CM domain and let
   \beqn
0 \to F_n \to F_{n-1} \to \cdots \to F_0 \to 0 \tag{$*$}
  \eeqn
be a complex of finitely-generated free $R$-modules.  Fix a positive integer
$d>0$. Assume:

 (a) some irreducible component $Z$ of the support
 of $M := \oplus_{i\geq 1} H_i(F_*)$ has codimension $\geq d$ in Spec $R$, and
  \beqn
F_i =0 , \quad\text{ for all } i>d. \tag{b}
  \eeqn
Then,
  \[   H_i(F_*) = 0 , \quad\text{ for all }i>0.   \]
  \end{lemma}

  \begin{proof}  Assume, if possible, that $M\neq 0$.  Let $I\subset R$ be the
  annihilator of $M$ and let $\fp$ be the (minimal) prime ideal containing $I$
  corresponding to $Z$. Then,
    \beqn \label{e4.13}
M\otimes_R R_{\fp} \neq 0,
\eeqn
and

    \beqn \label{e4.14}
\depth \bigl( M\otimes_R R_{\fp}) =0.
  \eeqn
Next observe that
  \begin{align*}
\depth (F_*\otimes_R R_{\fp}) &= \depth R_{\fp}\\
&:= \depth_\fp\, R_{\fp} \\
&= \codim (\fp R_{\fp}), \;\text{ since $R_{\fp}$ is CM}\\
&= \codim (\fp )\\
&\geq d.
  \end{align*}

Now, by applying the acyclicity lemma of Peskine-Szpiro [PS, Lemme 1.8] to the
complex $F_*\otimes_R R_{\fp}$ and using the identities \eqref{e4.13},\eqref{e4.14},
we get a contradiction.

Thus, $M=0$, proving the lemma.
    \end{proof}

\begin{corollary}  \label{4.3} Let $Y$ be an irreducible CM variety
   and $d>0$ a positive integer. Let
  \[
0 \from \cg^n \overset{\delta^{n-1}}{\longleftarrow} \cg^{n-1} \from \cdots
\overset{\delta^0}{\longleftarrow} \cg^0 \from 0
  \]
be a complex of locally free $\co_Y$-modules of finite rank satisfying the following:

\vskip1ex

1) The support of the sheaf $\oplus_{i<d} \mathscr{H}^i (\cg^*)$ has an irreducible
component of codimension $\geq d$ in $Y$.

2) The sheaf $\mathscr{H}^j (\cg^*)=0$, for all $j>d$.

Then, $\mathscr{H}^j(\cg^*)=0$ for all
$j<d$ as well.
   \end{corollary}

\begin{proof}  We first claim by downward induction that $\Ker \delta^j$ is a
   direct summand of $\cg^j$, for any $j\geq d$.  The proof is similar to that
   given in the proof of Proposition  \ref{propa17}.  Thus,
   $$H^*(\cg^*) \simeq H^*(\cf^*),$$ where
    \[
\cf^i = \cg^i \,\, \text{for all } i<d , \,\,
\cf^d = \Ker \delta^d, \,\,\text{and}\, \,
\cf^i = 0 \,\,\,\text{for   }\, i>d.
  \]
Thus,  we can assume that $\cg^i=0$, for all $i>d$.  Now, we apply the last lemma
to get the result.
   \end{proof}

  \begin{proposition} \label{propa18}  For any $u,w\in W$,
    \[
\ext^j_{\co_{\bar{X}}} (\co_{X^u}, \co_{X_w}) =0, \quad\text{ for all } j<\ell (u).
    \]
    Thus,
    \[\tor_j^{\co_{\bar{X}}} (\xi^u, \co_{X_w})
    =0,\,\,\,\text{ for \,all\,}  j> 0.\]
  \end{proposition}

  \begin{proof}  We can, of course, replace $\bar{X}$ by $V^v$ (for $v\in W$).  Consider a locally $\co_{V^v}$-free resolution of finite rank:
    \[
0 \to \cf_n \to \cf_{n-1} \to \cdots \to \cf_0 \to \co_{X^u\cap V^v} \to 0.
  \]
Then, restricted to the open set $V^v$, $\ext^j_{\co_{\bar{X}}}(\co_{X^u}, \co_{X_w})$ is the $j$-th cohomology of the complex
  \[
0 \from \home_{\co_{\bar{X}}} (\cf_n, \co_{X_w}) \from \cdots \from \home_{\co_{\bar{X}}} (\cf_0, \co_{X_w}) \from 0.
  \]
Since $\cf_j$ is $\co_{\bar{X}}$-free,
  \[
\home_{\co_{\bar{X}}} (\cf_j, \co_{X_w}) \simeq \home_{\co_{X_w}} (\cf_j\underset{\co_{\bar{X}}}{\otimes}\co_{X_w}, \co_{X_w}).
  \]

Now, the first part of the proposition follows from Corollary
\ref{4.3} applied to $d=\ell(u)$  and Proposition \ref{propa17},
 by observing that the sheaf $\ext^j_{\co_{\bar{X}}} (\co_{X^u}, \co_{X_w})$
 has support in $X^u\cap X_w$, $X_w$ is an irreducible  CM variety
 (cf. [K, Theorem 8.2.2(c)]), and for
 $u\leq w$, $\codim_{X_w}(X^u\cap X_w) = \ell (u)$ (see [K, Lemma 7.3.10]).

 The second assertion of the proposition follows from the first part and Proposition
 \ref{propa17}.
    \end{proof}
As a consequence of Proposition \ref{propa18}, we prove Theorem \ref{thma14} (a).

\vskip1ex
  \noindent
  {\it Proof of Theorem \ref{thma14} (a).}
   Since the assertion is local in $\bp$, we can assume that $\bar{Y}_\bp
   \simeq \bp\times \bar{Y}$.  Thus,
  \beqn \label{e4.7}
\pi^*\, \co_{\bp_{\bf j}}(-\partial\bp_{\bf j}) \simeq \co_{\bp_{\bf j}}
(-\partial\bp_{\bf j})\boxtimes \co_{\bar{Y}}
  \eeqn
  \beqn \label{e4.8}
\widetilde{\xi^u\boxtimes\xi^v} \cong \co_{\bp}\boxtimes (\xi^u\boxtimes\xi^v)
  \eeqn
  \beqn \label{e4.9}
\co_{{(X_w\times X_w)}_\bp} \simeq \co_{\bp}\boxtimes \bigl( \co_{X_w}\boxtimes\co_{X_w}\bigr) .
  \eeqn
We assert that for any $\co_{(Y_w)_{\bp}}$-module $\cs$ (where $(Y_w)_{\bp}:=
 {(X_w \times X_w)}_\bp$)
  \beqn \label{e4.11}
\tor_i^{\co_{\bar{Y}_\bp}}\Bigl( \pi^* \bigl(\co_{\bp_{\bf j}}
(-\partial\bp_{\bf j})\bigr) \otimes
(\widetilde{\xi^u\boxtimes\xi^v}), \cs\bigr)
\simeq \tor_i^{\co_{(Y_w)_{\bp}}} \Bigl( \co_{(Y_w)_{\bp}}
\underset{\co_{\bar{Y}_\bp}}{\otimes}
\bigl(\pi^* \co_{\bp_{\bf j}}(-\partial\bp_{\bf j}) \otimes
(\widetilde{\xi^u\boxtimes\xi^v})\bigr) , \cs\Bigr).
   \eeqn
To prove \eqref{e4.11}, from Proposition \ref{propa18} and the isomorphisms
\eqref{e4.7} --\eqref{e4.9},
 it suffices to observe the following (where we take $R=\co_{\bar{Y}_\bp}, S={\co_{(Y_w)_{\bp}}}, 
M= \pi^* \bigl(\co_{\bp_{\bf j}}
(-\partial\bp_{\bf j})\bigr) \otimes
(\widetilde{\xi^u\boxtimes\xi^v})$ and  $N= \cs$).

 Let $R, S$ be commutative rings with ring homomorphism $R\to S$, $M$
 an $R$-module and $N$ an $S$-module, then 
$N \otimes_S (S\otimes_R M) \simeq N\otimes_R M.$ This gives rise to the following  isomorphism provided
$ \Tor_j^R(S,M) =0,\, \forall j>0$.
  \beqn \label{neweqn6} \Tor_i^R(M,N)\simeq \Tor_i^S(S\otimes_R M,N).
\eeqn

 Clearly,
 \begin{align*}
\tor_i^{\co_{{(Y_w)_{\bp}}}}&\Bigl(\co_{{(Y_w)_{\bp}}}\underset{\co_{\bar{Y}_\bp}}{\otimes}
\bigl(\pi^*\co_{\bp_{\bf j}} (-\partial\bp_{\bf j}) \otimes
(\widetilde{\xi^u\boxtimes\xi^v})\bigl) , \gam_* \tilde{\Del}_*\co_{(X_w)_{\bp}}\Bigr)\\
& \simeq \tor_i^{\co_{{(Y_w)_{\bp}}}}\Bigl((\gamma^{-1})_*\bigl(\co_{{(Y_w)_{\bp}}}
\underset{\co_{\bar{Y}_\bp}}{\otimes}
\bigl(\pi^*\co_{\bp_{\bf j}} (-\partial\bp_{\bf j}) \otimes
(\widetilde{\xi^u\boxtimes\xi^v})\bigr)\bigr) , \tilde{\Del}_*\co_{(X_w)_{\bp}}\Bigr).
\end{align*}

By Lemma \ref{lem6}, the closures of $\Gamma$-orbits in $(Y_w)_{\bp}$ are precisely
$(X_x\times X_y)_\bp$, for
$x,y \leq w$. By Proposition \ref{propa18} and the isomorphism \eqref{neweqn6} (applied to 
 $R=\co_{\bar{X}}, S=\co_{X_w}, 
M= \xi^u$
 and  $N= \co_{X_x}$), we get 
\beqn \label{neweqn7} \tor_j^{\co_{X_w}} (\co_{X_w}\otimes_{\co_{\bar{X}}}\,\xi^u, \co_{X_x}) = 0, \,\, \forall x\leq w, j\geq 1.
\eeqn
 Further, by the identities \eqref{e4.7} - \eqref{e4.9} and \eqref{neweqn7},
 $\mathcal{F}:= \co_{{(Y_w)_{\bp}}}\underset{\co_{\bar{Y}_\bp}}{\otimes}
\bigl(\pi^*\co_{\bp_{\bf j}} (-\partial\bp_{\bf j}) \otimes
(\widetilde{\xi^u\boxtimes\xi^v})\bigl)$ is homologically transverse to the
 $\Gamma$-orbit closures in $(Y_w)_{\bp}$. Thus,
applying [AGM, Theorem 2.3] (with
their $G=\Gamma$, $X = (Y_w)_{\bp}$, $\mathcal{E}= \tilde{\Del}_*\co_{(X_w)_{\bp}}$,
and their
$\mathcal{F}$ as the above $\mathcal{F}$) (a result originally due to Sierra [Si, Theorem 1.2])  we get the following identity:
 \beqn\label{e4.12}
\tor_i^{\co_{{(Y_w)_{\bp}}}}\Bigl(\co_{{(Y_w)_{\bp}}}\underset{\co_{\bar{Y}_\bp}}{\otimes}
\bigl(\pi^*\co_{\bp_{\bf j}} (-\partial\bp_{\bf j}) \otimes
(\widetilde{\xi^u\boxtimes\xi^v})\bigl) , \gam_* \tilde{\Del}_*\co_{(X_w)_{\bp}}\Bigr)
=0,
\,\,\text{for all}\,  i>0.
  \eeqn
(Observe that
 even though $\Gamma$ is infinite dimensional,
its action on ${(Y_w)_{\bp}}$ factors through the action of a finite dimensional quotient group
$\bar{\Gamma}$  of $\Gamma$.) 

Observe that $\gam (\tilde{\Delta}(X_w)_{\bp}) \subset (Y_w)_{\bp}$ and thus
by \eqref{e4.11}-\eqref{e4.12}, we get
  \[
\tor_i^{\co_{\bar{Y}_\bp}}\bigl( \pi^* (\co_{\bp_{\bf j}} (-\partial\bp_{\bf j})) \otimes
(\widetilde{\xi^u\boxtimes\xi^v}),
 \gam_*\tilde{\Del}_* \co_{(X_w)_{\bp}}\bigr) =0, \,\,\text{for all}\,  i>0.
  \]
  This proves Theorem \ref{thma14} (a).
  \qed

   \begin{lemma} \label{4.5} For any $u,w\in W$,
  \[
\tor_j^{\co_{\bar{X}}} (\co_{X^u}, \co_{X_w}) = 0, \qquad\text{for all $j>0$.}
  \]
  \end{lemma}

  \begin{proof}  We can of course replace $\bar{X}$ by the open set $V^v$
  (for  $v\in W$) and consider the free resolution by $\co_{V^v}$-modules of
  finite rank:
    \[
0 \to \cf_n \overset{\del_{n-1}}{\longrightarrow} \cf_{n-1} \to \cdots
\overset{\del_0}{\longrightarrow} \cf_0 \to \co_{X^u\cap V^v} \to 0.
  \]
By downward induction, we show that $\cd_i := \OpIm \del_i$ is a direct
summand of $\cf_i$, for all $i\geq \ell (u)$. Of course, the assertion holds
for $i=n$.  By induction, assume that $\cd_{i+1}$ is a direct summand (where
 $i\geq \ell (u)$).  Thus,
  \beqn
\tag{$\cc_1$} 0\to \cd_{i+1}^{\bot} \overset{\del_i}{\longrightarrow}
\cf_i \to \cdots \to \cf_0 \to \co_{X^u\cap V^v} \to 0
  \eeqn
is a free resolution, where $\cd^{\bot}_{i+1}$ is any $\co_{V^v}$-submodule of
$\cf_{i+1}$ such that $\cd_{i+1} \oplus \cd^{\bot}_{i+1} = \cf_{i+1}$.

Consider the short exact sequence:
  \beqn
\tag{$\cc_2$} 0\to \cd_{i+1}^{\bot} \overset{\del_i}{\longrightarrow}
\cf_i \to \cf_i/ \del_i\bigl(\cd_{i+1}^{\bot}) \to 0.
  \eeqn

This gives rise to the exact sequence:
  \begin{align*}
0 \to &\home_{\co_{V^v}}\bigl(\cf_i/\del_i (\cd^{\bot}_{i+1}), \co_{V^v}\bigr)
\to \home_{\co_{V^v}}\bigl(\cf_i, \co_{V^v}\bigr) \overset{\del^*_i}{\longrightarrow}\\
&\home_{\co_{V^v}}\bigl(\cd^{\bot}_{i+1}, \co_{V^v}\bigr) \to \ext^1_{\co_{V^v}}
\bigl(\cf_i/\del_i (\cd^{\bot}_{i+1}), \co_{V^v}\bigr) \to 0,
  \end{align*}
where the last zero is due to the fact that $\cf_i$ is $\co_{V^v}$-free.

From the resolution ($\cc_1$) and the identity \eqref{eq1.0} (since $i\geq \ell (u)$ by
assumption), we see that the above map $\del^*_i$ is surjective.  Hence,
  \[
\ext^1_{\co_{V^v}} \bigl(\cf_i/\del_i (\cd^{\bot}_{i+1}), \co_{V^v}\bigr) =0
  \]
and hence
 \[
\ext^1_{\co_{V^v}} \bigl(\cf_i/\del_i (\cd^{\bot}_{i+1}), \cd^{\bot}_{i+1}\bigr) =0,
  \]
since $\cd^{\bot}_{i+1}$ is a free $\co_{V^v}$-module.

Thus, the short exact sequence ($\cc_2$) splits.  In particular, $\cd_i =
\OpIm \del_i$ is a direct summand.  This completes the induction and hence we
 get a free resolution:
  \beqn
\tag{$\cc_3$}  0 \to \cd^{\bot}_{\ell (u)} \to \cf_{\ell (u)-1} \to \cdots \to
\cf_0 \to \co_{X^u\cap V^v} \to 0.
  \eeqn

In particular,
  \[
\tor^{\co_{\bar{X}}}_j (\co_{X^u}, \co_{X_w})=0, \quad\text{for all } j> \ell (u).
  \]
Of course, $\tor^{\co_{\bar{X}}}_j (\co_{X^u}, \co_{X_w})$, restricted to $V^v$, is the
$j$-th homology of the chain complex (which is a complex of finitely generated free $\co_{X_w\cap V^v}$-modules)
  \beqn
\tag{$\cc_4$}  0 \to \cd^{\bot}_{\ell (u)} \otimes_{\co_{V^v}} \co_{X_w\cap V^v}
\to \cf_{\ell (u)-1} \otimes_{\co_{V^v}} \co_{X_w\cap V^v} \to \cdots \to
\cf_0 \otimes_{\co_{V^v}} \co_{X_w\cap V^v} \to 0.
  \eeqn
Clearly, the support of the homology $\oplus_{i\geq 1} \mathscr{H}_i (\cc_4)$ is
contained
in $X^u\cap X_w$.  As observed in the proof of Proposition \ref{propa18},
$X^u\cap X_w$ is of codimension $\ell (u)$
in $X_w$.

Thus, by Lemma \ref{4.4} with $d=\ell(u)$,
  \[  \mathscr{H}_i (\cc_4) =0, \qquad\text{for all } i>0.   \]
This proves the lemma.
  \end{proof}

\begin{remark} As pointed out by the referee, the above lemma can also be deduced from Proposition \ref{propa18} by using Theorem 
\ref{prop:main} and the long exact sequence for $\tor$.
\end{remark}
As a consequence of the above Lemma \ref{4.5}, we get the following generalization.
\begin{corollary} \label{newcor5.5} For any finite union $Y=\cup_{i=1}^k\, X^{v_i}$ of opposite Schubert varieties, and any $w\in W$,

\vskip1ex

(a)  $
\tor_j^{\co_{\bar{X}}} (\co_{Y}, \co_{X_w}) = 0,$ \,\,for all $j>0$.
  
(b) $H^j(X_n, \co_{Y \cap X_w}) =0,$ \,\, for all $j>0$,
\noindent
where $n$ is any positive integer such that $X_n\supset X_w$.

\vskip1ex

In particular, the lemma applies to $Y=\partial X^u$. 
\end{corollary}
\begin{proof} (a): We prove (a) by double induction on the number of components $k$ of $Y$ and the dimension of 
$Y \cap X_w$ (i.e., the largest dimension of the irreducible components of $Y \cap X_w$; we declare the dimension of 
the empty space to be $-1$). If $Y$ has one component, i.e., $k=1$, then (a) follows
from Lemma \ref{4.5}. If dim $(Y \cap X_w)= -1$ (i.e., $Y \cap X_w$ is empty), then clearly 
\beqn \label{newequation8} 
\tor_j^{\co_{\bar{X}}} (\co_{Y}, \co_{X_w}) = 0, \,\, \text{for all $j\geq 0$}.
\eeqn
So, assume that $k\geq 2$ and $Y \cap X_w$ is nonempty. We can assume that $v_1$ is not larger than any $v_i$, 
for $i\geq 2$ (for otherwise we can drop $X^{v_1}$ from the union without changing $Y$). Let $Y_1 :=X^{v_1}$ and 
$Y_2 :=\cup_{i\geq 2}\, X^{v_i}$. Then, if $Y_1 \cap X_w$ is nonempty, $Y_1 \cap X_w = X^{v_1}_w$ properly contains $ Y_1 \cap Y_2\cap X_w$, since 
$v_1 \in X^{v_1}_w$ but $v_1 \notin Y_2\cap X_w$. In particular, $X^{v_1}_w$ being irreducible, 
\beqn \label{newequation9} 
\dim (Y\cap X_w) \geq \dim (Y_1\cap X_w) > \dim (Y_1\cap Y_2\cap X_w).
\eeqn
The short exact sequence of sheaves:
\[\co_Y \to \co_{Y_1}\oplus \co_{Y_2} \to \co_{Y_1\cap Y_2} \to 0\]
yields the long exact sequence
\begin{align} \label{newequation10} 
\cdots \to \tor_{j+1}^{\co_{\bar{X}}}\left(\co_{Y_1\cap Y_2} , \co_{X_w}\right) \to &\tor_{j}^{\co_{\bar{X}}}\left(\co_{Y} , 
\co_{X_w}\right)\to \tor_{j}^{\co_{\bar{X}}}\left(\co_{Y_1}\oplus \co_{Y_2} , \co_{X_w}\right) \to \\
&\tor_{j}^{\co_{\bar{X}}}\left(\co_{Y_1\cap Y_2} , \co_{X_w}\right) \to \cdots . \notag
\end{align}
Now, since $Y_2$ has $k-1$ components, induction on the number of components gives
\beqn\label{neweqn11}
\tor_{j}^{\co_{\bar{X}}}\left(\co_{Y_1}\oplus \co_{Y_2} , \co_{X_w}\right) =0, \qquad\text{for all } j>0. 
\eeqn
Since the scheme theoretic intersection $Y_1\cap Y_2$ is reduced (cf. \S2) and it is a finite union of $X^u$ s with 
$\dim (Y\cap X_w) > \dim (Y_1\cap Y_2\cap X_w)$ (by equation \eqref{newequation9} ), by induction we get
\beqn\label{neweqn12}
\tor_{j}^{\co_{\bar{X}}}\left(\co_{Y_1 \cap Y_2} , \co_{X_w}\right) =0, \qquad\text{for all } j>0. 
\eeqn
So, from the equations \eqref{neweqn11} - \eqref{neweqn12}  and the exact sequence \eqref{newequation10}, we get (a).

\vskip1ex

(b) We use the same induction as in (a). For $k=1$, i.e.,  $Y\cap X_w= X_w^{v_1}$, the result is a particular case of 
[KuS, Corollary 3.2]. Now, take any $Y=\cup_{i=1}^k\, X^{v_i}$ and let $Y_1, Y_2$ be as in the (a)-part. By the (a)-part, we have the sheaf exact sequence:

\[
\xymatrix{
0 \ar[r]& \co_Y\otimes_{\co_{\bar{X}}}\,\co_{X_w}\ar[d]^{\wr}\ar[r] &   \left(\co_{Y_1}\oplus \co_{Y_2}\right) \otimes_{\co_{\bar{X}}}\,\co_{X_w}  
 \ar[d]^{\wr}\ar[r] &    \co_{Y_1\cap Y_2}\otimes_{\co_{\bar{X}}}\,\co_{X_w}\ar[d]^{\wr}\ar[r] & 0\\
0 \ar[r]& \co_{Y\cap X_w}\ar[r] &   \left(\co_{Y_1\cap X_w}\oplus \co_{Y_2\cap X_w}\right) \ar[r] &    \co_{Y_1\cap Y_2\cap X_w} \ar[r] & 0.
}
\]
The corresponding long exact cohomology sequence gives:
\begin{align*} \dots \to H^{j-1}\left(X_n, \co_{Y_1 \cap X_w}\oplus\co_{Y_2 \cap X_w} \right) \to & H^{j-1}\left(X_n, \co_{Y_1 \cap Y_2\cap X_w}\right) \to 
H^{j}\left(X_n, \co_{Y \cap X_w} \right) \to \\
&H^{j}\left(X_n, \co_{Y_1 \cap X_w}\oplus\co_{Y_2 \cap X_w} \right) \to \dots .
\end{align*}
By induction, 
\[H^{j}\left(X_n, \co_{Y_1 \cap X_w} \oplus \co_{Y_2 \cap X_w} \right) =0,\,\,\forall j>0, \,\,\,\text{and}\,\,  H^{j-1}\left(X_n, \co_{Y_1 \cap Y_2\cap X_w} \right) =0,\,\,\forall j>1.
\]
Thus, from the above long exact sequence,
\[H^{j}\left(X_n, \co_{Y\cap X_w}\right)=0,\,\,\forall j>1.
\]
Write $Y_1\cap Y_2= \cup_{l=1}^d\, X^{u_l}$. Hence, $Y_1\cap Y_2\cap X_w= \cup_{l=1}^d\, X^{u_l}_w$. Thus, if nonempty, 
$Y_1\cap Y_2\cap X_w$ is connected as each of  $X^{u_j}_w$ contains $w$. This gives that 
\[H^{0}\left(X_n, \co_{Y_1 \cap X_w} \oplus \co_{Y_2 \cap X_w} \right) \to H^{0}\left(X_n, \co_{Y_1 \cap Y_2\cap X_w}\right) 
\]
is surjective, which gives the vanishing of $H^{1}\left(X_n, \co_{Y \cap X_w}\right).$ This proves the (b)-part.
\end{proof}

As a consequence of the above Lemma \ref{4.5}, we get the following.

  \begin{lemma}  \label{4.6} For any $u,w\in W$ and any $j\geq 0$,
   \beqn  \label{eq4.6.1}
\ext^j_{\co_{X_w}} \bigl( \co_{X^u\cap X_w}, \co_{X_w}\bigr) =0, \qquad\text{for }
 j\neq \ell (u).
  \eeqn
Moreover,
 \beqn \label{e4.6.2}
\ext^j_{\co_{\bar{X}}} (\co_{X^u}, \co_{\bar{X}}) \otimes_{\co_{\bar{X}}}
 \co_{X_w} \simeq \ext^j_{\co_{X_w}} \bigl( \co_{X^u\cap X_w}, \co_{X_w}\bigr) .
  \eeqn
  \end{lemma}.

  \begin{proof}  Again we can replace $\bar{X}$ by $V^v$ (for $v\in W$).
  Consider
a $\co_{V^v}$-free resolution (cf. the proof of Lemma \ref{4.5}, specifically $\cc_3$):
  \[
0 \to \cf_{\ell (u)} \to \cdots \to \cf_0 \to \co_{X^u\cap V^v} \to 0.
  \]

By Lemma \ref{4.5}, the following is a free $\co_{X_w\cap V^v}$-module resolution:
  \beqn \label{neweqn15}
0 \to \cf_{\ell (u)} \otimes_{\co_{V^v}} \co_{X_w\cap V^v} \to \cdots \to
\cf_0\otimes_{\co_{V^v}} \co_{X_w\cap V^v} \to
 \co_{X^u\cap V^v} \otimes_{\co_{V^v}}\co_{X_w\cap V^v} \to 0.
   \eeqn

Observe that $\co_{X^u\cap V^v} \otimes_{\co_{V^v}}\co_{X_w\cap V^v}$
$\simeq \co_{X^u\cap X_w\cap V^v}$, being the definition of the scheme theoretic intersection.  Thus, $\ext^j_{\co_{X_w}}
\bigl( \co_{X^u\cap X_w}, \co_{X_w}\bigr)$, restricted to the open set
$X_w\cap V^v$, is the $j$-th cohomology of the cochain complex:
  \[
0 \leftarrow \home_{\co_{X_w\cap V^v}} \Bigl( \cf_{\ell (u)} \otimes_{\co_{V^v}}
\co_{X_w\cap V^v}, \co_{X_w\cap V^v}\Bigr) \leftarrow \cdots
\leftarrow \home_{\co_{X_w\cap V^v}} \Bigl( \cf_0 \otimes_{\co_{V^v}}
\co_{X_w\cap V^v}, \co_{X_w\cap V^v}\Bigr) \leftarrow 0.
  \]
Since $\ext^{j}_{\co_{X_w}} \big( \co_{X^u\cap X_w}, \co_{X_w}\bigr)$ has support in $X^u\cap X_w$ and $X^u\cap X_w$ has codimension 
$\ell(u)$ in $X_w$ (see the proof of Proposition \ref{propa18},  by Lemma \ref{4.4}, we get $\ext^j_{\co_{X_w}} \big( \co_{X^u\cap X_w},
\co_{X_w}\bigr) =0,$ for any
 $j\neq \ell (u)$.
This proves \eqref{eq4.6.1}.

For any $i$,
  \beqn \label{e4.6.3}
\home_{\co_{X_w\cap V^v}} \Bigl( \cf_i \otimes_{\co_{V^v}} \co_{X_w\cap V^v},
\co_{X_w\cap V^v}\Bigr) \simeq \home_{\co_{V^v}} ( \cf_i, \co_{V^v})
 \otimes_{\co_{V^v}} \co_{X_w\cap V^v}.
  \eeqn
Further, by the identity \eqref{eq1.0},
  \[
0 \leftarrow \ext^{\ell (u)}_{\co_{V^v}} \bigl(\co_{X^u\cap V^v}, \co_{V^v}\bigr)
 \leftarrow \home_{\co_{V^v}} (\cf_{\ell (u)}, \co_{V^v}) \leftarrow \cdots
\leftarrow \home_{\co_{V^v}} ( \cf_0 ,  \co_{V^v}) \leftarrow 0
  \]
is a free $\co_{V^v}$-module resolution of $\ext^{\ell (u)}_{\co_{V^v}}
(\co_{X^u\cap V^v}, \co_{V^v})$.  Hence, by the resolution \eqref{neweqn15} and the isomorphism \eqref{e4.6.3}, we get
  \beqn \label{neweqn16}
\ext^{\ell(u)-j}_{\co_{X_w}} \big( \co_{X^u\cap X_w},
\co_{X_w}\bigr)\simeq \tor_j^{\co_{\bar{X}}} \bigl( \ext^{\ell (u)}_{\co_{\bar{X}}} (\co_{X^u},
\co_{\bar{X}}), \co_{X_w}\bigr), \qquad\text{for all } j\geq 0.
  \eeqn
Thus, by the isomorphism \eqref{neweqn16},
  \[
\ext^{\ell (u)}_{\co_{X_w}} \bigl(\co_{X^u\cap X_w}, \co_{X_w}\bigr) \simeq
\ext^{\ell (u)}_{\co_{\bar{X}}} (\co_{X^u}, \co_{\bar{X}}) \otimes_{\co_{\bar{X}}}
 \co_{X_w} .
  \]

  This proves \eqref{e4.6.2}, by using the identity \eqref{eq1.0} and
  \eqref{eq4.6.1}.
   \end{proof}

\begin{lemma}  For any $v\leq w$ and $u\in W$,
  \[
\tor^{\co_{X_w}}_i \bigl(\co_{X^u\cap X_w}, \co_{X_v}\bigr) =0, \quad\text{for all }
i>0.
  \]
  \end{lemma}

    \begin{proof}  We can replace $\bar{X}$ by $V^\theta$ (for $\theta\in W$).
    Take a $\co_{\bar{X}}$-free resolution (see ($\cc_3$) of Lemma \ref{4.5}):
  \[
0 \to \cf_{\ell (u)} \to \cdots \to \cf_1 \to \cf_0 \to \co_{X^u} \to 0.
  \]
By Lemma \ref{4.5},
  \beqn
0 \to \cf_{\ell (u)} \otimes_{\co_{\bar{X}}} \co_{X_w} \to \cdots \to \cf_0
 \otimes_{\co_{\bar{X}}} \co_{X_w} \to \co_{X^u\cap X_w} \to 0 \tag{$\cs_1$}
  \eeqn
is a $\co_{X_w}$-free resolution of $\co_{X^u\cap X_w}$.  Thus,
by the base extension (cf. [L, \S3, Chap. XVI]),  $\tor_i^{\co_{X_w}}
\bigl(\co_{X^u\cap X_w}, \co_{X_v}\bigr)$ is the $i$-th homology of the complex:
  \[
0 \to \cf_{\ell (u)} \otimes_{\co_{\bar{X}}} \co_{X_v} \to \cdots \to \cf_0
\otimes_{\co_{\bar{X}}} \co_{X_v} \to 0.
  \]

From the exactness of ($\cs_1$) for $w$ replaced by $v$, we get the lemma.
  \end{proof}

\section{Desingularization of Richardson varieties and flatness for the
$\Gamma$-action}\label{sec1}

Let $S\subset W$ be a finite ideal and, as in Definition \ref{n2.1}, let $V^{s}$
be the corresponding
$B^{-}$-stable open subset $\cup_{w\in S}\,( w\ B^-\cdot x_{o})$ of
$\bar{X}$, where $x_o$ is the base point $1.B$ of $\bar{X}$.
It is a $B^{-}$-stable subset, since by [KS, \S2],
$$
V^{S}=\bigcup\limits_{w\in S}B^{-} wx_{o}.
$$
\begin{lemma}\label{basic} For any $v\in W$ and any finite ideal $S\subset W$
containing $v$,
there exists a closed normal subgroup
${N}^-_S$ of $B^-$ of finite codimension such that the quotient
 $ Y^v(S):={N}^-_S\setminus X^{v}(S)$ acquires a canonical structure of a
 $B^-$-scheme of finite type over the base field $\bc$ under the left multiplication action of
 $B^-$ on $Y^v(S)$, so that the quotient  $q: X^{v}(S) \to
 Y^{v}(S)$ is a principal ${N}^-_S$-bundle, where $X^{v}(S):=X^v\cap V^S$.

 Of course, the map $q$ is $B^-$-equivariant.
 \end{lemma}
 \begin{proof} For any $u\in W$, the subgroup $U^-_u:=U^-\cap uU^-u^{-1}$ acts freely
and transitively on $C^u$ via left multiplication, since the opposite cell $C^u=(U^-\cap uU^-u^{-1})uB/B$. Thus,
$U^-_S:=\cap_{u\in S} \,U^-_u$ acts freely on $V^S$. Clearly,  $w\ B^-\cdot x_{o}$, for any $w\in S$, is stable under 
$U_S^-$ and, further, each orbit of $U_S^-$ in the open subset $w\ B^-\cdot x_{o}$ of $\bar{X}$  is closed in $w\ B^-\cdot x_{o}$ (use [K, 6.1.5(c)]). Thus, each orbit of 
$U_S^-$ is closed in their union $V^S$.    In fact, $U_S^-$ acts properly on $V^S$. Hence,
 $U^-_S$  acts freely and properly
on $ X^{v}(S)$. Take any closed normal subgroup
${N}^-_S$ of $B^-$ of finite codimension contained in $U^-_S$. Then, the quotient
 $ Y^v(S):={N}^-_S\setminus X^{v}(S)$ acquires a canonical structure of a
 $B^-$-scheme of finite type over $\bc$ under the left
 multiplication action of
 $B^-$ on $Y^v(S)$, so that the quotient  $q: X^{v}(S) \to
 Y^{v}(S)$ is a principal ${N}^-_S$-bundle. 
 \end{proof}
 \begin{remark} (a) The above lemma allows us to define various local properties
 of $X^v$. In particular, a point $x\in X^v$ is called {\it normal} (resp. {\it CM})
 if the corresponding point in the quotient $Y^{v}(S)$ has that property, where $S$
 is a finite ideal such that $x\in X^v(S).$ Clearly, the property does not depend
 upon the choice of $S$ and ${N}^-_S$.

 (b) It is possible that the scheme $Y^v(S)$ is not separated. However,
 as observed by M. Kashiwara,  we can choose
 our closed normal subgroup
${N}^-_S$ of $B^-$ of finite codimension contained in $U^-_S$ appropriately so that
$Y^v(S)$ is indeed separated. In fact, we give the following more general result due to him.
\end{remark}
Let $\bf{k}$ be a field and let $\{S_\lambda\}_{\lambda \in \Lambda}$ be a filtrant projective system of 
quasi-compact $\bf{k}$-schemes locally of finite type over $\bf{k}$ . Assume that $f_{\lambda, \mu}:S_\mu 
\to S_\lambda$  is an affine morphism.
Set  $S=  \Invlt_{\lambda}\,S_\lambda $ and  let $p_\lambda:S\to S_\lambda$ be the canonical projection.
\begin{lemma} (due to M. Kashiwara) If $S$ is separated, then $S_\lambda$ is separated for some $\lambda$.
\end{lemma}
\begin{proof} Take a smallest element $\lambda_o\in \Lambda$. It is enough to show that for a pair of affine open subsets 
$U_o,V_o$ of $S_{\lambda_o}$,  $U_\lambda\cap V_\lambda\to U_\lambda\times V_\lambda$ is a closed 
embedding for some $\lambda$, where $U_\lambda :=f_{\lambda_o,\lambda}^{-1}(U_o)$ and  $V_\lambda 
:=f_{\lambda_o,\lambda}^{-1}(V_o)$. 

Note that $U_\lambda\cap V_\lambda$ is quasi-compact and of finite type over $\bf{k}$. Set $U=
p_{\lambda_o}^{-1}(U_o)$ and $V=
p_{\lambda_o}^{-1}(V_o)$. Since $S$ is separated,  $U \cap V \to U\times V$ is a closed 
embedding. In particular,  $U\cap V$ is affine. 

We have projective system of schemes $\{U_\lambda\cap V_\lambda\}_{\lambda \in \Lambda}$ and 
$\{U_\lambda\}_{\lambda \in \Lambda}$, and a projective system of morphisms 
$\{U_\lambda\cap V_\lambda \to U_\lambda\}_{\lambda \in \Lambda}$. Since 
 $\Invlt_{\lambda}\,(U_\lambda \cap V_\lambda) \simeq U\cap V$ is affine, the morphism 
 $\Invlt_{\lambda}\,(U_\lambda \cap V_\lambda) \to  \Invlt_{\lambda}\,(U_\lambda)$ is an affine morphism. Hence,
  $U_{\lambda_1} \cap V_{\lambda_1}\to U_{\lambda_1}$ is an affine morphism for some 
$\lambda_1$ by [GD, Theorem 8.10.5]. Hence, $U_{\lambda_1} \cap V_{\lambda_1}$ is affine. Now, by the assumption, 
$\co_S(U)\otimes \co_S(V)\to \co_S(U\cap V)$ is surjective. Since $U_o\cap V_o\to U_o$ is of finite type, 
$U\cap V\to U$ is of finite type. Hence, $\co_S(U\cap V)$ is an $\co_S(U)$-algebra of finite type. Since 
$\co_S(U)\otimes \co_S(V) \simeq	 \Dirlt_\lambda \bigl(\co_S(U)\otimes \co_{S_\lambda}(V_\lambda)\bigr)$, 
there exists $\lambda_2\to \lambda_1$ such that  $\co_S(U)\otimes \co_{S_{\lambda_2}}(V_{\lambda_2}) \to 
\co_S(U\cap V)$ is surjective. This means that $U\cap V \to U\times V_{\lambda_2}$ is a closed embedding. Now, 
consider the projective system
$$U_\lambda\cap V_\lambda\to U_\lambda\times V_{\lambda_2}.$$
Its projective limit with respect to $\lambda$ is isomorphic to $U\cap V\to U\times V_{\lambda_2},$ which is a 
closed embedding. Hence, again by loc. cit., $U_\lambda\cap V_{\lambda_3}\to U_{\lambda_3}\times V_{\lambda_2}$
is a closed embedding for some $\lambda_3 \to \lambda_2$. Then, $U_{\lambda_3}\cap V_{\lambda_3}\to U_{\lambda_3}\times V_{\lambda_3}$
is a closed embedding.
\end{proof}

\begin{theorem}\label{thm2}
For any $v\in W$ and any finite ideal $S\subset  W$ containing $v$,
there exists a smooth irreducible $B^{-}$-scheme $Z^{v}(S)$ and a
projective  $B^{-}$-equivariant morphism
$$
\pi^{v}_S:Z^{v}(S)\to X^{v}(S),
$$
satisfying the following conditions:

(a) The restriction  ${(\pi^{v}_S)}^{-1}(C^v)\to
  C^v$
is an isomorphism.

(b) $\partial Z^v(S):={(\pi^{v}_S)}^{-1}(\partial X^{v}(S))$ is a divisor with
  simple normal crossings, where
  $X^{v}(S):=
X^{v}\cap V^{S}$ and $
\partial X^{v}(S):= (\partial
X^{v})\cap V^{S}$.

(Here smoothness of $Z^v(S)$ means that there exists a closed subgroup ${N}^-_S$
of $B^-$ of finite codimension which acts freely and properly on $Z^v(S)$, such
 that the quotient is a smooth scheme of finite type over $\bc$.)
\end{theorem}

\begin{proof}
Observe that the action of $B^-$ on $Y^v(S)$ factors through the action of
 the finite dimensional algebraic group $B^-/{N}^-_S$, where $Y^v(S)$ is as defined
 in Lemma \ref{basic}. Now, take a
 $B^-$-equivariant desingularization $\theta: \bar{Z}^v(S) \to
 Y^{v}(S)$ such that $\theta$ is a projective morphism, $\theta^{-1} \bigl({N}^-_S\setminus C^v\bigr) \to
 {N}^-_S\setminus C^v$ is an isomorphism and $\theta^{-1} \bigl({N}^-_S\setminus (\partial X^{v}(S))\bigr)$ is a
 divisor with simple normal crossings
 (cf. [Ko, Proposition 3.9.1]\footnote{I thank Zinovy Reichstein for this reference.}; also see [Bi], [RY]). Now, taking the fiber product
 $$ Z^v(S)= \bar{Z}^v(S) \times_{Y^v(S)}\, X^{v}(S)$$
 clearly proves the theorem.
 \end{proof}

For  $w\in W$, take the ideal $S_{w}=\{u\leq w\}$. Then,
by [K, Lemma 7.1.22(b)],
$$
X^{v}(S_w)\cap X_{w}=X^{v}_{w}.
$$
\begin{lemma}\label{lem4}
The map
$$
\mu_{w}:U^{-}\times Z_{w}\to \bar{X},\quad
(g,z)\mapsto g\cdot\theta_{w}(z)
$$
is a smooth morphism,
where $\theta_w:Z_w\to X_w$ is the $B$-equivariant  BSDH desingularization
corresponding to a fixed reduced decomposition $w=s_{i_1}\dots s_{i_n}$
(cf. proof of Proposition \ref{n3.1}).
\end{lemma}

\begin{proof}
Consider the map
$$
\bar{\mu}_w : G\mathop{\times}^{B}Z_{w}\to \bar{X},\quad [g,z]\mapsto
g\theta_{w}(z).
$$

Because of the $G$-equivariance, it is a locally trivial fibration. Moreover,
it has smooth  fibers of finite type over $\bc$ (isomorphic with $Z_{w^{-1}}$, for the 
decomposition $w^{-1}=s_{i_n}\dots s_{i_1}$):

To see this, let $Z'_w$  be the fiber product:
\[
\xymatrix@=.5cm{
Z'_w\ar[dd]_{\theta'_{w}}
\ar[rr]&&
Z_{w}\ar[dd]^{\theta_w}\\
 & \square & \\
G\ar[rr] &&
\bar{X}.
}
\]
Then, we have the fiber diagram:
\[
\xymatrix@=.5cm{
G \times^B \,Z'_w\ar[dd]_{\hat{\mu}_{w}}
\ar[rr]&&
G \times^B \,Z_{w}\ar[dd]^{\bar{\mu}_w}\\
 & \square & \\
G\ar[rr] &&
\bar{X}.
}
\]
In particular, the fibers of $\bar{\mu}_w$ are isomorphic with the fibers of $\hat{\mu}_w$. Now, it is easy to see that the map
\[Z'_{w^{-1}} \to G \times^B \,Z'_w, \,\,\, z'\mapsto [\theta'_{w^{-1}} (z'), i(z')]\]
gives an isomorphism of $Z_{w^{-1}}$ with the fiber of $\hat{\mu}_{w}$ over $1$, where $i: Z'_{w^{-1}} \to Z'_{w} $ is the 
isomorphism induced from the map $(p_n, \dots, p_1) \mapsto (p_1^{-1}, \dots, p_n^{-1}).$ 

 In particular,
$\bar{\mu}_w$ is a smooth morphism and hence so is its restriction to the open subset
$U^{-}\times Z_{w}$.
\end{proof}

  \begin{proposition} \label{n5.6}  For any symmetrizable Kac-Moody group $G$ and any $v\leq w\in W$,
  the Richardson variety $X^v_w := X_w\cap X^v \subset \bar{X}$ is irreducible,
  normal and CM (and, of course, of finite type over $\bc$ since so is $X_w$).

  Moreover, $C_w\cap C^v$ is an open dense subset of $X^v_w$.
    \end{proposition}

  \begin{proof}  Consider the multiplication map
  \[
\mu : G \times^{B} X_w \to \bar{X}, \quad [g,x] \mapsto gx.
  \]
Then, $\mu$ being $G$-equivariant, it is a fibration.  Consider the pull-back
fibration:
    \begin{alignat*}{2}
&F^v_w\; \overset{\hat{i}}{\hookrightarrow}  &&G \times^{B}X_w\\
\hat{\mu}&\downarrow &&\; \downarrow\mu\\
&X^v \; \underset{i}{\hookrightarrow} &&\,\,\,\bar{X}\,\,,
  \end{alignat*}
where $i$ is the inclusion map.

Also, consider the projection map
  \[
\pi : G \times^{B}X_w \to \bar{X}, \quad [g,x] \mapsto gB.
  \]
Let $\hat{\pi}$ be the restriction $\hat{\pi} := \pi \circ \hat{i}:
F^v_w \to \bar{X}$.  Observe that $i$ being $B^-$-equivariant (and $\mu$ is
$G$-equivariant) $\hat{i}$ is $B^-$-equivariant and hence so is $\hat{\pi}$.
 In particular, $\hat{\pi}$ is a fibration over the open cell
 $B^-B/B \subset \bar{X}$. Moreover, since $\hat{\pi}$  is $B^-$-equivariant (in 
particular, $U^-$-equivariant) and $U^-$ acts transitively on $B^-B/B $ with trivial isotropy, 
$\hat{\pi}$ is a trivial fibration restricted to $B^-B/B $.

Now, by [KS, Propositions 3.2, 3.4], $X^v$ is normal and CM (and, of course,
 irreducible).  Also, $X_{w^{-1}}$ is normal, irreducible and CM [K, Theorem 8.2.2].  Thus,
 $\hat{\mu}$ being a fibration with fiber $X_{w^{-1}}$ (as can be seen by considering the embedding 
$X_{w^{-1}} \hookrightarrow G\times^B\,{\tilde{X}}_w, \,\,gB \mapsto [g, g^{-1}],$ where $\tilde{X}_w$ 
is the inverse image of $X_w$ in $G$), $F^v_w$ is irreducible, normal and CM, and hence
 so is its open subset $\hat{\pi}^{-1}(B^-B/B)$.  But, $\hat{\pi}$ is a trivial fibration
 restricted to $B^-B/B$ with fiber over $1\cdot B$ equal to $X^v_w =
 X_w\cap X^v$.  Thus, we get that $X^v_w$ is irreducible, normal and CM under the
 scheme theoretic intersection. Moreover, since  $X^v_w$ is Frobenius split
 in char. $p>0$ (cf. [KuS, Proposition 5.3]), we get that it is reduced.

 Clearly, $C_w\cap C^v$ is an open subset of $X^v_w$. So, to prove that
 $C_w\cap C^v$ is dense in $X^v_w$, it suffices to show that it is nonempty, which
 follows from the proof of [K, Lemma 7.3.10].
  \end{proof}
\begin{remark} \label{newremark} By the same proof as above, applying Corollary \ref{9.5}, we see that
$X_w\cap \partial X^v$ is CM.
\end{remark}

\begin{theorem}\label{thm3}
For any $v\leq w$, consider the fiber product
$$
Z^{v}(S_w)\mathop{\times}_{\bar{X}} Z_{w},
$$
where $Z_{w}$ is the BSDH ($B$-equivariant)
desingularization of $X_{w}$ (corresponding to a fixed reduced decomposition
$w=s_{i_{1}}\ldots s_{i_{n}}$ of $w$) and $\pi^{v}_{S_w}:Z^{v}(S_w)\to X^{v}(S_w)$
is a $B^-$-equivariant
desingularization of $X^v(S_w)$ as in Theorem \ref{thm2}.

Then, $Z^{v}(S_w)\mathop{\times}_{\bar{X}} Z_{w}$ is  a smooth
 projective irreducible
$T$-variety (of finite type over $\bc$) with a canonical $T$-equivariant morphism
\[\pi^v_w: Z^{v}(S_w)\mathop{\times}_{\bar{X}} Z_{w} \to X^v_w.\]
Moreover,  $\pi^{v}_{w}$ is a $T$-equivariant desingularization which is an isomorphism 
restricted to the inverse image of the dense open subset $C^v\cap C_w$ of $X_w^v$. From now on,
we abbreviate
\[Z^v_w:=
Z^{v}(S_w)\mathop{\times}_{\bar{X}} Z_{w}.\]
\end{theorem}

\begin{proof} Consider the commutative diagram:
\[
\xymatrix{
Z^{v}(S_w)\displaystyle\mathop{\times}_{\bar{X}}Z_{w}\ar[r]\ar[ddr]_{\pi^{v}_{w}}
& X^{v}(S_w)\displaystyle\mathop{\times}_{\bar{X}}X_{w}\ar@{=}[d]\\
 & X^{v}(S_w)\cap X_{w}\ar@{=}[d]\\
 & X^{v}_{w},
}
\]
where the horizontal map is the fiber product of the two desingularizations
and $\pi^v_w$ is the horizontal map under the above identification of
$X^{v}(S_w)\displaystyle\mathop{\times}_{\bar{X}}X_{w}$ with $X^v_w$.
Clearly, $\pi^{v}_{w}$ is $T$-equivariant and it is an isomorphism restricted
to the inverse image of the dense open subset
$C^v\cap C_w$ of $X^v_w$. In particular, $\pi^{v}_{w}$ is birational.

Define $E^{v}_{w}$ as the fiber product
\[
\xymatrix@=.5cm{
E^{v}_{w}\ar[dd]_{f^{v}_{w}}
\ar[rr]^{\hat{\mu}^{v}_{w}} &&
Z^{v}(S_{w})\ar[dd]^{\pi^{v}_{S_w}}\\
 & \square & \\
U^{-}\times Z_{w}\ar[rr]_{\mu_{w}} &&
\bar{X},
}
\]
where $\mu_w$ is as in Lemma \ref{lem4}.
Since $\mu_{w}$ is a smooth morphism by Lemma \ref{lem4},
so is $\hat{\mu}^{v}_{w}$. But
$Z^{v}(S_{w})$ is a smooth scheme and hence so is
$E^{v}_{w}$. Now, since both of $U^{-}\times
Z_{w}$ and $Z^{v}(S_{w})$ are $U^{-}$-schemes
(with $U^-$ acting on $U^{-}\times Z_{w}$ via the left multiplication
on the first factor)
and the morphisms $\pi^{v}_{S_w}$ and $\mu_{w}$ are
$U^{-}$-equivariant,
$E^{v}_{w}$ is a $U^{-}$-scheme (and $f^v_w$ is $U^-$-equivariant). Consider the
composite morphism
$$
E^{v}_{w}\xrightarrow{f^{v}_{w}}U^{-}\times
Z_{w}\xrightarrow{\pi_{1}} U^{-},
$$
where $\pi_1$ is the projection on the first factor.

It is $U^{-}$-equivariant with respect to the left multiplication of
$U^{-}$ on $U^{-}$. Let $\mathbb{F}$ be the fiber of $\pi_{1}\circ
f^{v}_{w}$ over $1$. Define the isomorphism
\[
\xymatrix{
E^{v}_{w}\ar[dr]_{\pi_{1}\circ f^{v}_{w}} & & U^{-}\times
\mathbb{F}\ar[ll]_{\sim}^{\theta}\ar[dl]^{\bar{\pi}_{1}}\\
 & U^{-} &
}
\]
\[
\theta (g,x) = g\cdot x, \,\,
\theta^{-1}(y) = \left((\pi_{1}\circ
f^{v}_{w})(y),\left(\pi_{1}\circ f^{v}_{w}(y)\right)^{-1}y\right).
\]

Since \ $E^{v}_{w}$ is a smooth scheme, so is
$\mathbb{F}$.

But,
$$
\mathbb{F}:=Z^{v}_w.
$$

Now, $\pi^{v}_{S_w}$ is a projective morphism onto $X^{v}(S_{w})$
and hence $\pi^{v}_{S_w}$ is a projective morphism considered as
a map  $Z^{v}(S_{w})\to V^{S_{w}}$. Also,
$\mu_{w}$ has its image inside $V^{S_{w}}$, since
$BuB/B\subset uB^{-}B/B$ for any $u\in W$.

Thus, $f^{v}_{w}$ is a projective morphism and
hence
$$
(f^{v}_{w})^{-1}(1\times
Z_{w})=Z^{v}_w
$$
is a projective variety.

Now, as observed by D. Anderson and independently by M. Kashiwara,  $Z^{v}_w$ is
irreducible:

Since $Z^{v}(S_{w})\to X^v(S_{w})$ is a proper desingularization, all its fibers
are connected and hence so are all the nonempty fibers of $f^v_w$. Now,
$\mu_w^{-1}(\OpIm \pi_{S_w}^v)= U^-\times Y,$ where $Y\subset Z_w$ is the closed
subvariety defined as the inverse image of the Richardson variety $X^v_w$ under the
BSDH desingularization $\theta_w:Z_w\to X_w$. Since $X^v_w$  is irreducible,
$\theta_w$ is proper,
and all the fibers of $\theta_w$ are connected, $Y=\theta_w^{-1}(X^v_w)$ is connected and hence
so is $\mu_w^{-1}(\OpIm  \pi_{S_w}^v)$. Since the pull-back of a proper morphism
is proper (cf. [H, Corollary 4.8, Chap. II]), the surjective morphism $f^v_w:E^v_w \to
U^-\times Y$ is proper. Now, $U^-\times Y$ being connected and all the fibers of
$f^v_w$ over $U^-\times Y$  are nonempty and connected, we get that $E^v_w$
is connected and hence so is $\mathbb{F}$. Thus, $\mathbb{F}$ being smooth,
it is irreducible.
This proves the theorem.
\end{proof}

The action of $B$ on $Z_w$ factors through the action of a finite dimensional
quotient group $\bar{B}= B_w$ containing the maximal torus
$H$. Let $\bar{U}$ be the image of $U$ in $\bar{B}$.

\begin{lemma}\label{lem7}
For any $u\leq w$, the map $\bar{\mu}:\bar{U}\times
Z^{u}_{w}\to Z_{w}$ is a smooth
morphisms and hence so is $\bar{B}\times
Z^{u}_{w}\to Z_{w}$, where $(b,z)\mapsto b\cdot \pi_2(z)$, for
$b\in \bar{B}, z\in Z^u_w$. (Here $\pi_2: Z^u_w\to Z_w$ is the canonical
projection map.)
\end{lemma}

\begin{proof}
First of all, the map
$$
\mu':G\times^{B^{-}}\,Z^{u}(S_{w})\to \bar{X},\quad [g,z]\mapsto
g\pi^{u}_{S_w}(z)
$$
being $G$-equivariant, is a locally trivial fibration. (It is trivial over the open subset $U^- \subset \bar{X}$.) 

We next claim that the following diagram is a Cartesian diagram:
\begin{equation*}
\vcenter{\xymatrix@=.7cm{
 U\times Z^{u}_{w}\ar[dd]_{\mu}\ar[rr] & &
U\times Z^{u}(S_w)\ar[dd]^{\hat{\mu}'}\\
 & \square & \\
 Z_{w}\ar[rr] && \bar{X},
}}\tag{$\mathfrak{D}$}
\end{equation*}
where $\mu(u,z)= u\cdot\pi_2(z)$ and
$\hat{\mu}'(u,z)= u\cdot \pi^u_{S_w}(z)$.
Define the map
\[
\theta: U\times Z^{u}_{w} \to
(U\times
Z^{u}(S_w))\mathop{\times}_{\bar{X}}  Z_{w},\,\,\,\,
(u,z) \mapsto ((u,\pi_{1}(z)), u\cdot\pi_2(z)),
\]
where $\pi_{2}:Z^{u}_{w}\to Z_{w}$
and $\pi_{1}:Z^{u}_{w}\to Z^{u}(S_w)$ are
the canonical morphisms.

Define the map
\[
\theta' : (U\times
Z^{u}(S_w))\mathop{\times}_{\bar{X}}  Z_{w} \to U\times
Z^{u}_{w},\quad
((u,z_{1}), z_{2})\mapsto (u,(z_{1},u^{-1}z_{2})).
\]

Clearly $\theta$ and $\theta'$  are inverses to each other and hence
$\theta$ is an isomorphism. Thus, the above diagram $(\mathfrak{D})$
is a Cartesian diagram:

Now, consider the pull-back diagram:
\[
\xymatrix@=.5cm{
\mathbb{E}\ar[dd]_{\beta}\ar[rr]^-{\alpha} & &
G\displaystyle\mathop{\times}^{B^{-}}Z^{u}(S_{w})\ar[dd]^{\mu'}\\
& \square & \\
Z_{w}\ar[rr] & & \bar{X}.
}
\]

Since $\mu'$ is a locally trivial fibration, so is the map
$\beta$. Moreover, since the diagram $(\mathfrak{D})$ is a Cartesian
 diagram, $\alpha^{-1}(U\times Z^{u}(S_{w}))\simeq
U\times Z^{u}_{w}$ and $\beta_{|U\times
  Z^{u}_{w}}=\mu$. Thus, the differential of
  $\mu$ is surjective at the Zariski tangent spaces.

Since the morphism $\mu:U\times
Z^{u}_{w}\to Z_{w}$ factors through a
finite dimensional quotient $\bar{\mu}:\bar{U}\times
Z^{u}_{w}\to Z_{w}$, the
differential of $\bar{\mu}$ continues to be surjective at the
Zariski tangent spaces. Since $\bar{U}$,
$Z^{u}_{w}$ and $Z_{w}$ are smooth
varieties, we see that $\bar{\mu}:\bar{U}\times
Z^{u}_{w}\to Z_{w}$ is a smooth
morphism (cf. [H, Chap. III, Proposition 10.4]). 

To prove that the map $\bar{B}\times
Z^{u}_{w}\to Z_{w}$ is a smooth morphism, it suffices to observe that 
$H\times Z_w \to Z_w, \,(h,z) \mapsto h\cdot z$ ,  is a smooth morphism. This proves the lemma.
\end{proof}

\begin{lemma}\label{lem9}
The map $\bar{B}\times X^{u}_{w}\to X_{w}, (b,x)\mapsto b\cdot x$, is a flat
morphism for any $u\leq w$.
\end{lemma}

\begin{proof}
The map
$$
\mu:G {\times}^{U^{-}}\,X^{u}(S_w)\to \bar{X},\quad [g,x]\mapsto g\cdot x,
$$
being $G$-equivariant, is a fibration. In particular, it is a flat map
and hence its restriction (to an open subset) $\mu':B\times X^{u}(S_w)\to \bar{X}$
is a flat
map. Now, ${\mu'}^{-1} (X_{w})=B\times X^{u}_{w}$. Thus,
$\mu':B\times X^{u}_{w}\to X_{w}$ is a flat map. Now, since
$B\times X^{u}_w \to \bar{B}\times X^{u}_w$ is a locally trivial fibration
(in particular, faithfully flat), the map $\bar{B}\times X^{u}_w \to X_w$
is flat (cf. [M, Chap. 3, \S 7]). This proves the
lemma.
\end{proof}

The canonical action of $\Gamma= \Gamma_{B\times B}$ on $(Z_w^2)_\bp$ descends
to the action of a finite dimensional quotient group $\bar{\Gamma}=\Gamma_w:$
$$ \Gamma  \twoheadrightarrow  \bar{\Gamma}=\Gamma_w \twoheadrightarrow
 GL(N+1)^{r},$$
 where $(Z_w^2)_\bp$ and $\Gamma$ are as in Section \ref{sec6}. In fact, we can (and do)
 take
 $$\bar{\Gamma}= \bar{\Gamma}_0 \rtimes GL(N+1)^{r},$$
 where $\bar{\Gamma}_0 $ is the group of global sections of the bundle
 $E(T)_\bp
\displaystyle\mathop{\times}^{T}(\bar{B}^{2}) \to \bp$, where $\bar{B}=B_w$ is defined
just above Lemma \ref{lem7}.

\begin{lemma}\label{lem5}
For any ${\bf j}=(j_{1},\ldots,j_{r})\in [N]^r$ and $u, v \leq w$, the map
$$
\tilde{m}:\bar{\Gamma}\times
\left(Z^{u,v}_{w}\right)_{\bf j}\to
(Z_w^2)_\bp
$$
is a smooth morphism, where $Z^{u,v}_{w}:=Z^{u}_{w} \times Z^{v}_{w}$ under the
diagonal action of $T$, $\left(Z^{u,v}_{w}\right)_{\bf j}$ is the inverse image
of $\bp_{\bf j}$ under the map $E(T)_\bp
\displaystyle\mathop{\times}^{T} Z^{u,v}_w \to \bp$ and $\tilde{m}(\gamma, x)=
\gamma\cdot \pi_2(x)$. (Here $\pi_2: \left(Z^{u,v}_{w}\right)_{\bf j}\to
(Z_w^2)_\bp$ is the map induced from the canonical projection
$p:Z^{u}_{w} \times Z^{v}_{w} \to Z_w^2$.)
\end{lemma}

\begin{proof}
Consider the following commutative diagram, where both the horizontal
right side maps are
fibrations with left side as fibers:
\[
\xymatrix{
\bar{\Gamma}_{0}\times Z^{u,v}_{w}\ar[d]^{m'}\ar[r] & \bar{\Gamma} \times
\left(Z^{u,v}_{w}
\right)_{{\bf j}}\ar[d]^{\tilde{m}}\ar[r] &
GL(N+1)^{r}\times \mathbb{P}_{{\bf j}}
\ar[d]^{m''}\\
Z^{2}_{w}\ar[r] & (Z_w^2)_\bp\ar[r] &
\mathbb{P}=(\mathbb{P}^{N})^{r} .
}
\]
Here $m'$ is the restriction of $\tilde{m}$ and $m''$ is the restriction of the
standard map $GL(N+1)^{r}\times \mathbb{P} \to \bp$ induced from the action of
$GL(N+1)$ on $\bp^N$. Thus, $m'$ takes $(\gamma, z)$ to $\gamma(*)\cdot p(z)$, where
$*$ is the base point in $\bp$.
Clearly, $m''$ is a smooth morphism since it is $GL(N+1)^{
  r}$-equivariant and $GL(N+1)^{r}$ acts transitively
on $\mathbb{P}$. We next claim that $m'$ is a
smooth morphism: By the analogue of Lemma \ref{lem6} for $\Gamma_B$ replaced by 
$\Gamma=\Gamma_{B\times B}$ (see the remark following Lemma \ref{connected}), it suffices to show that
$$
\bar{B}^{2}\times Z^{u,v}_{w}\to
Z^{2}_{w}
$$
is a smooth morphism, which follows from Lemma \ref{lem7} asserting that
$
\bar{B}\times Z^{u}_{w}\to Z_{w}
$
is a smooth morphism. Since $m'$ and $ m''$ are smooth morphisms, so is
$\tilde{m}$ by [H, Chap. III, Proposition 10.4].
\end{proof}

\begin{lemma}\label{lem8} Let $u,v \leq w$.
The map $m:\bar{\Gamma} \times (X^{u,v}_{w})_{{\bf j}}\to (X^{
  2}_{w})_\bp$ is flat, where $m$ is defined similarly to the map
  $\tilde{m}: \bar{\Gamma} \times (Z^{u,v}_{w})_{{\bf j}}\to (Z^{
  2}_{w})_\bp$ as in Lemma \ref{lem5}.

Similarly, its restriction  $m': \bar{\Gamma}\times \partial ((X^{u,v}_{w})_{{\bf j}})\to
(X_w^2)_\bp$ is flat, where $X^{u,v}_w:=X^{u}_w \times X^{u}_w$ ,
$$
\partial \left((X^{u,v}_{w})_{{\bf j}}\right):=\left((\partial X^{u,v})
\cap (X^{
  2}_{w})\right)_{{\bf j}}\cup (X^{u,v}_{w})_{\partial \mathbb{P}_{{\bf j}}},
  $$
  $(X^{u,v}_{w})_{\partial \mathbb{P}_{{\bf j}}}$ is the inverse image of
  $\partial \mathbb{P}_{{\bf j}}$ under the standard quotient map $E(T)_\bp\times^T
  X^{u,v}_{w}\to \bp$,
  and $\partial X^{u,v}:=
\left((\partial X^{u})\times
X^{v}\right)\cup \left(X^{u}\times (\partial X^{v})\right).$
\end{lemma}

\begin{proof}
Consider the following diagram where both the horizontal right side  maps are
locally trivial fibrations with left side as fibers:
\[
\xymatrix{
\bar{\Gamma}_{0}\times X^{u,v}_{w}\ar[d]^{\hat{m}'}\ar[r] & \bar{\Gamma}\times
(X^{u,v}_{w})_{{\bf j}}\ar[d]^{{m}}\ar[r] & GL(N+1)^{
  r}\times \mathbb{P}_{{\bf j}}\ar[d]^{m''}\\
X^{ 2}_{w}\ar[r] & (X_w^2)_\bp\ar[r] &
\mathbb{P}=(\mathbb{P}^{N})^{r}
}
\]

Since the two horizontal maps are fibrations and $m''$ is a smooth morphism
(cf. proof of Lemma \ref{lem5}), to prove that $m$ is flat, it suffices to show that
$\hat{m}':\bar{\Gamma}_{0}\times X^{u,v}_{w}\to X^{2}_{w}$ is a flat
morphism. By the analogue of Lemma \ref{lem6} for $\Gamma$, it suffices to show that
$$
(\bar{B}^{2})\times X^{u,v}_{w}\to X^{ 2}_{w}
$$
is a flat morphism,  which folows from Lemma \ref{lem9}.

Observe first that, by the same proof as that of Lemma \ref{lem9}, the morphism $\bar{B} ^2\times ((\partial X^{u,v} \cap X_w^2) \to
X_w^2$ is flat. Now, to prove that the map $\bar{\Gamma} \times \partial((X^{u,v}_{w})_{{\bf j}})\to
(X_w^2)_\bp$ is flat, observe that
(by the same proof as that of the first part) it is flat
restricted to the components $\Gamma_{1}:=\bar{\Gamma} \times ((\partial
X^{u,v})\cap X^{2}_{w})_{{\bf j}}$ and $\Gamma_{2}:=\bar{\Gamma}\times
(X^{u,v}_{w})_{\partial \mathbb{P}_{{\bf j}}}$ and also restricted to the
intersection $\Gamma_{1}\cap \Gamma_{2}$. Thus, it is flat on
$\Gamma_{1}\cup \Gamma_{2}$, since for affine  scheme $Y=Y_{1}\cup
Y_{2}$, with closed subschemes $Y_{1}$, $Y_{2}$, and a morphism $f:Y \to X$ of schemes,  the sequence
$$0\to
k[Y]\to k[Y_{1}]\oplus k[Y_{2}]\to k[Y_{1}\cap Y_{2}]\to 0$$ is exact as a sequence of $k[X]$-modules.
\end{proof}

The following  Lemmas \ref{newlem1} and \ref{n7.3} are not used in the paper. However, we
have included them here for their potential usefulness. The following lemma is used in the proof of Lemma \ref{n7.3}.
\begin{lemma}  \label{newlem1} For any  $u\leq w$,
\[\co_{X^{u}}\left(-\partial
X^{u}\right)\otimes_{\co_{\bar{X}}}\,\co_{X_{w}}\left(-\partial
X_{w}\right)\simeq \co_{X^{u}_{w}}\left(- \left((\partial
X^{u}_{w}) \cup (X^{u}\cap \partial X_{w})\right)\right),\]
where recall that $\partial {X^{u}_{w}}:=(\partial X^u)\cap X_w$ taken as the scheme theoretic intersection inside $\bar{X}$.
\end{lemma}
\begin{proof} First of all
$$0\to \co_{X^{u}}\left(-\partial
X^{u}\right)\otimes_{\co_{\bar{X}}}\,\co_{X_{w}}\to \co_{X^{u}}\otimes_{\co_{\bar{X}}}\,\co_{X_{w}}=\co_{X^u_w}\to 
\co_{\partial X^{u}_w}\to 0$$
is exact since (by Corollary \ref{newcor5.5})
\beqn\label{n8.5.1}\tor_1^{\co_{\bar{X}}}\left(\co_{\partial
X^{u}}, \co_{X_w}\right)=0.
\eeqn
Thus, 
\beqn\label{n8.5.2} \co_{X^{u}}\left(-\partial
X^{u}\right)\otimes_{\co_{\bar{X}}}\,\co_{X_{w}}\simeq \co_{X^u_w}\left(
-\partial X^{u}_w\right).
\eeqn
Similarly,
\beqn\label{n8.5.3} 0\to \co_{X^{u}}\left(-\partial
X^{u}\right)\otimes_{\co_{\bar{X}}}\,\co_{X_{w}}\left(-\partial X_w\right)\to \co_{X^{u}}\left(-\partial X^u\right)\otimes_{\co_{\bar{X}}}\,\co_{X_{w}}
\to \co_{X^{u}}\left(-\partial X^u\right)\otimes_{\co_{\bar{X}}}\,\co_{\partial X_{w}}
\to 0
\eeqn
is exact since 
\beqn\label{n8.5.4}\tor_1^{\co_{\bar{X}}}\left(\co_{X^u}\left(-\partial
X^{u}\right), \co_{\partial X_w}\right)=0.
\eeqn
To prove \eqref{n8.5.4}, observe that, by a proof similar to that of Corollary \ref{newcor5.5},
\beqn\label{n8.5.5}\tor_j^{\co_{\bar{X}}}\left(\co_{X^u}, \co_{\partial X_w}\right)=0,\,\,\,\text{and}\,\,
\tor_j^{\co_{\bar{X}}}\left(\co_{\partial X^u}, \co_{\partial X_w}\right)=0,\,\,\,\text{for all} \,\, j>0.
\eeqn
Now,  \eqref{n8.5.2}, \eqref{n8.5.3}  and \eqref{n8.5.5} together prove the lemma.
\end{proof}

\begin{lemma} \label{n7.3} Let $u\leq w$. As $T$-equivariant sheaves, the dualizing sheaf
$$
\omega_{X^{u}_{w}}\simeq \co_{X^{u}_{w}}\left(-\left((\partial
X^{u}_{w})\cup (X^{u}\cap \partial X_{w})\right)\right),
$$
where $X^{u}\cap \partial (X_{w})$ is taken as the scheme theoretic intersection inside $\bar{X}$.
\end{lemma}
\begin{proof} Since $X^{u}_{w}$ is CM by Proposition \ref{n5.6} (in particular,
so is $X_w$) and the codimension of $X^{u}_{w}$ in $X_w$ is $\ell(u)$,
the dualizing sheaf
\beqn \label{e7.3.1}\omega_{X^{u}_{w}}\simeq \ext^{\ell (u)}_{\co_{X_w}}
\bigl(\co_{X^u_w}, \omega_{X_w}\bigr),
\eeqn
(cf. [E, Theorm 21.15]).
By the same proof as that of Lemma \ref{4.6} 
\beqn \label{e7.3.2}
\ext^{\ell (u)}_{\co_{X_w}}
\bigl(\co_{X^u_w}, \omega_{X_w}\bigr)\simeq \ext^{\ell (u)}_{\co_{\bar{X}}}
\bigl(\co_{X^u}, \co_{\bar{X}}\bigr)\otimes_{\co_{\bar{X}}} \, \omega_{X_w}.
  \eeqn
By [GK, Proposition 2.2], as $T$-equivariant sheaves,
\beqn \label{e7.3.3}
\omega_{X_w}\simeq e^{-\rho}\mathcal{L}(-\rho)\otimes \co_{X_w}(-\partial X_w).
\eeqn
(Even though we assume that $G$ is of finite type in [GK], but the same proof
 works for a general Kac-Moody group.)
Thus, the lemma follows by combining the isomorphisms \eqref{e7.3.1} -
\eqref{e7.3.3} together with Theorem \ref{prop:main} (due to Kashiwara)
and Lemma \ref{newlem1}.
\end{proof}

\section{$\mathcal{Z}$ has rational singularities}\label{sec2}

Recall the definition of $\bp$ and $\bp_{\bf j}$ and the embedding
$\tilde{\Delta}$ from Section \ref{sec6}.
 Fix $u,v \leq w$ and ${\bf j}$. Also recall the definition of the quotient group
 $\bar{\Gamma}=\Gamma_w$ of $\Gamma$ and the map $\tilde{m}$ from Lemma
 \ref{lem5} and the map $m$ from Lemma \ref{lem8}.

In the following commutative diagram, $\tilde{\mathcal{Z}}$ is defined as the fiber
product
$\bigl(\bar{\Gamma}\times (Z^{u,v}_{w})_{{{\bf j}}}\bigr)\times_{(Z^2_w)_\bp}\,
\tilde{\Delta}((Z_w)_\bp)$, and ${\mathcal{Z}}$ is defined as the fiber product
$\bigl(\bar{\Gamma}\times (X^{u,v}_{w})_{{{\bf j}}}\bigr)\times_{{(X^2_w)_\bp}}\,
\tilde{\Delta}((X_w)_{\bp})$.
In particular, both of ${\mathcal{Z}},  \tilde{\mathcal{Z}}$ are schemes
of finite type over $\bc$.
The map $f$ is the restriction of $\theta$ to $\tilde{\mathcal{Z}}$
(via $\tilde{i}$) with image inside ${\mathcal{Z}}$. The maps
$\tilde{\pi},$ and  $\pi$ are obtained from the projections
to the $\bar{\Gamma}$-factor via the maps $\tilde{i}$ and
$i$ respectively.

\begin{figure}[h]
\centering
\includegraphics[scale=.9]{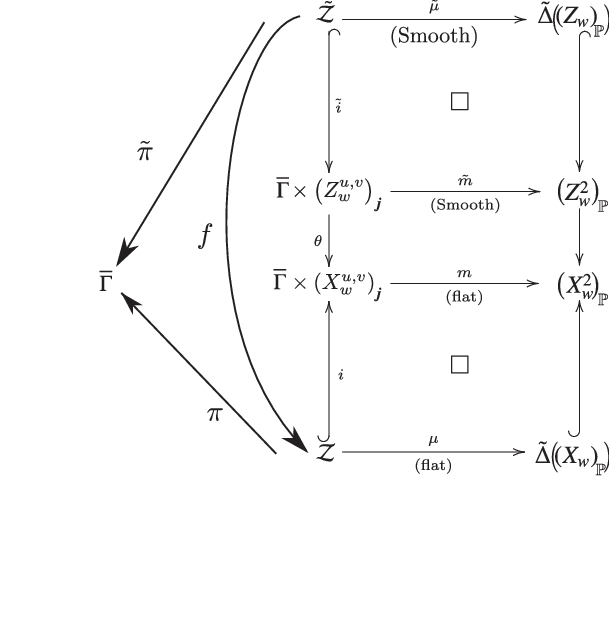}
\end{figure}

\begin{lemma}\label{lem10}
$\Pic  \bar{\Gamma}$ is trivial.
\end{lemma}

\begin{proof}
 First of all, by definition given above Lemma \ref{lem5}, $\bar{\Gamma}$ is the semi-direct
 product of $GL(N+1)^{
  r}$ with
$\bar{\Gamma}_{0}=\Gamma(E(T)_\bp\displaystyle\mathop{\times}^{T}(\bar{B}^2))\simeq
H^2\times
\Gamma(E(T)_\bp\displaystyle\mathop{\times}^T (\bar{U}^2))$.

Since $\bar{U}^2$ is $T$-isomorphic with its Lie algebra,
$\Gamma(E(T)_\bp\displaystyle\mathop{\times}^T(\bar{U}^2))$ is an affine space.
Thus, as a variety, $\bar{\Gamma}$ $\bigl($which is isomorphic with $GL(N+1)^r\times H^2\times
\Gamma(E(T)_\bp\displaystyle\mathop{\times}^T(\bar{U}^2))\bigr) $ is an open subset of
an affine space $\mathbb{A}^N$. In particular, any prime divisor of $\bar{\Gamma}$
extends to a prime divisor of  $\mathbb{A}^N$, and thus its ideal is principal. Hence,
$\Pic(\bar{\Gamma})=\{1\}$.
\end{proof}

The following result is a slight variant of  [FP, Lemma on page 108].

\begin{lemma} \label{n6.3} Let $f:W\to X$  be a flat morphism from a
 pure-dimensional CM
scheme $W$ of finite type over $\bc$  to a CM  irreducible variety $X$ and let $Y$ be a closed
CM
subscheme of $X$ of pure codimension $d$. Set
$Z:= f^{-1}(Y)$. If $\codim_Z(W)\geq d$, then equality holds and $Z$  is
CM.
\end{lemma}
\begin{proof} (due to N. Mohan Kumar) The assertion (and the assumptions) of the lemma is clearly local,
so we have a local map $A\to B$ of local
rings with $B$ flat over $A$. If $P\subset A$ is a prime ideal of codimension $d$
with $PB$ of pure codimension $d$, we only need to check that $B/PB$ is CM. But,
$A/P$ is CM, so we can pick a regular sequence $\{a_1,\dots, a_d\}$ mod $P$. Flatness
of $f$
ensures that these remain a regular sequence in $B/PB$.
\end{proof}

We also need the following  original [FP, Lemma on page 108].

\begin{lemma} \label{n6.4} Let $f:W\to X$  be a  morphism from a pure-dimensional CM
scheme $W$  of finite type over $\bc$ to a  smooth irreducible variety $X$ and let $Y$ be a closed
CM
subscheme of $X$ of pure codimension $d$. Set
$Z:= f^{-1}(Y)$. If $\codim_Z(W)\geq d$, then equality holds and $Z$  is
CM.
\end{lemma}
\begin{proposition} \label{propn6.2} The schemes $\mathcal{Z}$ and
 $\tilde{\mathcal{Z}}$ are irreducible and the map $f: \tilde{\mathcal{Z}}
 \to \mathcal{Z}$ is a proper birational map. Thus, $\tilde{\mathcal{Z}}$
 is a desingularization of $\mathcal{Z}$.

 Moreover, $\mathcal{Z}$ is CM with
 \beqn \label{neq6.2.1}
\dim \mathcal{Z}=|{\bf j}|+\ell(w)-\ell(u)-\ell(v)+\dim\bar{\Gamma},
\eeqn
where $|{\bf j}|:=\sum_i j_i$\, for \,${\bf j}=(j_1, \dots, j_r)$.
 \end{proposition}
\begin{proof} We first show that $\tilde{\mathcal{Z}}$ and $\mathcal{Z}$
are pure dimensional.

Since $\tilde{m}$ is a smooth (in particular, flat) morphism,
$\Iim (\tilde{m})$ is an open subset of $(Z^2_w)_\bp$ (cf. [H, Exercise 9.1,
Chap. III]). Moreover, clearly $\Iim (\tilde{m})\supset (C^2_w)_\bp$, thus
$\Iim (\tilde{m})$ intersects $\tilde{\Delta}((Z_w)_\bp)$. Applying
[H, Corollary 9.6, Chap. III] first to the morphism $\tilde{m}: \bar{\Gamma}\times
(Z^{u,v}_{w})_{{{\bf j}}} \to \Iim (\tilde{m})$ and then to its restriction
$\tilde{\mu}$ to $\tilde{\mathcal{Z}}$, we see that $\tilde{\mathcal{Z}}$
is pure dimensional. Moreover,
\begin{align}\label{neq6.2.2}
\dim \tilde{\mathcal{Z}}&=
  \dim \bar{\Gamma}+|{\bf j}|+ \dim \bigl(Z^{u,v}_{w}\bigr) - \dim
  \bigl((Z^2_w)_\bp\bigr) + \dim \bigl(\tilde{\Delta}((Z_w)_\bp)\bigr) \notag \\
  &=\dim \bar{\Gamma}+|{\bf j}|+ \ell(w)-\ell(u)-\ell(v).
\end{align}

By the same argument, we see that $\mathcal{Z}$ is also pure dimensional.

We now show that $\tilde{\mathcal{Z}}$ is irreducible:

The smooth morphism $\tilde{m}: \bar{\Gamma}\times
(Z^{u,v}_{w})_{{{\bf j}}} \to
(Z^2_w)_\bp$ is  $\bar{\Gamma}$-equivariant with respect to the left multiplication
of $\bar{\Gamma}$ on the first factor of $\bar{\Gamma}\times
(Z^{u,v}_{w})_{{{\bf j}}}$ and the standard action of $\bar{\Gamma}$ on
$(Z^2_w)_\bp$. Since $(C^2_w)_\bp$ is a single $\bar{\Gamma}$-orbit
(by the analogue of Lemma \ref{lem6} for $B$ replaced by $B\times B$), $ \tilde{m}^{-1}((C^2_w)_\bp) \to (C^2_w)_\bp$
is a locally trivial fibration in the analytic topology. Further, since the fundamental group
$\pi_1((C^2_w)_\bp)=\{1\}$ and, of course,  $ \tilde{m}^{-1}((C^2_w)_\bp)$
is irreducible (in particular, connected), from the long exact homotopy sequence
for the fibration $ \tilde{m}^{-1}((C^2_w)_\bp) \to (C^2_w)_\bp$, we get
that all its fibers are connected. Thus,
the open subset $\tilde{\mathcal{Z}} \cap \tilde{m}^{-1}((C^2_w)_\bp)$
is connected as the fibers and the base are connected. Hence, it is irreducible
(being smooth). Consider the closure
$\tilde{\mathcal{Z}}_1:= \overline{\bigl(\tilde{\mathcal{Z}} \cap
\tilde{m}^{-1}((C^2_w)_\bp)\bigr)}.$ Then, $\tilde{\mathcal{Z}}_1$
is an irreducible component of $\tilde{\mathcal{Z}}$. If possible, let
$\tilde{\mathcal{Z}}_2$ be another irreducible component of $\tilde{\mathcal{Z}}$.
Then, $\tilde{\mu}(\tilde{\mathcal{Z}}_2) \subset
\tilde{\Delta}\bigl((Z_w\setminus C_w)_\bp \bigr)$. Since
$\dim\bigl(\tilde{\Delta}\bigl((Z_w\setminus C_w)_\bp \bigr)\bigr) <
\dim\bigl(\tilde{\Delta}((Z_w)_\bp)\bigr)$ and each fiber of
$\tilde{\mu}_{|\tilde{\mathcal{Z}}_2}$ is of dimension at most that of any fiber of
$\tilde{\mu}$, we get that $\dim \bigl(\tilde{\mathcal{Z}}_2\bigr) <
\dim \bigl(\tilde{\mathcal{Z}}_1\bigr)$. This is a contradiction since
$\tilde{\mathcal{Z}}$ is of pure dimension. Thus, $\tilde{\mathcal{Z}}=
\tilde{\mathcal{Z}}_1$ and hence $\tilde{\mathcal{Z}}$ is irreducible.

The proof of the irreducibility of $\mathcal{Z}$ is similar. The only extra observation
we need to make is that $\tilde{\mathcal{Z}} \cap
\tilde{m}^{-1}((C^2_w)_\bp)$ maps surjectively onto ${\mathcal{Z}} \cap
{m}^{-1}((C^2_w)_\bp)$ under $f$; in particular, ${\mathcal{Z}} \cap
{m}^{-1}((C^2_w)_\bp)$ is irreducible.

The map  $f$ is clearly proper. Moreover, it is an isomorphism restricted
to the (nonempty) open subset
$$\tilde{\mathcal{Z}} \cap \bigl(\bar{\Gamma} \times \bigl((C^u\cap C_w)\times
(C^v\cap C_w)\bigr)_{\bf j}\bigr)$$ onto its image (which is an open subset
of $\mathcal{Z}$). (Here we have identified the inverse image $(\pi_w^u)^{-1}
(C^u\cap C_w)$ inside $Z^u_w$ with $C^u\cap C_w$ under the map $\pi_w^u$,
 cf. Theorem \ref{thm3}.)

The identity \eqref{neq6.2.1} follows from \eqref{neq6.2.2}
 since $\dim \mathcal{Z}= \dim \tilde{\mathcal{Z}}.$ Thus,
\[\codim_\mathcal{Z}\,\bigl(\bar{\Gamma}\times (X^{u,v}_{w})_{{{\bf j}}}\bigr)=
\codim_{\tilde{\Delta}((X_w)_{\bp})}\,\bigl({(X^2_w)}_\bp\bigr)=\ell(w).
\]

 Finally, $\mathcal{Z}$ is CM by Proposition \ref{n5.6} and  Lemmas \ref{lem8}
 and \ref{n6.3}.
This completes the proof of the proposition.
\end{proof}
\begin{lemma} \label{normal} The scheme $\mathcal{Z}$ is
normal,
 irreducible and CM.
\end{lemma}
\begin{proof} By Proposition \ref{propn6.2}, $\mathcal{Z}$ is irreducible and CM.

As in the proof of Lemma \ref{lem9}, the map
$$
\mu_o:G {\times}^{U^{-}}\, X^{u}(S_u')\to \bar{X},\quad [g,x]\mapsto g\cdot x,
$$
being $G$-equivariant, is a locally trivial fibration, where $S_u':=\{v\in W: \ell (v)\leq
\ell  (u)+1\}$. Moreover, its fibers are clearly isomorphic with
$F^u:=\cup_{u \leq v: \ell (v)\leq \ell (u)+1}\, Bv^{-1}U^-/U^-$. Now, since
$X^u$ is normal (cf. [KS, Proposition 3.2]) and any $B^-$-orbit in $X^{u}(S_u')$  is 
of codimension $\leq 1$ in $X^u$, $X^{u}(S_u')$ is smooth and similarly so is $F^u$. (Here
the smoothness of $F^u$ means that there exists a closed normal subgroup $B_1$ of
$B$ of finite codimension such that $B_1$ acts freely and properly on $F^u$
 such that the quotient $B_1\setminus F^u$ is a
smooth scheme of finite type over $\bc$, cf. Lemma \ref{basic}.) Thus, $\mu_o$ is a smooth morphism and hence so is its
restriction to the open subset $B\times X^{u}(S_u') \to \bar{X}$. Let
$\mu_o(w):B\times (X^{u}(S_u')\cap X_w) \to X_w$ be the restriction of the latter to
the inverse image of $X_w$.  The map $\mu_o(w)$ clearly factors through
a smooth morphism $\bar{\mu}_o(w): \bar{B}\times (X^{u}(S_u')\cap X_w)\to X_w $,
where $\bar{B}$ is a finite dimensional quotient group of $B$.
Hence,  $\bar{\mu}_o(w)^{-1}(X_w^o)= \bar{B}\times (X^{u}(S_u')\cap X_w^o)$ is a
smooth variety, where
$X_w^o:=X_w\setminus \Sigma_w $ and $\Sigma_w$ is the singular locus of $X_w$.

Following the same argument as in the proof of Lemma \ref{lem8}, we get that
the restriction of the map $m: \bar{\Gamma}\times (X^{u,v}_w)_{\bf j} \to
(X_w^2)_\bp$ to  the map $\bar{m}:\bar{\Gamma}\times \bigl((X^{u}(S_u')\times
X^v(S_v'))\cap X^{2}_{w} \bigr)_{\bf j} \to
(X_w^2)_\bp$ is a smooth morphism (with open image $Y$) and hence so is its
restricton
$\hat{m}: \bar{m}^{-1}(\tilde{\Delta}((X_w)_{\bp})) \to
\tilde{\Delta}((X_w)_{\bp})$. (Observe that $Y$ does intersect $
\tilde{\Delta}((X_w)_{\bp})$, for otherwise $\bigl(\bar{\Gamma}\cdot
\tilde{\Delta}((X_w)_{\bp})\bigr)\cap Y =\emptyset$, which would imply that
$\bigl((C^2_w)_\bp\bigr) \cap Y = \emptyset$. A contradiction.) Thus,
$\hat{m}^{-1}
(\tilde{\Delta}((X_{w}^o)_\bp))$ is a smooth variety,
which is open in $\mathcal{Z}=m^{-1}(\tilde{\Delta}((X_w)_{\bp}))$.
Let us denote the
complement of $\bar{\Gamma}\times \bigl((X^{u}(S_u')\times
X^v(S_v'))\cap X^{2}_{w}\bigr)_{\bf j}$ in  $\bar{\Gamma}\times (X^{u,v}_w)_{\bf j}$
by $F$ and  denote  $\hat{m}^{-1}
(\tilde{\Delta}((\Sigma_w)_\bp))$ by $F'$. Then, $F'$ is of codimension $\geq 2$
in $\bar{m}^{-1}(\tilde{\Delta}((X_w)_{\bp}))$ and hence in
$\mathcal{Z}$. Clearly, $F$ is of codimension $\geq 2$ in
 $\bar{\Gamma}\times (X^{u,v}_w)_{\bf j}$. Also, if $F$ is nonempty, the restriction of the map $m$ to
$F$ is again flat (by the same proof as that of Lemma \ref{lem8}) with image an
open subset of $(X_w^2)_\bp$ intersecting $\tilde{\Delta}((X_w)_{\bp})$. Thus,
the codimension of $F\cap \mathcal{Z}$ in $\mathcal{Z}$
is $\geq 2$.
This shows that the complement of the smooth locus of $\mathcal{Z}$  in
$\mathcal{Z}$ is of codimension
$\geq 2$. Moreover,  $\mathcal{Z}$ is CM by Proposition \ref{propn6.2}. Thus,
 by Serre's
criterion (cf. [H, Theorem 8.22(a), Chap. II]),  $\mathcal{Z}$ is normal.
\end{proof}

The following lemma and Proposition \ref{ratlsing} are taken from our recent joint work with S. Baldwin [BaK]. 
Proposition \ref{ratlsing}  is used to give a shorter proof (than  our original proof) of Theorem \ref{prop20} (b). 
\begin{lemma}
\label{lemact}
Let $G$ be a group acting on a set $X$ and let $Y\subset X$.  Consider the action map $m:G\times Y\to X$.  For $x\in X$ denote the orbit of $x$ by $O(x)$ and the stabilizer by $\text{Stab}(x)$.  Then, $\text{Stab}(x)$ acts on the fiber $m^{-1}(x)$ and $\text{Stab}(x)\backslash m^{-1}(x)\simeq O(x)\cap Y$.
\end{lemma}

\begin{proof}
It is easy to check that $$m^{-1}(x)=\left\{(g,h^{-1}x):h\in G, \text{ } h^{-1}x\in  Y,\text{ } g\in\text{Stab}(x)\cdot h\right\}.$$
Thus, $\text{Stab}(x)$ acts on $m^{-1}(x)$ by left multiplication on the left component.  Since every element of $O(x)\cap Y$ is of the form $h^{-1}x$ for some $h\in G$,  the second projection $m^{-1}(x)\to O(x)\cap Y$ is surjective.  This map clearly factors through the quotient to give a map $\text{Stab}(x)\backslash m^{-1}(x)\to O(x)\cap Y$.  To show this induced map is injective, note first that each class has a representative of the form $(h,h^{-1}x)$.  Now, if $(h_1,h_1^{-1}x)$ and $(h_2,h_2^{-1}x)$ satisfy $h_1^{-1}x=h_2^{-1}x$ then $h_2 h_1^{-1} x = x$, i.e., $h_2 h_1^{-1}\in \text{Stab}(x)$, i.e., $h_2 \in \text{Stab}(x)\cdot h_1$, i.e., $(h_1,h_1^{-1}x)$ and $(h_2,h_2^{-1}x)$ belong to the same class.
\end{proof}

\begin{proposition}
\label{ratlsing}
The scheme $\ZZ$ has rational singularities.  
\end{proposition}

\begin{proof}  Since $\mu$ is flat and $\tilde{\Delta}((X_w)_\bp)$ has rational singularities (cf. [K, Theorem 8.2.2(c)]), by [El, Theorem 5] it is sufficient to show that the fibers of $\mu$ are disjoint unions of irreducible varieties with rational singularities.  

Let $x\in \tilde{\Delta}((C_{w'})_\bp)$,  where $w'\leq w$.  Then, by Lemmas \ref{lemact} and \ref{lem6} (for $\Gamma_{B\times B}$), we have $\text{Stab}(x)\backslash\mu^{-1}(x)\simeq (X^u\cap C_{w'}\times X^v\cap C_{w'})_{\bf j}$, where $\text{Stab}(x)$ is taken with respect to the action of $\bar{\Gamma}$ on $(X^2_w)_\bp$.  By [Se, Proposition 3, \S2.5],  the quotient map $\bar{\Gamma}\to\text{Stab}(x)\backslash\bar{\Gamma}$ is locally trivial in the \'{e}tale topology.  

Consider the pullback diagram:
\[\xymatrix{\mu^{-1}(x) \ar[d] & \subseteq & \bar{\Gamma}\times (X^{u,v}_w)_{\bf j}\ar[d] \\
\text{Stab}(x)\backslash\mu^{-1}(x) & \subseteq & \left(\text{Stab}(x)\backslash\bar{\Gamma}\right)\times (X^{u,v}_w)_{\bf j}.}
\]
Since the right vertical map is a locally trivial fibration in the \'{e}tale topology, the left vertical map is too.  Now, $\text{Stab}(x)\backslash\mu^{-1}(x)\simeq (X^u\cap C_{w'}\times X^v\cap C_{w'})_{\bf j}$ has rational singularities by [KuS, Theorem 3.1].  Further, $\text{Stab}(x)$ being smooth and $\mu^{-1}(x)\to\text{Stab}(x)\backslash\mu^{-1}(x)$ being locally trivial in the \'{e}tale topology, we get that $\mu^{-1}(x)$ is a disjoint union of irreducible varieties with rational singularities by [KM, Corollary 5.11].

\end{proof}
\begin{proposition}\label{lem11}
The scheme $\partial\mathcal{Z}$ is pure of codimension $1$ in
$\mathcal{Z}$ and it is
CM, where the closed subscheme $\partial \mathcal{Z}$ of
$\mathcal{Z}$ is defined
as
$$
\partial \mathcal{Z}:=\left(\bar{\Gamma}\times
\partial\left((X^{u,v}_{w})_{\bf j}\right)\right)
\mathop{\times}_{{(X_w^2)}_\bp}\tilde{\Delta}((X_w)_{\bp}),
$$
where $\partial\left((X^{u,v}_{w})_{\bf j}\right)$ is defined in Lemma \ref{lem8}.
\end{proposition}

\begin{proof}
By Lemma \ref{lem8}, the map $\bar{\Gamma}\times \partial
\left((X^{u,v}_{w})_{\bf j}\right)\xrightarrow{m'} (X_w^2)_\bp$ is a flat morphism.
Moreover, $\partial \left((X^{u,v}_{w})_{\bf j}\right)$ is pure of
codimension 1 in $(X^{u,v}_{w})_{\bf j}$. Further, $\Iim m'= \Iim m $, if
$\partial \mathbb{P}_{\bf j}\neq\emptyset$. If $\partial
\mathbb{P}_{{\bf j}}=\emptyset$,
\begin{align*}
\Iim m' &\supset \bigl(\left(\left(\cup_{u\to u'\leq \theta\leq
  w}\,C_\theta\right)\times \left(\cup_{v\leq \theta'\leq
  w}\, C_{\theta'} \right)\right)
\cup \left(\left(\cup_{u\leq \theta\leq w}\, C_\theta
\right)\times \left(\cup_{v\to v'\leq \theta'\leq
  w}\, C_{\theta'} \right)\right)\bigr)_{\mathbb{P}}.
\end{align*}

In particular, if non empty, $\Iim m'$ is open in $(X^2_w)_\bp$ (since $m'$ is flat) intersecting
$\tilde{\Delta}((X_w)_{\bp})$. Thus, by [H, Corollary 9.6, Chap. III], each
fiber of $m'$ (if non empty) is pure of dimension
$$
\dim \bar{\Gamma}+\dim (X^{u,v}_{w})_{\bf j}-\dim ((X^2_w)_\bp)
  -1 .
$$
Again applyig [H, Corollary 9.6, Chap. III], we get that $\partial\mathcal{Z}$
is pure of dimension
 $$
\dim \bar{\Gamma}+\dim (X^{u,v}_{w})_{\bf j}-\dim ((X^2_w)_\bp)
  -1+ \dim (\tilde{\Delta}((X_w)_{\bp})).
$$
Hence, by the identity \eqref{neq6.2.1}, $\partial\mathcal{Z}$ is pure of codimension 1 in
$\mathcal{Z}$. Further, both of $\left((\partial X^{u})\cap X_w \right)\times X^v_w$ and 
$X^u_w\times \left((\partial X^v)\cap X_w \right) $ are CM by Proposition \ref{n5.6} and 
Remark \ref{newremark} and so is their intersection. Moreover, their intersection is of pure 
codimension $1$ in both of them. Hence, their union is CM, which follows, e.g., by 
[K, Theorem A.36] and hence so is  $\left((\partial X^{u,v})\cap X^2_w\right)_{\bf j}$. Also, $(X^{u,v}_{w})_{\partial \bp_{\bf j}}$
and the intersection  $$\left((\partial X^{u,v})\cap X^2_w\right)_{\bf j}\cap
(X^{u,v}_{w})_{\partial \bp_{\bf j}} =
\left((\partial X^{u,v})\cap X^2_w\right)_{\partial \bp_{\bf j}}$$
are CM since $\partial \bp_{\bf j}$ is CM. Thus, their union
 $\partial \left((X^{u,v}_w)_{\bf j}\right)$ is CM since the intersection
 $\left((\partial X^{u,v})\cap X^2_w\right)_{\partial \bp_{\bf j}}$ is CM of pure
 codimension $1$ in both of $\left((\partial X^{u,v})\cap X^2_w\right)_{\bf j}$
 and $(X^{u,v}_{w})_{\partial \bp_{\bf j}}$.
 Thus, $\partial \mathcal{Z}$ is CM by
Lemma \ref{n6.3} applied to the morphism $\bar{\Gamma}\times
\partial ((X^{u,v}_{w})_{\bf j}) \to 
(X_w^2)_\bp$.
\end{proof}

As a consequence of Proposition \ref{lem11} and Lemma \ref{n6.4}, we get the following.
\begin{corollary} \label{coro12}  Assume that $c^w_{u,v}({\bf j}) \neq 0$, where
$c^w_{u,v}({\bf j})$ is defined by the identity \eqref{e0.0}.
Then, for general $\gamma\in \bar{\Gamma}$, the fiber $N_\gamma:={\pi}^{-1} (\gamma)\subset
 \mathcal{Z}$ is CM of pure dimension, where the morphism $\pi:\mathcal{Z}\to
\bar{\Gamma}$ is defined in the beginning of this section.

In fact, for any $\gamma\in \bar{\Gamma}$ such that $N_\gamma$ is pure of dimension
\begin{equation}
\dim N_\gamma=\dim \mathcal{Z}-\dim \bar{\Gamma}=|{\bf j}|+\ell(w)-\ell(u)-\ell(v),
\end{equation}
$N_\gamma$ is CM (and this condition is satisfied for general $\gamma$).

Similarly,  if $|{\bf j}|+\ell(w)-\ell(u)-\ell(v) > 0$, for general $\gamma\in
 \bar{\Gamma}$, the fiber $M_\gamma:={\pi_1}^{-1} (\gamma)\subset
 \partial\mathcal{Z}$ is CM of pure codimension $1$ in $N_\gamma$, where $\pi_1$ is the restriction of the map $\pi$ to
 $\partial \mathcal{Z}$. If $|{\bf j}|+\ell(w)-\ell(u)-\ell(v) =0$, for general $\gamma
 \in \bar{\Gamma}$, the fiber $M_\gamma$ is empty.

 In particular, for general $\gamma\in\bar{\Gamma}$, 
$$
\ext^{i}_{\co_{N_{\gamma}}}\left(\co_{N_{\gamma}}(-M_{\gamma}),\omega_{N_{\gamma}}\right)=0,\quad\text{for
  all}\quad i>0,$$
  where $\co_{N_{\gamma}}(-M_{\gamma})$ denotes the ideal sheaf of $M_\gamma$
  in $N_\gamma$ and $\omega_{N_{\gamma}}$ is the dualizing sheaf of $N_\gamma$.
\end{corollary}

\begin{proof}
We first show that $\pi$ is a surjective morphism under the assumption that
$c^{w}_{u,v}({\bf j})\neq 0$. By the definition,
\begin{equation}
\Iim \pi=\left\{\gamma\in
\bar{\Gamma}:\gamma\left((X^{u,v}_{w})_{\bf j}\right)\cap
\tilde{\Delta}((X_w)_{\bp})\neq \emptyset\right\}.\label{eq2}
\end{equation}

Since $\bar{\Gamma}$ is connected by Lemma \ref{connected}, by the expression of
$c^{w}_{u,v}({\bf j})$ as in Lemma \ref{lemma3.3},
$\gamma\left((X^{u,v})_{\bf j}\right)\cap
\tilde{\Delta}((X_w)_{\bp})\neq \emptyset$ for any $\gamma\in \bar{\Gamma}$. But,
$\gamma\left((X^{u,v})_{\bf j}\right)\cap \tilde{\Delta}((X_w)_{\bp})=
\gamma\left((X^{u,v}_{w})_{\bf j}\right)\cap \tilde{\Delta}((X_w)_{\bp})$,
for any $\gamma\in \bar{\Gamma}$. Thus, $\pi$ is surjective.

By Lemmas \ref{n6.4} and \ref{normal}  applied to the morphism
$\pi:\mathcal{Z}\to
\bar{\Gamma}$, we get that if $N_\gamma$ is pure and 
\begin{equation}\label{eq38}
\codim_{\mathcal{Z}} N_{\gamma} =\dim \bar{\Gamma},
\end{equation}
then $N_{\gamma}$ is CM.

Now the condition \eqref{eq38} is satisfied for $\gamma$ in a dense open subset of
$\bar{\Gamma}$ by [S, Theorem 1.25, \S 6.3, Chap. I]. Thus, $N_{\gamma}$ is CM for general
$\gamma$.

Similarly, we prove that  $M_{\gamma}$ is CM for general
$\gamma$:

We first show that  $\pi_{1}:\partial \mathcal{Z}\to\bar{\Gamma}$ is surjective if
$|{\bf j}|+\ell(w)-\ell(u)-\ell(v)>0$. For if $\pi_{1}$ is not surjective, its image
would be a proper closed subset of $\bar{\Gamma}$,
since $\pi_1$ is a projective morphism. Hence,  for
general $\gamma\in\bar{\Gamma}$, $M_\gamma=\emptyset$, i.e., $
N_\gamma\subset \mathcal{Z}\backslash \partial \mathcal{Z}$. But
$\mathcal{Z}\backslash \partial \mathcal{Z}$ is an affine scheme, and
$N_\gamma$ is a projective scheme of positive dimension
(because of the assumption $|{\bf j}|+\ell(w)-\ell(u)-\ell(v)>0$). This is a
contradiction, and hence $\pi_1$ is surjective. Thus, if
$|{\bf j}|+\ell(w)-\ell(u)-\ell(v)>0$,
we get that for general $\gamma\in\bar{\Gamma}$, by [S, Theorem 1.25,
\S 6.3, Chap. I] applied to the irreducible components of $\partial \mathcal{Z}$, $M_\gamma$ is pure and
\begin{equation}\label{neq39}
\codim_{\partial \mathcal{Z}} M_\gamma=\dim \bar{\Gamma}.
\end{equation}

Now, by the same argument as above, we get that for general $\gamma\in\bar{\Gamma}$,
$M_{\gamma}$ is CM. Moreover, since $\partial \mathcal{Z}$ is of pure codimension
$1$ in $\mathcal{Z}$, we get (by equations
\eqref{eq38}-\eqref{neq39}) that $M_{\gamma}$ is of pure codimension $1$ in
$N_{\gamma}$ (for general $\gamma$).

If $|{\bf j}|+\ell(w)-\ell(u)-\ell(v)=0$, then $\dim (\partial \mathcal{Z}) <
\dim \bar{\Gamma}$. So, in this case, $\Iim \pi_1$ is a proper closed subset of
$\bar{\Gamma}$.

Since (for general $\gamma$) $M_{\gamma}$ is of pure codimension $1$ in
$N_{\gamma}$ and both are CM, 
\[\ext^{i}_{\co_{N_{\gamma}}}\left(\co_{N_{\gamma}}(-M_{\gamma}),\omega_{N_{\gamma}}
\right)=0,\quad\text{for
  all}\quad i>0.\]
  To prove this, use the long exact $\ext$ sequence associated to the sheaf
  exact sequence:
  \[0\to \co_{N_{\gamma}}(-M_{\gamma}) \to \co_{N_{\gamma}} \to \co_{M_{\gamma}}\to
  0\]
  and the result that
  \[\ext^{i}_{\co_{N_{\gamma}}}\left(\co_{M_{\gamma}},\omega_{N_{\gamma}}
\right)=0,\quad\text{unless
  }\quad i=1,\]
(cf. [I, Proposition 11.33 and Corollary 11.43]).
  This completes the proof of the corollary.
  \end{proof}

\section{Study of $R^{p}f_{*}\left(\omega_{\tilde{\mathcal{Z}}}(\partial\tilde{\mathcal{Z}})\right)$}\label{sec4}

From now on we assume that $c^w_{u,v}({\bf j}) \neq 0,$ where $c^w_{u,v}({\bf j})$
is defined by the identity \eqref{e0.0}. We follow the notation from the big
diagram in Section
\ref{sec2}.

\begin{lemma}\label{lem16}
The line bundle $\mathcal{L}(\rho)_{|X^{u}}$ has a section with the
zero set precisely equal to $\partial X^{u}$.

In particular,
$$
\mathcal{L}(\rho)_{|X^{u}_{w}}\sim \sum_{i}b_{i}X_{i},\quad b_{i}>0,
$$
where $X_{i}$ are the irreducible components of $(\partial X^{u})\cap
X_{w}$.
\end{lemma}

\begin{proof}
Consider the Borel-Weil isomorphism
$\chi:L(\rho)^{\vee}\xrightarrow{\sim}H^{0}(\bar{X},\mathcal{L}(\rho))$ given by
$\chi(f)(gB) = [g, f(ge_\rho)]$,
where $e_\rho$ is a highest weight vector of the  irreducible highest weight
$G^{\min}$-module
$L(\rho)$ with highest weight $\rho$ and $L(\rho)^{\vee}$ is the restricted dual of $L(\rho)$
(cf. [K, $\S$8.1.21]). Then, it is easy to
see (using [K, Lemma 8.3.3]) that the section $\chi(e^{*}_{u\rho })_{|X^{u}}$ has the zero set
exactly equal to $\partial X^{u}$, where $e_{u\rho }$ is the extremal
weight vector of $L(\rho)$ with weight $u\rho $ and $e^{*}_{u\rho }\in
L(\rho)^{\vee}$
is the linear form which takes value 1 on $e_{u\rho }$ and 0 on any
weight vector of $L(\rho)$ of weight different from $u\rho $. This proves the
lemma.
\end{proof}

A $\Bbb Q$-Cartier $\Bbb Q$-divisor $D$ on an irreducible  projective variety $X$ is called
{\it nef} (resp. {\it big}) if $D$ has
nonnegative intersection with every irreducible curve in $X$ (resp. if dim $H^0(X, \co_X(mD))
>
c m^{\text{dim}\, X},$ for some $c>0$ and $m >> 1$). If $D$ is ample, it is nef and big
(cf. [KM, Proposition 2.61]).

Let  $\pi:X\to Y$ be a proper
morphism between schemes and let $D$ be a $\Bbb Q$-Cartier $\Bbb Q$-divisor on $X$.
Assume that $X$ is irreducible. Then, $D$
is said to be {\it $\pi$-nef} (resp. {\it $\pi$-big}) if $D$ has
nonnegative intersection with every irreducible curve in $X$ contracted by $\pi$ (resp.
if rank $\pi_*\co_X(mD)> c m^n$ for some  $c>0$ and $m >> 1$, where $n$ is the dimension
of a general fiber of $\pi$).

\begin{proposition}\label{prop17}
There exists a nef and big line bundle $\mathcal{M}$ on
$(Z^{u,v}_{w})_{\bf j}$ with a section
with support precisely equal to
$\partial\left((Z^{u,v}_{w})_{\bf j}\right)$, where
$\partial\left((Z^{u,v}_{w})_{\bf j}\right)$ is, by definition, the inverse image of
$\partial\left((X^{u,v}_{w})_{\bf j}\right)$ under the canonical map
$(Z^{u,v}_{w})_{\bf j} \to (X^{u,v}_{w})_{\bf j}$ induced from the $T$-equivariant
 map $\pi_w^{u,v} : Z^{u,v}_w := Z^u_w \times Z^v_w \to X^{u,v}_w := X^u_w \times X^v_w$
and $\partial\left((X^{u,v}_{w})_{\bf j}\right)$ is defined in Lemma \ref{lem8}.

Moreover, such a line bundle $\mathcal{M}$ can be chosen to be the pull-back of
an ample line bundle $\mathcal{M}'$ on $(X^{u,v}_{w})_{\bf j}$.

\end{proposition}

\begin{proof}
Take an ample line bundle $\mathcal{H}$ on $\mathbb{P}_{{\bf j}}$ with a
section with support precisely equal to
$\partial\mathbb{P}_{{\bf j}}$. Also, let
${\mathcal{L}}_{Z^{u,v}_{w}}(\rho \boxtimes \rho)$
be the pull-back of the line bundle $\mathcal{L}(\rho)\boxtimes
\mathcal{L}(\rho)$ on $\bar{X}\times \bar{X}$ via the standard
morphism
$$
Z^{u,v}_{w}\to \bar{X}\times \bar{X}.
$$

Since, $e^{u\rho + v\rho}{\mathcal{L}}_{Z^{u,v}_{w}}(\rho \boxtimes \rho)$ is
a $T$-equivariant line bundle, we get the line bundle
$\tilde{\mathcal{L}}(-\rho\boxtimes
-\rho):=E(T)_{\bf j}\times^{T}\, \Bigl(e^{u\rho + v\rho}
{\mathcal{L}}_{Z^{u,v}_{w}}(\rho\boxtimes \rho)\Bigr)\to
(Z^{u,v}_{w})_{\bf j}$ over the base
space
$(Z^{u,v}_{w})_{\bf j}$. Now,
consider the line bundle (for some large enough $N>0$):
$$
\mathcal{M}:=\tilde{\mathcal{L}}(-\rho\boxtimes -\rho)\otimes
\pi^{*}(\mathcal{H}^{N}),
$$
where
$\pi:E(T)_{\bf j}\times^{T}\,Z^{u,v}_{w}\to
\mathbb{P}_{{\bf j}}$ is the canonical projection. Take the section
$\theta$ of $\tilde{\mathcal{L}}(-\rho\boxtimes -\rho)$ given by
 $[e,z]\mapsto [e, 1_{u\rho +v\rho}\otimes (\bar{\chi}(e^{*}_{u\rho })\boxtimes
    \bar{\chi}(e^{*}_{v\rho}))(z)]$ for $e\in E(T)_{\bf j}$ and
    $z\in Z^{u,v}_{w}$, where $1_{u\rho +v\rho}$ denotes the constant section 
of the trivial line bundle over $Z^{u,v}_w$ with the $T$-action on the fiber given 
by the $H$-weight ${u\rho +v\rho}$ and $\bar{\chi}\boxtimes
    \bar{\chi}$ is the pull-back of the Borel-Weil isomorphism
   ${\chi}\boxtimes
    {\chi}: L(\rho)^{\vee}\otimes L(\rho)^{\vee} \simeq H^0(\bar{X}^2,
    \mathcal{L}(\rho) \boxtimes \mathcal{L}(\rho))$ to $Z^{u,v}_{w}$
    (cf. proof of Lemma \ref{lem16}). Also, take any section $\sigma$ of
    $\mathcal{H}^{N}$ with its zero set precisely equal to $\partial
    \mathbb{P}_{{\bf j}}$ and let $\hat{\sigma}$ be its pull-back to
    $(Z^{u,v}_{w})_{\bf j}$. Then, the zero set of the tensor product of
    these sections $\theta$ and $\hat{\sigma}$ is precisely equal to
    $\partial\left((Z^{u,v}_{w})_{\bf j}\right)$ (cf. proof of Lemma
    \ref{lem16}).

     The line bundle  $\mathcal{M}$ is the pull-back of the line
bundle $\mathcal{M}':=\tilde{\mathcal{L}}'(-\rho \boxtimes -\rho)\otimes
\pi^{*}_{1}(\mathcal{H}^{N})$ on
$E(T)_{\bf j}\displaystyle\mathop{\times}^{T}X^{u,v}_{w}$ via the standard morphism
$$
E(T)_{\bf j}\mathop{\times}^{T}\,Z^{u,v}_{w}\to
E(T)_{\bf j}\mathop{\times}^{T}\,X^{u,v}_{w},
$$
where $\pi_{1}$ is the projection
$E(T)_{\bf j}\displaystyle\mathop{\times}^{T}X^{u,v}_{w}\to
\mathbb{P}_{{\bf j}}$ and $\tilde{\mathcal{L}}'(-\rho\boxtimes -\rho)$ is the
line bundle
$$
E(T)_{\bf j}\displaystyle\mathop{\times}^{T}\left(e^{u\rho+v\rho}(\mathcal{L}(\rho)\boxtimes
\mathcal{L}(\rho))_{|X^{u,v}_{w}}\right).
$$

Then, by [KM, Proposition 1.45 and Theorems 1.37 and 1.42],
$\mathcal{M}'$ is ample on $(X^{u,v}_{w})_{\bf j}$ for large
enough $N$. Since the pull-back of an ample line bundle via a birational
morphism is nef and big (cf. [D, $\S$1.29]), $\mathcal{M}$ is nef and big.
This proves
the proposition.
\end{proof}

We recall the following `relative Kawamata-Viehweg vanishing theorem'
(cf. [D, Exercise 2 on page 217]); replace Debarre's $D$ by $D'$ and take
$D':=\mathcal{L} - D/N$.

\begin{theorem}\label{thm15}
Let $\tilde{\pi}:\tilde{Z}\to \bar{\Gamma}$ be a proper surjective morphism
of irreducible varieties with $\tilde{Z}$ a smooth variety. Let $\mathcal{L}$ be a
line bundle on $\tilde{Z}$ such that $\mathcal{L}^{N}(-D)$ is
$\tilde{\pi}$-nef and $\tilde{\pi}$-big for a simple normal crossing divisor
$$
D=\sum_{i}a_{i}D_{i},\quad \text{where}\quad 0<a_{i}<N,\quad\text{for
  all}\quad i.
$$

Then,
$$
R^{p}\tilde{\pi}_{*}(\mathcal{L}\otimes
\omega_{\tilde{Z}})=0,\quad\text{for all}\quad p>0.
$$
\qed
\end{theorem}

\begin{proposition}\label{prop19}
For the morphism $\tilde{\pi}:\tilde{\mathcal{Z}}\to \bar{\Gamma}$
(cf. the big diagram in Section \ref{sec2}),
$$R^{p}\tilde{\pi}_{*}\bigl(\omega_{\tilde{\mathcal{Z}}}(\partial\tilde{\mathcal{Z}})
\bigr)=0,
\,\,\,
\text{for \,all}\,\,
p>0,$$
where $\partial \tilde{\mathcal{Z}} :=f^{-1} (\partial \mathcal{Z})$ ($\partial \mathcal{Z}$ being dfined in 
Proposition \ref{lem11} taken here with the reduced scheme structure) and $\omega_{\tilde{\mathcal{Z}}}(\partial\tilde{\mathcal{Z}})$ 
denotes the sheaf $\Hom_{\co_{\tilde{\mathcal{Z}}}}\left(\co_{\tilde{\mathcal{Z}}}(-\partial \tilde{\mathcal{Z}}),
\omega_{\tilde{\mathcal{Z}}}\right)$. 

(Observe that $f$ being a desingularization of a normal scheme $\mathcal{Z}$ and $\partial \mathcal{Z}$
being reduced, $\partial \tilde{\mathcal{Z}}$ is a reduced scheme.)
\end{proposition}

\begin{proof}
Fix a nef and big line bundle $\mathcal{M}$ on
$(Z^{u,v}_{w})_{\bf j}$ with
its divisor $\sum\limits_{i=1}^d b_{i}Z_{i}$ (with $b_{i}>0$) supported precisely
in
$\partial\left((Z^{u,v}_{w})_{\bf j}\right)$, which is the pull-back of an
ample line bundle $\mathcal{M}'$ on
$(X^{u,v}_{w})_{\bf j}$
(cf. Proposition \ref{prop17}). Choose an integer $N>b_{i}$, for all
$i$. Consider the line bundle $\mathcal{L}$ on the smooth scheme $\tilde{\mathcal{Z}}$
corresponding to the reduced divisor $\partial\tilde{\mathcal{Z}}$ (observe that $\partial\tilde{\mathcal{Z}}$ is a divisor of 
$\tilde{\mathcal{Z}}$, i.e., a pure scheme of codimension $1$ in $\tilde{\mathcal{Z}}$, since it is the zero set of a line 
bundle on $\tilde{\mathcal{Z}}$) and let $D$
be the divisor on $\tilde{\mathcal{Z}}$:
$$
D=\sum_{i}(N-b_{i})\tilde{Z}_{i},
$$
where
$$
\tilde{Z}_{i}:=(\bar{\Gamma}\times Z_{i})\mathop{\times}_{(Z^2_w)_\bp}
\tilde{\Delta}\left((Z_w)_\bp\right).
$$
Observe that each $\tilde{Z}_{i}$ is a smooth irreducible divisor of $\tilde{\mathcal{Z}}$ and, moreover, 
for any collection $\tilde{Z}_{i_1}, \dots, \tilde{Z}_{i_q}$, $1\leq i_1< \dots < i_q\leq d $,  the intersection $\cap_{p=1}^q \,\tilde{Z}_{i_p}$
 (if nonempty) is smooth of pure codimension $q$ in $\tilde{\mathcal{Z}}$. 
(To prove this, use Theorem \ref{thm2} and follow the proofs of Theorem \ref{thm3}, Lemmas \ref{lem7} and 
\ref{lem5} and Proposition \ref{propn6.2}.) In particular,  $\tilde{Z}_{i}$'s are distinct. It is easy to see that 
\[\partial \tilde{\mathcal{Z}}= \sum \tilde{{Z}}_i\]
and hence it is a simple normal crossing divisor. 
Then,
\begin{align*}
\mathcal{L}^{N}(-D) &= \co_{\tilde{\mathcal{Z}}}\left(\sum_{i}
b_{i}\tilde{Z}_{i}\right)\\[4pt]
&\simeq \tilde{i}^{*}\left(\co_{\bar{\Gamma}\times
  (Z^{u,v}_{w})_{\bf j}}(
\sum b_{i} (\bar{\Gamma}\times Z_{i}))\right).
\end{align*}
Moreover, since $\sum b_{i}Z_{i}$ is a nef divisor on
$(Z^{u,v}_{w})_{\bf j}$ and
$\tilde{i}$ is injective, $\mathcal{L}^{N}(-D)$ is $\tilde{\pi}$-nef (cf. [D, \S1.6]).

Observe further that, by definition, the line bundle $\mathcal{L}^{N}(-D)$ on
$\tilde{\mathcal{Z}}$ is the pull-back of the line bundle $\mathcal{S}:=
i^*(\epsilon\boxtimes \mathcal{M}')$ on $\mathcal{Z}$ via $f$, where $\epsilon$
is the trivial line bundle on $\bar{\Gamma}$. Now, $\mathcal{M}'$ being an
ample line bundle on $(X^{u,v}_{w})_{\bf j}$, $\mathcal{S}$ is $\pi$-big. But,
$f$ being birational, the general fibers of $\tilde{\pi}$ have the same
dimension as the general fibers of $\pi$ (use [S, Theorem 1.25, $\S$6.3, Chap. I]).
Hence, $\mathcal{L}^{N}(-D)$ is $\tilde{\pi}$-big.

The map  $f$ is surjective since it is proper and birational
by Proposition  \ref{propn6.2}. Also, the map $\tilde{\pi}$ is surjective since so is $\pi$ (cf. the proof of Corollary
\ref{coro12}).
Thus, by Theorem \ref{thm15}, the proposition follows.
\end{proof}

\begin{theorem}\label{prop20}
For the morphism $f:\tilde{\mathcal{Z}}\to \mathcal{Z}$,
\begin{itemize}
\item[\rm(a)]
  $R^{p}f_{*}\left(\omega_{\tilde{\mathcal{Z}}}(\partial\tilde{\mathcal{Z}})\right)=0$,
  for all $p>0$, and

\item[\rm(b)] $f_{*}(\omega_{\tilde{\mathcal{Z}}}(\partial
  \tilde{\mathcal{Z}}))= \omega_{\mathcal{Z}}(\partial \mathcal{Z})$.
\end{itemize}
\end{theorem}

\begin{proof}
The map $f$ is surjective as observed above. Following the notation in the proof of
Proposition \ref{prop19}, $\mathcal{L}^{N}(-D)$ is $\tilde{\pi}$-nef
and $\tilde{\pi}$-big. Since the fibers of $f$ are contained in the
fibers of $\tilde{\pi}$, $\mathcal{L}^{N}(-D)$ is $f$-nef. Moreover, since $f$ is
birational, clearly $\mathcal{L}^{N}(-D)$ is $f$-big.
Now, applying Theorem \ref{thm15} to the morphism
$f:\tilde{\mathcal{Z}}\to \mathcal{Z}$, we get the (a) part of the
proposition.

\vskip1ex

(b) First, we claim 

\begin{equation}
\label{pullbackhom}
\OO_{\widetilde{\ZZ}}(\partial\widetilde{\ZZ})\simeq \Hom_{\OO_{\widetilde{\ZZ}}}(f^*\OO_\ZZ(-\partial\ZZ),\OO_{\widetilde{\ZZ}}),
\end{equation}
where $\OO_{\widetilde{\ZZ}} (\partial\widetilde{\ZZ}):= \Hom_{\OO_{\widetilde{\ZZ}}}(\OO_{\widetilde{\ZZ}}(-\partial\widetilde{\ZZ}),\OO_{\widetilde{\ZZ}}).$ 
To see this, first note that by [St, Tag 01HJ, Lemma 25.4.7], since $f^{-1}(\partial\ZZ)=\partial\widetilde{\ZZ}$ is the scheme-theoretic inverse image, the natural  morphism 
$$f^*\left(\OO_\ZZ(-\partial\ZZ)\right)\to\OO_{\widetilde{\ZZ}}(-\partial \widetilde{\ZZ})$$
 is surjective.  As $f$ is a desingularization (cf. Proposition \ref{propn6.2}), the kernel of this morphism is supported on a proper closed subset of $\widetilde{\ZZ}$ 
and hence is a torsion sheaf.  This implies that the dual map $\OO_{\widetilde{\ZZ}}(\partial\widetilde{\ZZ})\to 
\Hom_{\OO_{\widetilde{\ZZ}}}(f^*\left(\OO_\ZZ(-\partial\ZZ)\right),\OO_{\widetilde{\ZZ}})$ is an isomorphism, proving (\ref{pullbackhom}). 

To complete the proof of the (b)-part of the theorem, we compute:
\begin{align*}
f_*(\omega_{\widetilde{\ZZ}} (\partial\widetilde{\ZZ})) &= f_*(\omega_{\widetilde{\ZZ}}\otimes \Hom_{\OO_{\widetilde{\ZZ}}}(f^*\OO_\ZZ(-\partial\ZZ),\OO_{\widetilde{\ZZ}})), \,\,\, \text{by \eqref{pullbackhom}}\\
&=f_* \Hom_{\OO_{\widetilde{\ZZ}}}(f^*\OO_\ZZ(-\partial\ZZ),\omega_{\widetilde{\ZZ}})\\
&=\Hom_{\OO_{\ZZ}}(\OO_\ZZ(-\partial\ZZ),f_*\omega_{\widetilde{\ZZ}}),\,\,\, \text{by adjunction (cf. [H, Chapter II, \S5])}\\
&=\Hom_{\OO_{\ZZ}}(\OO_\ZZ(-\partial\ZZ),\omega_{\ZZ}), \,\,\, \text{by Proposition \ref{ratlsing} and [KM, Theorem 5.10]}\\
&=\omega_\ZZ(\partial\ZZ).
\end{align*}

This completes the proof of the (b)-part.
\end{proof}

As an immediate consequence of Proposition \ref{prop19}, Theorem \ref{prop20}
and the Grothendieck spectral sequence (cf. [J, Proposition 4.1, Part I]), we get
the following:
\begin{corollary}\label{lem14}
Let $\pi:\mathcal{Z}\to \bar{\Gamma}$ be the morphism as in the big diagram in
Section \ref{sec2}. Then,
$$
R^{p}\pi_{*}\bigl(\omega_{\mathcal{Z}}(\partial \mathcal{Z})\bigr)=0,\quad\text{for
  all}\quad p>0.$$
\end{corollary}

\section{Proof of Theorem \ref{thma14} (b)} \label{sec3}

By using Kashiwara's
result: $
\xi^{u}=\co_{X^{u}}(-\partial X^{u})
$ (cf. Theorem \ref{prop:main}) and the vanishing:
\[
\tor_1^{\co_{\bar{Y}_\bp}}\bigl( \gam_* \tilde{\Del}_*\co_{(X_w)_{\bp}}, \co_{\partial (X^{u,v}_{\bf j})}\bigr) = 0,\,\,\,\text{for general} \,\,
\gamma\in \bar{\Gamma},\]
(which can be proved by an argument similar to the proof of Theorem \ref{thma14} (a) using Corollary \ref{newcor5.5}),
Theorem \ref{thma14} (b) is clearly equivalent to the
following vanishing:
\begin{theorem}\label{thm13}
Assume that $c^w_{u,v}({\bf j}) \neq 0.$ For general $\gamma\in \bar{\Gamma}$,
 $$H^{p}\left(X^{u,v}_{\bf j}\cap
  \gamma\tilde{\Delta}((X_w)_{\bp}),\co(-\bar{M}_{\gamma})\right)=0,\,\,\,
  \text{for \,all}\,\,
p\neq |{\bf j}|+\ell(w)-\ell(u)-\ell(v),
$$
where $\bar{M}_{\gamma}:=M_{\gamma^{-1}}$ is the subscheme $\left(\partial
(X^{u,v}_{\bf j})\right)\cap \gamma\tilde{\Delta}((X_w)_{\bp})$
and (as earlier)
$$
\partial \left(X^{u,v}_{\bf j}\right):=\left(\partial X^{u}\times
X^{v}\right)_{\bf j}\cup \left(X^{u}\times \partial
X^{v}\right)_{\bf j}\cup \left(X^{u}\times
X^{v}\right)_{\partial\mathbb{P}_{{\bf j}}},
$$
and $\co(-\bar{M}_{\gamma})$ denotes the ideal sheaf of $\bar{M}_{\gamma}$
in $X^{u,v}_{\bf j}\cap
  \gamma\tilde{\Delta}((X_w)_{\bp})$.
\end{theorem}

\begin{proof}
By Lemma \ref{normal} and Proposition \ref{lem11},
$\mathcal{Z}$ and $\partial \mathcal{Z}$ are  CM
and  $\partial\mathcal{Z}$ is pure of codimension $1$ in
$\mathcal{Z}$. Thus,
 we get  the vanishing (cf. the proof of Corollary \ref{coro12}):
\begin{equation}
\ext^{i}_{\co_{\mathcal{Z}}}\left(\co_{\mathcal{Z}}(-\partial \mathcal{Z}),
\omega_{\mathcal{Z}}\right)=0,\quad \text{for all}\quad i\geq 1.\label{eq303}
\end{equation}

Also, by Corollary \ref{coro12},
 for general $\gamma\in\bar{\Gamma}$,
$$
\ext^{i}_{\co_{\bar{N}_{\gamma}}}\left(\co_{\bar{N}_{\gamma}}(-\bar{M}_{\gamma}),
\omega_{\bar{N}_{\gamma}}\right)=0,\quad\text{for
  all}\quad i>0,$$
  where
 $\bar{N}_\gamma:= N_{\gamma^{-1}}$  is the subscheme $
(X^{u,v}_{\bf j})\cap \gamma\tilde{\Delta}((X_w)_{\bp})$.

Hence, by the Serre duality (cf. [H, Theorem 7.6, Chap. III]) applied to $\bar{N}_{\gamma}$ and the local to global
Ext spectral sequence (cf. [Go, Th\'eor\`eme 7.3.3, Chap. II]), the theorem
 is equivalent to the vanishing (for general $\gamma\in \bar{\Gamma}$) :
\begin{equation}
H^{p}\left(\bar{N}_{\gamma},\Hom_{\co_{\bar{N}_{\gamma}}}
\left(\co_{\bar{N}_{\gamma}}(-\bar{M}_{\gamma}),\omega_{\bar{N}_{\gamma}}\right)\right)
=0,\,\,\,\text{for \, all}\,\, p>0,\label{eq5}
\end{equation}
since (for general $\gamma\in \bar{\Gamma}$) $\bar{N}_{\gamma}$ is CM and $\dim \bar{N}_{\gamma}=
|{\bf j}|+\ell(w)-\ell(u)-\ell(v)$ (cf. Corollary \ref{coro12}).

For general $\gamma\in \bar{\Gamma}$,
\begin{align}
\omega_{\mathcal{Z}}(\partial \mathcal{Z})_{|\pi^{-1}(\gamma^{-1})} &\simeq
\omega_{\pi^{-1}(\gamma^{-1})}(\partial\mathcal{Z}\cap
\pi^{-1}(\gamma^{-1}))\notag\\[3pt]
&= \omega_{\bar{N}_{\gamma}}(\bar{M}_{\gamma}), \label{eq6}
\end{align}
where $\omega_{\bar{N}_{\gamma}}(\bar{M}_{\gamma}) := 
\Hom_{\co_{\bar{N}_{\gamma}}}
\left(\co_{\bar{N}_{\gamma}}(-\bar{M}_{\gamma}),\omega_{\bar{N}_{\gamma}}\right).$ 
To prove the above, observe first that by [S, Theorem 1.25, $\S$6.3, Chap. I] and [H,
Exercise 10.9, Chap. III] applied to $\pi$, there exists an open nonempty subset $\bar{\Gamma}_o
\subset \bar{\Gamma}$ such that $\pi: \pi^{-1}(\bar{\Gamma}_o) \to \bar{\Gamma}_o$
is a flat morphism. (By the proof of Corollary \ref{coro12}, $\pi$ is surjective.) Now, since  $\bar{\Gamma}_o$ is smooth and
$\mathcal{Z}$ and $\partial\mathcal{Z}$ are CM,  and the assertion is
 local in $\bar{\Gamma}$, it suffices
to observe (cf. [I, Corollary 11.35]) that for a  nonzero function $\theta$ on  $\bar{\Gamma}_o$, the sheaf of $\co_{\mathcal{Z}_\theta}$-module
$$\mathcal{S}/\theta \cdot\mathcal{S}\simeq \Hom_{\co_{\mathcal{Z}_\theta}}
\bigl(\co_{\mathcal{Z}}(-\partial \mathcal{Z})/
\theta \cdot\co_{\mathcal{Z}}(-\partial \mathcal{Z}),
\omega_{\mathcal{Z}_\theta}\bigr),$$
 where $\mathcal{Z}_\theta$ denotes the zero scheme of $\theta$ in
 $\mathcal{Z}$ and the sheaf $\mathcal{S}:=\Hom_{\co_{\mathcal{Z}}}
 \left(\co_{\mathcal{Z}}(-\partial \mathcal{Z}),
\omega_{\mathcal{Z}}\right).$  Choosing $\theta$ to be in  a
local coordinate system, we can continue and get \eqref{eq6}.

Now, the vanishing of
$R^{p}\pi_{*}\bigl(\omega_{\mathcal{Z}}(\partial\mathcal{Z})\bigr)$ for
$p>0$ (cf. Corollary \ref{lem14}) implies the following vanishing, for general $\gamma\in \bar{\Gamma}$,
\begin{equation}\label{eq7}
H^{p}\left(\bar{N}_\gamma,\omega_{\bar{N}_{\gamma}}(\bar{M}_{\gamma})\right)=0,
\quad\text{for
  all}\quad p>0.
\end{equation}
To prove this, since $\mathcal{Z}$ and $\partial\mathcal{Z}$ are CM,
$\bar{\Gamma}_o$ is smooth
and $\pi: \pi^{-1}(\bar{\Gamma}_o) \to \bar{\Gamma}_o$ is flat, observe that
$\omega_{\mathcal{Z}}(\partial\mathcal{Z})$ is flat over the base $\bar{\Gamma}_o$:

To show this, let $A=\OO_{\bar{\Gamma}_o}$, $B=\OO_{\pi^{-1}(\bar{\Gamma}_o)}$, and $M=\omega_\ZZ(\partial\ZZ)|_{\pi^{-1}(\bar{\Gamma}_o)}$.  By 
taking stalks, we immediately reduce to showing that for an embedding of local rings $A\subset B$ such that $A$ is regular and $B$ is flat over $A$, we 
have that $M$ is flat over $A$.  Now, to prove this, let $\{x_1,\dots,x_d\}$ be a minimal set of generators of the maximal ideal of $A$.  Let 
$K_\bullet=K_\bullet(x_1,\dots,x_d)$ be the Koszul complex of the $x_i$'s over $A$.  Then, recall that a $B$-module $N$ is flat over $A$ 
iff $K_\bullet\otimes_A N$ is exact except at the extreme right, i.e., $H^i(K_\bullet\otimes_A N)=0$ for $i<d$ (cf. [E, Theorem 6.8].  
Thus, by hypothesis, $K_\bullet\otimes_A B$ is exact except at the extreme right and hence  the $x_i$'s form a $B$-regular sequence 
by  [E, Theorem 17.6].  Now, since $\OO_\ZZ$ and $\OO_{\partial\ZZ}$ are CM and $\partial \ZZ$ is pure of codimension $1$ in $\ZZ$, 
we have that $\OO_\ZZ(-\partial\ZZ)$ is a CM $\OO_\ZZ$-module.  Thus, by [I, Proposition 11.33], we have that $M$ is a CM $B$-module 
of dimension equal to $\dim B$.  Therefore, by \cite[Exercise 11.36]{I}, the $x_i$'s form a regular sequence on the $B$-module $M$.  
Hence, $(K_\bullet\otimes_A B)\otimes_B M\simeq K_\bullet\otimes_A M$ is exact except at the extreme right by [E, Corollary 17.5].  
This proves that $M$ is flat over $A$, as desired.

Hence, \eqref{eq7} folows from the semicontinuity theorem
(cf. [H, Theorem 12.8 and Corollary 12.9,
Chap. III] or [Ke, Theorem 13.1]).

Thus,  \eqref{eq5} (which is nothing but \eqref{eq7}) is established. Hence, the theorem follows and thus Theorem \ref{thma14} (b)
is established.
\end{proof}

\newpage

\begin{center}
{\bf Appendix by Masaki Kashiwara}
\end{center}

\section{ Determination of the dualizing sheaf of $X^v$} \label{appendix}

Let  $v \in W$.
Set $\mathfrak{C}^v\seteq\bigcup_{y\in W, \,\ell(y)\le \ell(v)+1}\,C^y$,
where $C^y:= B^-yB/B \subset \bar{X}$. (By definition, $\mathfrak{C}^v$ only depends upon $\ell(v)$.)
Then, $\mathfrak{C}^v$ is an open subset of $\bar{X}$.
Moreover, $X^v\cap \mathfrak{C}^v$ is a smooth scheme, since $X^v$ is normal
([KS, Proposition~3.2]) and any $B^-$-orbit in $X^v\cap \mathfrak{C}^v$ is of codimension $\leq 1$.
Recall from Section 3, the definition of
$$\xi^v:=
e^{-\rho} \cl (\rho )\om_{X^v}
= e^{-\rho} \cl (-\rho ) \ext^{\ell (v)}_{\co_{\bar{X}}}
(\co_{X^v}, \co_{\bar{X}} ).$$
Since $\co_{X^v}$ is a CM ring (cf. [KS, Proposition 3.4]), we have
 that $\xi^v$ is a CM
$\co_{X^v}$-module. Also, since $X^v\cap \mathfrak{C}^v$ is a smooth scheme,
$\xi^v\vert_{\mathfrak{C}^v}$ is an invertible $(\O_{X^v}\vert_{\mathfrak{C}^v})$-module.

For any $y\in W$, let $i_y:  \{\pt\}\to {\bar{X}}$ be
the morphism given by $\pt\mapsto yx_o$.
Then, we have (as an $H$-module):
\beqn \label{e9.1}
i_y^*\mathcal{L}(\lambda)\simeq\C_{-y\lambda},\quad
\text{for any character $\lambda$ of $H$.}
\eeqn

Let $\pi_i: {\bar{X}}\to {\bar{X}}_i$ be the projection as in the proof of
Proposition \ref{prop2.6}.

\begin{lemma} \label{lem1}
On some $B^-$-stable neighborhood of $C^v$,
we have a  $B^-$-equivariant  isomorphism
$\xi^v\simeq\O_{X^v}$.
\end{lemma}
\begin{proof}
Since $\xi^v\vert_{\mathfrak{C}^v}$ is an invertible $B^-$-equivariant $\O_{X^v}\vert_{\mathfrak{C}^v}$-module,
it is enough to show that
$i_v^*\xi^v\simeq \C$ as an $H$-module.
This follows from $i_v^*\bl(\ext_{\co_{\bar{X}}}^{\ell(v)}(\O_{X^v},\O_{\bar{X}}))
\simeq \det(T_{vx_o}{\bar{X}}/T_{vx_o}X^v)\simeq \C_{\rho-v\rho}$
and $i_v^*\mathcal{L}(-\rho)\simeq\C_{v\rho}$ by \eqref{e9.1}.
\end{proof}

Set $A_v\seteq\{y\in W : y>v\,\,\text{and}\, \ell(y)=\ell(v)+1\}$.
The above lemma implies that, as  $B^-$-equivariant $\O_{\bar{X}}$-modules,
\beqn \label{e9.2}
\xi^v\vert_{\mathfrak{C}^v}\simeq\O_{X^v}(\ssum_{y\in A_v}m_yX^y)\vert_{\mathfrak{C}^v},
\eeqn
for some $m_y\in\Z$.
Recall that  $\partial X^v=\bigcup_{y\in A_v}X^y$.

\begin{lemma} \label{9.2} We have
$\xi^v\vert_{\mathfrak{C}^v}\simeq \O_{X^v}(-\partial X^v)\vert_{\mathfrak{C}^v}$, where
 $\O_{X^v}(-\partial X^v)\subset \O_{X^v}$
is the ideal sheaf of the reduced subscheme $\partial X^v$ of $X^v$.
\end{lemma}
\begin{proof}
The proof is similar to the one of Lemma~\ref{lem1}.
For $y\in A_v$, $y$ is a smooth point of $X^v$ (since $X^v\cap \mathfrak{C}^v$ is smooth). Hence, we have
\eqn
i_y^*\bl(\ext_{\co_{\bar{X}}}^{\ell(v)}(\O_{X^v},\O_{\bar{X}}))
&\simeq &\det(T_{yx_o}{\bar{X}}/T_{yx_o}X^v)\\
&\simeq&
  \det(T_{yx_o}{\bar{X}}/T_{yx_o}X^y)\otimes \det(T_{yx_o}X^v/T_{yx_o}X^y)^{\otimes (-1)}\\
&\simeq& \C_{\rho-y\rho}\otimes \det(T_{yx_o}X^v/T_{yx_o}X^y)^{\otimes (-1)}.
\eneqn
Hence, we obtain
$i_y^*\xi^v\simeq  (T_{yx_o}X^v/T_{yx_o}X^y)^{\otimes (-1)}$ as an $H$-module by \eqref{e9.1}.
On the other hand, we have
$$i_y^*\bl(\O_{X^v}(\ssum_{z\in A_v}m_zX^z))
\simeq (T_{yx_o}X^v/T_{yx_o}X^y)^{\otimes m_y},\quad\text{as an $H$-module.}$$
Hence, by \eqref{e9.2}, we have $m_y=-1$.
Note that $T_{yx_o}X^v/T_{yx_o}X^y$ is not a trivial $H$-module
by the following lemma.
\end{proof}

\begin{lemma} Let $v$, $y\in W$ satisfy
$v<y$ and $\ell(y)=\ell(v)+1$. Then,
$$T_{yx_o}X^v/T_{yx_o}X^y\simeq \C_{\beta}$$
as $H$-modules,
where $\beta$ is the positive real root such that
$yv^{-1}=s_\beta$.
\end{lemma}
\begin{proof}
We prove this by induction on  $\ell(y)$.
Take a simple reflection $s_i$ such that $ys_i<y$.

\smallskip
\noindent
(i) If $vs_i>v$, then we have $y=vs_i$. Thus,
 $T_{yx_o}X^v/T_{yx_o}X^y\simeq T_{yx_o}\pi_i^{-1}\pi_i(yx_o)$
and hence we have $T_{yx_o}X^v/T_{yx_o}X^y\simeq \C_{-y\alpha_i}
\simeq\C_{v\alpha_i}= \C_{\beta}$.

\smallskip
\noindent
(ii) If $vs_i<v$, then $\pi_i:  X^v\to {\bar{X}}_i$ is a local embedding at
$yx_o$ since $C^v\cup C^y$ is open in $X^v$, $\pi_i\vert_{C^v\cup C^y}$ is an
injective map onto an open subset of $\pi_i(X^v)$, and
$\pi_i(X^v)=\pi_i(X^{vs_i})$ is normal (since $X^{vs_i} \to \pi_i(X^{vs_i})$ is
a $\mathbb P^1$-fibration and $X^{vs_i}$ is normal by [KS, Proposition 3.2]). Moreover,
$\pi_i(X^v)$ is smooth at $\pi_i(yx_o)$ since the $B^-$-orbit of $\pi_i(yx_o)$ is of
codimension $1$ in $\pi_i(X^v)$.
Hence, we have $T_{yx_o}X^v/T_{yx_o}X^y\simeq
T_{\pi_i(yx_o)}\bigl(\pi_i(X^v)\bigr)/T_{\pi_i(yx_o)}\bigl(\pi_i(X^y)\bigr)
\simeq T_{ys_ix_o}X^{vs_i}/T_{ys_ix_o}X^{ys_i}$.
By the induction hypothesis, it is isomorphic to
$\C_{\beta}$.
\end{proof}

Let $j: \mathfrak{C}^v\hookrightarrow {\bar{X}}$ be the open embedding.

\begin{theorem} \label{prop:main} For any $v\in W$,  we have a $B^-$-equivariant
isomorphism:
$$\xi^v\simeq
\O_{X^v}(-\partial X^v).$$

Hence, the dualizing sheaf $\omega_{X^v}$ of $X^v$ is $T$-equivariantly
isomorphic with
$$ \C_{\rho}\otimes\mathcal{L}(-\rho)\otimes\O_{X^v}(-\partial X^v).$$

\end{theorem}
\begin{proof}
We have a commutative diagram with exact rows:

\[ \xymatrix{0\ar[r]&
\O_{X^v}(-\partial X^v)\ar[d]\ar[r]
 & \O_{X^v}\ar[d]^{\bwr} \ar[r]
& \db{\O_{\partial X^v}}\ar@{>->}[d]  \ar[r]& 0\\
0\ar[r]&  j_*j^{-1}\O_{X^v}(-\partial X^v)\ar[r]&  j_*j^{-1}\O_{X^v}\ar[r]&
j_*j^{-1}\O_{\partial X^v}, }
\]
where the middle vertical arrow is an isomorphism because $X^v$ is normal and
$X^v\setminus \mathfrak{C}^v$ is of codimension $\geq 2$ in $X^v$,
and the right vertical arrow is a monomorphism because the closure of
$\partial X^v\cap \mathfrak{C}^v$ coincides with $\partial X^v$. Hence, we have
$j_*j^{-1}\O_{X^v}(-\partial X^v)\simeq \O_{X^v}(-\partial X^v)$.
On the other hand, since $\xi^v$ is a CM
$\O_{X^v}$-module,
we have
$$\xi^v\simeq j_*j^{-1}\xi^v\simeq j_*j^{-1}\O_{X^v}(-\partial X^v)
\simeq \O_{X^v}(-\partial X^v),$$
where the second isomorphism is due to Lemma \ref{9.2}.
\end{proof}

\begin{corollary}\label{9.5}
$\O_{X^v}(-\partial X^v)$ is a CM $\O_{X^v}$-module and
 $\O_{\partial {X^v}}$ is a CM ring.
\end{corollary}
\begin{proof}
Since $\xi^v$ is a CM $\O_{X^v}$-module, so is $\O_{X^v}(-\partial X^v)$ by the above theorem.

Applying the functor
$\home_{\O_{\bar{X}}}(\scbul,\O_{\bar{X}})$ to
the exact sequence:
 $$0\to \O_{X^v}(-\partial X^v)\to \O_{X^v}\to \O_{\partial{X^v}}\to0,$$
we obtain
$\ext_{\co_{\bar{X}}}^k( \O_{\partial{X^v}},\O_{\bar{X}})=0$,  for
$k\not=\ell(v),\ell(v)+1$. We  also have an exact sequence:
$$0\to\ext_{\co_{\bar{X}}}^{\ell(v)}( \O_{\partial{X^v}},\O_{\bar{X}})
\to \ext_{\co_{\bar{X}}}^{\ell(v)}( \O_{X^v},\O_{\bar{X}}).$$
Since $\partial X^v$ has codimension $\ell(v)+1$, we have
$\ext_{\co_{\bar{X}}}^{\ell(v)}( \O_{\partial X^v},\O_{\bar{X}})=0$.
Hence, $\O_{\partial{X^v}}$ is a  CM ring.
\end{proof}

\vskip3ex
\noindent
S.K.: Department of Mathematics, University of North Carolina,
Chapel Hill, NC 27599-3250, USA (email: shrawan$@$email.unc.edu)

\end{document}